\documentclass[final]{siamart0216}

\usepackage{amsfonts,amsmath,amssymb}
\usepackage{mathrsfs,mathtools,stmaryrd,wasysym}
\usepackage{enumerate}
\usepackage{esint}
\usepackage{graphicx,psfrag}
\usepackage{hyperref}
\DeclareMathAlphabet{\mathpzc}{OT1}{pzc}{m}{it}
\usepackage{lineno}
\usepackage{stmaryrd}
\usepackage{enumerate}
\usepackage{pgf,tikz,pgfplots}

\usepackage{pstricks-add}

\usepackage{enumitem}

\usepackage[notcite,notref]{showkeys} 

\DeclareFontEncoding{FMS}{}{}
\DeclareFontSubstitution{FMS}{futm}{m}{n}
\DeclareFontEncoding{FMX}{}{}
\DeclareFontSubstitution{FMX}{futm}{m}{n}
\DeclareSymbolFont{fouriersymbols}{FMS}{futm}{m}{n}
\DeclareSymbolFont{fourierlargesymbols}{FMX}{futm}{m}{n}
\DeclareMathDelimiter{\VERT}{\mathord}{fouriersymbols}{152}{fourierlargesymbols}{147}

\numberwithin{equation}{section}

\usepackage{mathrsfs}
\def\ds{\displaystyle}

\def\O{\Omega}

\def\b{\beta}

\def\CT{{\mathcal T}}

\def\H{\mathrm H}
\def\Q{\mathbb Q}
\def\I{\mathrm I}
\def\V{\mathrm V}

\def\L{\mathrm L}

\def\W{\mathrm W}

\def\div{\mathop{\mathrm{div}}\nolimits}

\def\bu{\boldsymbol{u}}

\def\bv{\boldsymbol{v}}

\def\bz{\boldsymbol{z}}

\def\bf{\boldsymbol{\textit{f}}}

\def\b0{\boldsymbol{0}}

\def\bPi{\boldsymbol{\Pi}}
\def\bXi{\boldsymbol{\Xi}}

\def\bw{\boldsymbol{w}}

\renewcommand\H{\mathrm{H}}
\renewcommand\L{\mathrm{L}}
\renewcommand\Q{\mathrm{Q}}
\def\bVh{\mathbf{V}_h}
\def\bWh{\mathbf{W}_h}

\def\bUh{\mathrm{U}_h}
\def\bVV{\mathrm{V}}
\def\bVVh{\mathrm{V}_h}
\def\bZ{\mathbf{Z}}
\def\bZh{\mathbf{Z}_h}

\def\mtW{\mathbf{W}}
\def\bV{\mathbf{V}}

\def\div{\mathop{\mathrm{div}}\nolimits}

\newtheorem{remark}[theorem]{Remark}
\renewcommand{\theequation}{\thesection.\arabic{equation}}

\usepackage{pgf,tikz}
\usetikzlibrary{arrows}

\newcommand{\TheTitle}{A virtual element method for a convective Brinkman-Forchheimer problem coupled with a heat equation}
\newcommand{\ShortTitle}{A VEM for Brinkman-Forchheimer coupled with a heat equation}
\newcommand{\TheAuthors}{D. Amigo, F. Lepe, E. Ot\'arola and G. Rivera}

\headers{\ShortTitle}{\TheAuthors}

\title{{\TheTitle}\thanks{DA is partially supported by ANID-Chile through grant ACT210087. EO is partially supported by ANID-Chile through FONDECYT grant 1220156. GR is partially supported by ANID-Chile through FONDECYT grant 1231619 and Universidad de Los Lagos Regular R02/21.}}
\author{Danilo Amigo\thanks{Departamento de Matem\'atica, Universidad del B\'io-B\'io, Concepci\'on, Chile.
(\email{danilo.amigo2101@alumnos.ubiobio.cl})}
\and
Felipe Lepe\thanks{Departamento de Matem\'atica, Universidad del B\'io-B\'io, Concepci\'on, Chile.
(\email{flepe@ubiobio.cl})}
\and
Enrique Ot\'arola\thanks{Departamento de Matem\'atica, Universidad T\'ecnica Federico Santa Mar\'ia, Valpara\'iso, Chile.
(\email{enrique.otarola@usm.cl}, \url{http://eotarola.mat.utfsm.cl/}).}
\and
Gonzalo Rivera\thanks{Departamento de Ciencias Exactas, Universidad de Los Lagos, Osorno, Chile.
(\email{gonzalo.rivera@ulagos.cl}).}
}

\ifpdf
\hypersetup{
  pdftitle={\TheTitle},
  pdfauthor={\TheAuthors}
}
\fi

\date{Draft version of \today.}

\begin{document}

\maketitle

\begin{abstract}
We develop a virtual element method to solve a convective Brinkman-Forchheimer problem coupled with a heat equation. This coupled model may allow for thermal diffusion and viscosity as a function of temperature. Under standard discretization assumptions, we prove the well posedness of the proposed numerical scheme. We also derive optimal error estimates under appropriate regularity assumptions for the solution. We conclude with a series of numerical tests performed with different mesh families that complement our theoretical findings.
\end{abstract}



\section{Introduction}\label{sec:intro}

Let $\O\subset\mathbb{R}^2$ be an open and bounded domain with Lipschitz boundary $\partial \Omega$. In this work, we are interested in designing and analyzing a divergence-free virtual element method (VEM) for the temperature distribution of a fluid modeled by a convection-diffusion equation coupled with a convective Brinkman--Forchheimer problem. This model can be described by the following nonlinear system of partial differential equations (PDEs):
\begin{equation}\label{def:BFH}
\left\{\begin{array}{rccc}
-\div(\nu(T)\nabla\bu)+(\bu\cdot\nabla)\bu
+
\bu
+
|\bu|^{r-2}\bu
+\nabla \textsf{p}  &=&  \boldsymbol{f}  \text{ in } \Omega,
\\
\div(\bu) &=& 0 \,\, \text{in} \,\, \Omega, \\
-\div(\kappa(T)\nabla T)+\bu\cdot\nabla T &=& g 
\,\,\text{in}\,\, \Omega,
\end{array}\right.
\end{equation}
together with the Dirichlet boundary conditions $\bu = \boldsymbol{0}$ and $T=0$ on $\partial\Omega$. The unknowns of the system are the velocity field $\bu$, the
pressure $\textsf{p}$, and the temperature $T$ of the fluid. The data are the external force $\boldsymbol{f}$, the external heat source $g$, the viscosity coefficient $\nu$, and the thermal diffusion coefficient $\kappa$. We note that $\nu(\cdot)$ and $\kappa(\cdot)$ are coefficients that can depend nonlinearly on the temperature. Finally, the parameter $r$ is chosen such that $r \in [3, 4]$. For a discussion of this parameter $r$, the Forchheimer term in \eqref{def:BFH}, and its physical implications, we refer to the reader to \cite{campana2024finite} and the references therein.

Beginning with the pioneering works \cite{MR2997471,MR3200242}, great efforts have been made to develop and analyze VEMs. These methods are a relatively new family of solution techniques that allow general polytopal meshes, arbitrary polynomial degrees, and yet conforming $\H^1$-approximations. The peculiarity of VEMs is that the discrete spaces consist of functions that are not known pointwise, but about which a limited amount of information is available. This limited information is sufficient to construct stiffness matrices and right-hand sides. On the other hand, VEMs allow great flexibility with regard to the shape of the elements; for example, convex and non-convex elements are permitted. This is an advantage over the classical finite element methods (FEMs), which make it possible to deal with domains that are difficult to discretize with triangles or quadrilaterals. For these reasons, the VEM is generally considered as a generalization of the FEM. The analysis of VEMs has been successfully developed for a large number of problems. We refer the interested reader to \cite{MR4510898, MR4586821} for a recent review and discussion. As for the development of VEMs for various linear and nonlinear flow problems, we refer the reader to the non-exhaustive list \cite{MR4667549,MR3164557,MR4754972,AVV,MR3626409,MR3796371,MR3895873,NMTMA-17-210}. An important feature of some methods for the treatment of incompressible fluids is that the divergence-free condition is preserved at the discrete level.

The aim of the present work is to develop and analyze a VEM to approximate the velocity, pressure, and temperature variables that solve the system \eqref{def:BFH}. Our analysis is inspired by the recent work \cite{AVV}, in which the authors propose a divergence-free VEM to solve the coupling between the Navier--Stokes equations and a suitable heat equation. Here, the viscosity coefficient depends on the temperature variable. In our work, we complement and extend the results in \cite{AVV} in the following two directions: First, the model we consider allows the dependence of the thermal diffusion coefficient $\kappa(\cdot)$ on the temperature variable. Second, we consider the so-called convective Brinkman--Forchheimer equations. In contrast to \cite{AVV}, our model also considers the term $\bu + |\bu|^{r-2}\bu$, with $r \in [3,4]$. The consideration of this term was suggested by Forchheimer \cite{forchheimer1901wasserbewegung}, who realized that Darcy’s law is not adequate for moderate Reynolds numbers. Indeed, Forchheimer found that the relationship between the Darcy velocity and the pressure gradient was nonlinear and that this nonlinearity appeared to be quadratic $(r=3)$ for a variety of experimental data \cite{forchheimer1901wasserbewegung}. This leads to a modification of the Darcy equations, usually referred to as the Darcy--Forchheimer equations \cite{MR4659334,MR4092292,MR4049400,MR2425154,MR2948707,MR3022234}. In \cite{forchheimer1901wasserbewegung}, Forchheimer also noted that some data sets could not be described by the quadratic correction so he also postulated that the correction of Darcy’s law could allow a polynomial expression for $\boldsymbol{u}$, e.g., $\boldsymbol{u} + |\boldsymbol{u}|\boldsymbol{u} + |\boldsymbol{u}|^2\boldsymbol{u}$ and $\boldsymbol{u} + |\boldsymbol{u}|^{r-2}\boldsymbol{u}$; see \cite[page 59]{muskat1937flow}, \cite[Section 2.3]{firdaouss1997nonlinear}, and \cite[page 12]{straughan2008stability}. In practice, $r$ takes the value $3$ and $4$ in several applications  \cite{balhoff2010polynomial,firdaouss1997nonlinear,MR2027361,MR1090724,MR1856630,skjetne1999new,MR3636305} and also fractional values such as $r = 7/2$ \cite[page 133]{skjetne1999new}. This is the reason why we consider $r \in [3,4]$. In view of these considerations, \eqref{def:BFH} can be seen as an extension of the model in \cite{AVV}.

\subsection{Contributions}

To the best of our knowledge, this is the first paper that analyzes a VEM for the nonlinear coupled problem \eqref{def:BFH}. Since several sources of nonlinearity are involved, analyzing a solution technique is far from trivial. In the following, we list what we consider to be the main contributions of our work:

$\bullet$ \emph{A VEM:} Inspired by the scheme proposed in \cite{AVV} and the discrete spaces proposed and analyzed in \cite{MR3626409,MR3796371}, we propose the VEM \eqref{def:weak_BFH_discrete}.

$\bullet$ \emph{Existence and uniqueness of discrete solutions:} We derive an existence result for the discrete problem \eqref{def:weak_BFH_discrete} without restriction on the
problem data by using a fixed point strategy; see Theorem \ref{thm:existence}. Moreover, we obtain a global uniqueness result when the problem data is suitably restricted; see Theorem \ref{thm:uniqueness}.

$\bullet$ \emph{Optimal error estimates:} Assuming that the continuous and discrete problems admit unique solutions and under suitable regularity assumptions for the continuous solution, we derive optimal error estimates for the proposed VEM; see Theorem \ref{thm:error_estimates}. The analysis borrows ideas and components from \cite{AVV}, \cite{MR3796371}, and \cite{NMTMA-17-210}.


\subsection{Outline}
The paper is structured as follows: We begin with section \ref{sec:notation_and_preliminaries}, where we introduce notations and basic assumptions that we will use in our work. In sections \ref{sec:BF}  and \ref{sec:heat}, we summarize some results related to the convective Brinkmann--Forchheimer problem and a suitable heat equation, respectively.
Section \ref{sec:BFH} is devoted to review existence, uniqueness, and stability results for the coupled problem \eqref{def:BFH}. The core of our work begins in section \ref{sec:vem}, where we first introduce the standard techniques for analyzing VEMs and propose our discrete scheme \eqref{def:weak_BFH_discrete}.
In section \ref{sec:ex_uniq_sol_vem}, we prove the existence of discrete solutions without restriction on the problem data, as well as a global uniqueness result when the problem data is appropriately restricted. Section \ref{sec:error} is devoted to the development of a rigorous analysis of error estimates. Finally, in section  \ref{sec:numericos}, we report a series of numerical tests in which we evaluate the performance of the proposed method for different configurations and polygonal meshes.
%
%
\section{Notation and preliminary remarks}
\label{sec:notation_and_preliminaries}
Let us establish the notation and the framework within which we will work.

 \subsection{Notation}
 \label{sec:notation}
 In this paper, $\Omega$ is an open and bounded polygonal domain of $\mathbb{R}^{2}$ with Lipschitz boundary $\partial \Omega$. If $\mathscr{X}$ and $\mathscr{Y}$ are normed vector spaces, we write $\mathscr{X} \hookrightarrow \mathscr{Y}$ to denote that $\mathscr{X}$ is continuously embedded in $\mathscr{Y}$. We denote by $\mathscr{X}'$ and $\|\cdot\|_{\mathscr{X}}$ the dual and the norm of $\mathscr{X}$, respectively.
For $ p \in (1,\infty)$, we denote by $q \in (1,\infty)$ its H\"older \emph{conjugate}, which is such that $1/p + 1/q = 1$. The relation $\texttt{a} \lesssim \texttt{b}$ means that $\texttt{a} \leq C \texttt{b}$, with a positive constant $C$ that is independent of $\texttt{a}$, $\texttt{b}$, and the discretization parameters. The value of $C$ can change at each occurrence. 

We use the standard notation for Lebesgue and Sobolev spaces. The spaces of vector-valued functions and the vector-valued functions themselves are denoted by bold letters. In particular, we use the following notation: $\bVV:=\H_0^1(\O)$, $\bV:=[\H_0^1(\O)]^2$, and $\Q:=\L_0^2(\O)$. As usual, the dual of $\H_0^1(\O)$ is denoted by $\H^{-1}(\O)$. We also introduce the space
$
\bZ := \{ \bv \in \bV : \, \div \bv = 0\}.
$
We conclude this section with the classical and well--known Poincar\'e inequality: For $v \in \bVV$, there exists $\textsf{C} = \textsf{C}(\Omega)$ such that
\begin{equation}\label{eq:Poincare}
\| v \|_{1,\O} \leq \textsf{C}| v |_{1,\O} \qquad \forall v \in \bVV.
\end{equation}
For the sake of simplicity, we assume that the vector-valued couterpart of \eqref{eq:Poincare}, i.e., $\| \mathbf{v} \|_{1,\O} \leq \textsf{C}| \mathbf{v} |_{1,\O}$ for all $\mathbf{v} \in \mathbf{V}$, holds with the same constant $\textsf{C}$.


\subsection{Data assumptions}
\label{sec:data}

We make the following assumptions on the viscocity $\nu(\cdot)$ and the diffusion coefficient $\kappa(\cdot)$.

\begin{itemize}
\item[\textbf{A0})] $\kappa(\cdot)$ is extrictly positive, bounded, and Lipschitz continuous, i.e., there exist constants $\kappa_{*}, \kappa^{*}, \kappa_{lip} > 0$ such that 
\begin{equation*}
0 < \kappa_* \leq \kappa(r) \leq \kappa^*,
\qquad 
|\kappa(r_1) - \kappa(r_2)| \leq \kappa_{lip}|r_1 - r_2| 
\qquad \forall r,r_1,r_2 \in \mathbb{R}.
\end{equation*}
\item[\textbf{A1})] $\nu(\cdot)$ is extrictly positive, bounded, and Lipschitz continuous, i.e., there exist constants $\nu_{*}, \nu^{*}, \nu_{lip} > 0$ such that 
\begin{equation*}
\begin{split}
0<\nu_* \leq \nu(r) \leq \nu^*, 
\qquad
|\nu(r_1) - \nu(r_2)| \leq \nu_{lip}|r_1 - r_2|
\qquad \forall r,r_1,r_2 \in \mathbb{R}.
\end{split}
\end{equation*}
\end{itemize}

\section{A convective Brinkman--Forchheimer problem}
\label{sec:BF}
We present existence and uniqueness results for the following weak formulation of a convective Brinkman--Forchheimer problem: Given $\boldsymbol{f}\in[\H^{-1}(\O)]^2$, find $(\bu,\textsf{p})\in\bV\times\Q$ such that 
\begin{equation}\label{def:weak_BF}
a_L(\bu,\bv)+c_N(\bu;\bu,\bv)+c_{F}(\bu;\bu,\bv)+ d(\bu,\bv) + b(\bv,\textsf{p})
= 
\left\langle\boldsymbol{f},\bv\right\rangle,
\quad 
b(\bu,\textsf{q}) = 0,
\end{equation}
for all $(\bv,\textsf{q})\in\bV\times\Q$. Here, $\left\langle \cdot ,\cdot \right\rangle$ denotes the duality pairing between $[\H^{-1}(\O)]^2$ and $\mathbf{V}$. The forms that occur in \eqref{def:weak_BF} are defined as follows: $a_L:\bV\times\bV \rightarrow \mathbb{R}$, $c_N, c_F :\bV\times\bV\times\bV \rightarrow \mathbb{R}$, $d:\bV\times\bV \rightarrow \mathbb{R}$, and $b:\bV\times\Q \rightarrow \mathbb{R}$ are such that
\begin{align}
a_L(\bv,\bw) & := \displaystyle \int_{\O} \nu\nabla \bv : \nabla \bw,
\qquad
c_N(\bz;\bv,\bw)  :=\displaystyle \int_{\O} (\bz\cdot\nabla)\bv\cdot\bw,
\label{eq:bilinear_form_aL_cN}
\\
c_{F}(\bz;\bv,\bw) & :=\displaystyle \int_{\O} |\bz|^{r-2}\bv\cdot\bw,
\qquad
d(\bv,\bw) :=\displaystyle \int_{\O} \bv\cdot\bw,
\label{eq:bilinear_form_cF_d}
\end{align}
and $b(\bv,\textsf{q}):= -\int_{\O} \textsf{q}\div(\bv)$.

We list some of the most important properties that these forms satisfy:

\begin{itemize}[leftmargin=*]
\item $a_L(\cdot,\cdot)$ is a coercive and continuous bilinear form: For every $\bv,\bw \in \bV$, we have
\begin{equation*}
a_L(\bv,\bv) \geq \nu_{*}|\bv|_{1,\O}^2, 
\qquad 
a_L(\bv,\bw) \leq \nu^*|\bv|_{1,\O}|\bw|_{1,\O}.
\end{equation*}

\item $c_N(\cdot;\cdot,\cdot)$ is skew-symmetric: For every $\bz \in \bZ$ and $\bv,\bw \in \bV$, we have
\begin{equation}\label{eq:c_bound}
c_N(\bz;\bv,\bw) + c_N(\bz;\bw,\bv)= 0,
\qquad
c_N(\bz;\bv,\bv) = 0.
\end{equation}
In addition, $c_N(\cdot;\cdot,\cdot)$ is continuous: For every $\bz, \bv,\bw \in \bV$, we have
\begin{equation}
 c_N(\bz;\bv,\bw) \leq C_{N}|\bz|_{1,\O}|\bv|_{1,\O}|\bw|_{1,\O}.
\end{equation}

\item $c_{F}(\cdot;\cdot,\cdot)$ is continuous: For every $\bz, \bv,\bw \in \bV$, we have
\begin{equation}\label{eq:cF_bound}
c_{F}(\bz;\bv,\bw) \leq C_{F}|\bz|_{1,\O}^{r-2}|\bv|_{1,\O}|\bw|_{1,\O}.
\end{equation}
\item $b(\cdot,\cdot)$ is a continuous bilinear form: For every $\textsf{q} \in \Q$ and $\bv \in \bV$, we have
\begin{equation}\label{def:b_bound}
b(\bv,\textsf{q}) \leq |\bv|_{1,\O}\|\textsf{q}\|_{0,\O}.
\end{equation}
Moreover, $b(\cdot,\cdot)$ satisfies an inf-sup condition: There exists a constant $\beta>0$ such that
\begin{equation}\label{eq:infsupcontinua}
\underset{\bv \in \bV}{\sup} \; \dfrac{b(\bv,\textsf{q})}{|\bv|_{1,\O}} \geq \beta \|\textsf{q}\|_{0,\O}
\quad
\forall \textsf{q} \in \Q.
\end{equation}
\end{itemize} 

\subsection{Existence, stability, and uniqueness results}

We present existence, uniqueness, and stability results for solutions of \eqref{def:weak_BF}. In particular, we present the existence of solutions without restriction on the data and a global uniqueness result when the data is suitably restricted \cite[Theorem 1, Proposition 2, and Theorem 3]{campana2024finite}.

\begin{proposition}[existence, stability, and uniqueness]\label{pro:existencia1}
Let $\O\subset \mathbb{R}^2$ be an open and bounded domain with Lipschitz boundary $\partial\O$, and let $\boldsymbol{f}\in[\H^{-1}(\O)]^2$. Then, there exists at least one solution $(\bu,\textsf{p})\in\bV\times\Q$ for problem \eqref{def:weak_BF} which satisfies 
\[
|\bu|_{1,\O} \leq \nu_*^{-1}\|\boldsymbol{f}\|_{-1}, 
\qquad
\|\textsf{p}\|_{0,\O} \leq \beta^{-1}\Lambda(\boldsymbol{f})\|\boldsymbol{f}\|_{-1},
\]
 where $\Lambda(\boldsymbol{f}) := 1 + \nu^*\nu_*^{-1} + C_N\nu_*^{-2}\|\boldsymbol{f}\|_{-1}+\textsf{C}^2\nu_*^{-1} + C_{F}\nu_*^{1-r}\|\boldsymbol{f}\|_{-1}^{r-2}$. If, in addition, $C_N\|\boldsymbol{f}\|_{-1} < \nu_*^2$, then there is a unique pair $(\bu,\textsf{p})\in\bV\times\Q$ that solves \eqref{def:weak_BF}.
\end{proposition}

\section{A nonlinear heat equation}\label{sec:heat}

We consider the following nonlinear heat equation in weak form: Given $g\in\H^{-1}(\O)$ and $\bv \in \bZ$, find $T\in\bVV$ such that
\begin{equation}\label{def:weak_H}
\mathfrak{a}(T;T,S)+\mathfrak{c}(\bv;T,S) = \left\langle g,S\right\rangle
\qquad
\forall S\in\bVV.
\end{equation}
Here, $\left\langle\cdot,\cdot\right\rangle$ denotes the duality pairing between $\H^{-1}(\O)$ and $\H^1(\O)$. The forms $\mathfrak{a}$ and $\mathfrak{b}$ are defined as follows:
\begin{align}
\label{eq:bilinear_form_fraka}
\mathfrak{a}:\bVV\times\bVV\times\bVV \rightarrow \mathbb{R}, \qquad \mathfrak{a}(X;R,S)&:= \displaystyle \int_{\O} \kappa(X)\nabla R \cdot \nabla S, 
\\ 
\label{eq:bilinear_form_frakc}
\mathfrak{c}:\bV\times\bVV\times\bVV \rightarrow \mathbb{R}, \qquad \mathfrak{c}(\bz;R,S)&:= \displaystyle \int_{\O} (\bz\cdot\nabla R)S.
\end{align}

We list some of the most important properties that these forms satisfy:

\begin{itemize}
\item Given $X \in \bVV$, the form $\mathfrak{a}(X;\cdot,\cdot)$ is coercive and continuous: 
\begin{equation}\label{eq:atemp_bound}
\mathfrak{a}(X;S,S) \geq \kappa_{*}|S|_{1,\O}^2, 
\qquad \mathfrak{a}(X;R,S) \leq \kappa^*|R|_{1,\O}|S|_{1,\O}
\qquad
\forall R,S \in \bVV.
\end{equation}
\item $\mathfrak{c}(\cdot;\cdot,\cdot)$ is skew-symmetric: For every $\bz \in \bZ$ and $R,S \in \bVV$, we have
\begin{equation}\label{eq:ctemp_bound}
\mathfrak{c}(\bz;R,S) + \mathfrak{c}(\bz;S,R) = 0, 
\qquad \mathfrak{c}(\bz;S,S) = 0.
\end{equation}
In addition, $\mathfrak{c}(\cdot;\cdot,\cdot)$ is continuous: For every $\bz \in \bV$ and $R,S \in \bVV$, we have
\begin{equation}
\mathfrak{c}(\bz;R,S) 
\leq 
\mathfrak{C}|\bz|_{1,\O}|R|_{1,\O}|S|_{1,\O}.
\end{equation}

\end{itemize} 

\subsection{Existence, stability, and uniqueness results} 
We present existence, stability, and uniqueness results for problem \eqref{def:weak_H} \cite[Theorem 4 and Remark 6]{campana2024finite}. 

\begin{proposition}[existence, stability, and uniqueness]\label{pro:existencia2}
Let $\Omega\subset\mathbb{R}^2$ be an open and bounded domain with Lipschitz boundary $\partial\Omega$, and let $g\in\H^{-1}(\O)$. Then, there exists at least one solution $T\in\H_0^1(\O)$ of problem \eqref{def:weak_H} which satisfies 
\begin{equation*}
|T|_{1,\O} \leq \kappa_*^{-1}\|g\|_{-1}.
\end{equation*}
If, in addition, \eqref{def:weak_H} has a solution $T_1 \in \W_{\mathfrak{q}}^1(\O) \cap \H_0^1(\Omega)$ such that
\[
 \kappa_*^{-1}\kappa_{lip}C_{\mathfrak{p}\hookrightarrow 2}|T|_{\W_{\mathfrak{q}}^1(\O)} < 1,
\]
then problem \eqref{def:weak_H} has no other solution $T_2 \in \H_0^1(\Omega)$. Here, $C_{\mathfrak{p}\hookrightarrow 2}>0$ is the best constant in $\H_0^1(\O)\hookrightarrow\L^{\mathfrak{p}}(\O)$, where $\mathfrak{p} < \infty$, and $\mathfrak{q}$ is chosen such that $1/\mathfrak{p}+1/\mathfrak{q}=1/2$.
\end{proposition}

\section{The coupled problem}\label{sec:BFH}
We now introduce a weak formulation for the system \eqref{def:BFH} and review existence and uniqueness results. The weak formulation is as follows: Given $\boldsymbol{f}\in[\H^{-1}(\O)]^2$ and $g\in\H^{-1}(\O)$, find $(\bu,\textsf{p},T)\in\bV\times\Q\times\bVV$ such that
\begin{equation}\label{def:weak_BFH}
\left\{\begin{array}{rcl}
a(T;\bu,\bv) + c_N(\bu;\bu,\bv)+c_{F}(\bu;\bu,\bv)+ d(\bu,\bv) + b(\bv,\textsf{p})&=& \left\langle\boldsymbol{f},\bv\right\rangle,
\\
b(\bu,\textsf{q}) &=& 0, \\
\mathfrak{a}(T;T,S)+\mathfrak{c}(\bu;T,S) &=& \left\langle g,S\right\rangle,
\end{array}\right.
\end{equation}
for all $(\bv,\textsf{q},S) \in \bV\times\Q\times\bVV$. The forms $c_N(\cdot; \cdot,\cdot)$, $c_F(\cdot; \cdot,\cdot)$, $d(\cdot,\cdot)$, and $b(\cdot,\cdot)$ are defined in Section \ref{sec:BF} while $\mathfrak{a}(\cdot; \cdot,\cdot)$ and $\mathfrak{c}(\cdot; \cdot,\cdot)$ are defined in Section \ref{sec:heat}. The form $a$ is defined as follows:
\begin{equation}
a:\bVV\times\bV\times\bV \rightarrow \mathbb{R},
\qquad
a(X;\bv,\bw):=\displaystyle \int_{\O} \nu(X)\nabla \bv : \nabla \bw.
\label{eq:bilinear_form_a}
\end{equation}

Since $\nu$ satisfies the properties in \textbf{A1}), it is immediate that, given $X \in \bVV$,
\begin{equation}\label{eq:a_coer_new}
a(X;\bv,\bv) \geq \nu_{*}|\bv|_{1,\O}^2, 
\qquad 
a(X;\bv,\bw) \leq \nu^*|\bv|_{1,\O}|\bw|_{1,\O}
\qquad
\forall \bv,\bw \in \bV.
\end{equation}

\begin{remark}[skew-symmetry]\rm
Let $\bz \in \bZ$. Define the forms $c_N^{\textit{skew}}(\bz ; \cdot, \cdot): \bV \times \bV \rightarrow \mathbb{R}$ and $\mathfrak{c}^{\textit{skew}}(\bz ; \cdot, \cdot): \bVV \times \bVV \rightarrow \mathbb{R}$ by
$
c_N^{\textit{skew}}(\bz;\bv,\bw) := 1/2 [c_N(\bz;\bv,\bw)-c_N(\bz;\bw,\bv)]
$
and
$
\mathfrak{c}^{\textit{skew}}(\bz;R,S) := 1/2[\mathfrak{c}(\bz;R,S)-\mathfrak{c}(\bz;S,R)]
$, 
respectively. As discussed in Sections \ref{sec:BF} and \ref{sec:heat}, $c_N$ and $\mathfrak{c}$ are skew-symmetric. As a result, $c_N(\bz;\bv,\bw) = c_N^{\textit{skew}}(\bz;\bv,\bw)$ for every $\bv,\bw \in \bV$ and $\mathfrak{c}(\bz;R,S) = \mathfrak{c}^{\textit{skew}}(\bz;R,S)$ for every $R,S \in \bVV$. This observation will be important for the development of a virtual element numerical scheme.
\label{rem:skew-symmetry}
\end{remark}

\subsection{Existence, stability, and uniqueness results}
We review the existence of solutions for the system \eqref{def:weak_BFH} without restriction on the data and a global uniqueness result when the data is suitably restricted and the solution is slightly smoother \cite[Theorems 8 and 10]{campana2024finite}.

\begin{proposition}[existence, stability, and uniqueness]\label{pro:existence3}
Let $\O\subset \mathbb{R}^2$ be an open and bounded domain with Lipschitz boundary $\partial\O$, let $\boldsymbol{f}\in[\H^{-1}(\O)]^2$, and let $g \in \H^{-1}(\Omega)$. Let $\nu$ and $\kappa$ be as in Section \ref{sec:data}. Then, the nonlinear system \eqref{def:weak_BFH} has at least one solution $(\bu,\textsf{p},T) \in \bV\times\Q\times\bVV$. Moreover, we have
\[
|\bu|_{1,\O} \leq \nu_*^{-1}\|\boldsymbol{f}\|_{-1}, 
\quad 
\|\textsf{p}\|_{0,\O} \leq \beta^{-1}\Lambda(\boldsymbol{f})\|\boldsymbol{f}\|_{-1},
\quad
|T|_{1,\O} \leq \kappa_*^{-1}\|g\|_{-1},
\]
where $\Lambda(\boldsymbol{f})$ is defined in the statement of Proposition \ref{pro:existencia1}.  Furthermore, if the system \eqref{def:weak_BFH} has a solution $(\bu_1,\textsf{p}_1,T_1) \in (\bV\cap \mtW_{2+\varepsilon}^1(\O))\times\Q\times(\bVV\cap\W_{2+\varepsilon}^1(\O))$ for some $\varepsilon > 0$ such that
\begin{align*}
\dfrac{C_N}{\nu_*^2}\|\boldsymbol{f}\|_{-1} + \dfrac{\nu_{lip}\mathcal{C}_{\varepsilon}\mathcal{C}^2_{4 \rightarrow 2}\|g\|_{-1}|\bu_1|_{\mtW_{2+\varepsilon}^1(\O)}}{\nu_*\kappa_*(\kappa_*-\kappa_{lip}C_{\epsilon}|T_1|_{\W_{2+\varepsilon}^1(\O)})} &< 1,
\\
\kappa_{lip}\mathcal{C}_{\epsilon}|T_1|_{\W_{2+\varepsilon}^1(\O)} &< \kappa_*,
\end{align*}
then the system \eqref{def:weak_BFH} has no other solution $(\bu_2,\textsf{p}_2,T_2)\in\bV\times\Q\times\bVV$. Here, $\mathcal{C}_{4\hookrightarrow 2}$ and $\mathcal{C}_{\varepsilon}$ denote the best constant in $\bVV \hookrightarrow \L^4(\O)$ and $\bVV \hookrightarrow \L^{2(2+\varepsilon)/\varepsilon}(\O)$ respectively.
\end{proposition}

\section{A virtual  element approximation}
\label{sec:vem}
In this section, we introduce a virtual element approximation for the nonlinear system \eqref{def:weak_BFH} and derive error estimates. For this purpose, we first introduce some notions and basic ingredients \cite{MR2997471,MR4586821}. From now on, we assume that $\Omega$ is a polygonal and bounded domain with Lipschitz boundary.

Let $\{\CT_h\}_{h>0}$ be a sequence of decompositions of $\Omega$ into general polygonal elements $E$. Here, $h:=\max \{ h_E: E\in\CT_h \}$ and $h_E$ denotes the the diameter of the element $E$. We denote by $|E|$ the area of the element $E \in \CT_h$.

\subsection{Mesh regularity}
We make the following assumptions on the sequence $\{\CT_h\}_{h>0}$ \cite{MR2997471,MR3626409,MR3796371,AVV,MR4586821}: There exists $\varrho > 0$ such that for all $h$ and every $E$ in $\CT_h$,
\begin{itemize}
\item[\textbf{A2})] $E$ is star-shaped with respect to a ball $B_\textsf{r}$ with radius $\textsf{r} \geq \varrho h_{E}$ \cite[assumption \textbf{A0.2}]{MR2997471} and
\item[\textbf{A3})] the distance between any two vertices of $E$ is $\ell \geq \varrho h_{E}$ \cite[assumption \textbf{A0.3}]{MR2997471}.
\end{itemize}

\subsection{Basic ingredients}
Let $k \in \mathbb{N}$, $t \in \mathbb{R}^{+}$, and let $p \in [1,+\infty]$. Let $A$ be an open and bounded domain in $\mathbb{R}^2$. We introduce some basic spaces that will be useful later. First, we introduce $\mathbb{P}_{k}(A)$ --- the set of polynomials on $A$ of degree less than or equal to $k$ ---  with the convention that $\mathbb{P}_{-1}(A) = \{ 0 \}$. Secondly, $\mathbb{P}_k(\mathcal{T}_h) = \{q \in \L^2(\O) :  q|_{E} \in \mathbb{P}_{k}(E) \; \forall E \in \mathcal{T}_h\}$. Finally, we introduce the space $\W^{t}_{p}(\mathcal{T}_h) = \{ v \in \L^{2}(\O) : \, v|_{E} \in \W^{t}_{p}(E) \; \forall E \in \mathcal{T}_h\}$. We equip the space $\W^{t}_{p}(\mathcal{T}_h)$ with the broken norm $\| \cdot \|_{\W_p^t(\CT_h)}$ and the broken seminorm $| \cdot |_{\W_p^t(\CT_h)}$, which are defined by
\begin{equation*}
\begin{array}{ll}
\|v\|_{\W_p^t(\CT_h)} := \left[ \displaystyle \sum_{E\in\CT_h} \|v\|_{\W_p^t(E)}^{p} \right]^{\frac{1}{p}}, 
\quad 
|v|_{\W_p^t(\CT_h)} := \left[ \displaystyle \sum_{E\in\CT_h} |v|_{\W_p^t(E)}^{p} \right]^{\frac{1}{p}}, 
\quad
p \in [1, +\infty),
\end{array}
\end{equation*}
respectively. If $p = + \infty$, then $\|v\|_{\W_{\infty}^t(\CT_h)} :=\max\{ \|v\|_{\W_{\infty}^t(E)}: E\in\CT_h \}$. We denote the vector-valued counterparts of $\mathbb{P}_{k}(A)$, $\mathbb{P}_k(\mathcal{T}_h)$, $\W^{t}_{p}(\mathcal{T}_h)$ by $[\mathbb{P}_k(A)]^2$, $[\mathbb{P}_k(\CT_h)]^2$, and $\mtW^{t}_{p}(\mathcal{T}_h)$, respectively.

For an element $E\in\CT_h$, we denote by $\textbf{x}_E$ the centroid of $E$. For a multi-index $\boldsymbol{\alpha} = (\alpha_1,\alpha_2) \in \mathbb{N}^2$, we define $|\boldsymbol{\alpha}| = \alpha_1 + \alpha_2$. Let $n \in \mathbb{N}$. A natural basis associated to the space $\mathbb{P}_n(E)$ is the set of normalized monomials
\begin{equation}
\mathbb{M}_{n}(E) := \left\lbrace 
\mathbf{m}_{\boldsymbol{\alpha}}: \boldsymbol{\alpha} \in \mathbb{N}^2, \, |\boldsymbol{\alpha}| \leq n \right\rbrace,
\qquad
\mathbf{m}_{\boldsymbol{\alpha}}(\mathbf{x}) := \left(\dfrac{\mathbf{x}-\textbf{x}_E}{h_E}\right)^{\boldsymbol{\alpha}}.
\end{equation}
Let $m \in \mathbb{N}$ be such that $m \leq n$. We also introduce
\begin{equation}
\mathbb{P}_{n\backslash m}(E) := \textrm{span}\left\lbrace 
\mathbf{m}_{\boldsymbol{\alpha}}: 
\boldsymbol{\alpha} \in \mathbb{N}^2, \, m + 1 \leq |\boldsymbol{\alpha}| \leq n \right\rbrace.
\end{equation}

\subsection{Projections}
In this section we introduce some appropriate projections that will be useful for our analysis. Let $E \in \CT_h$, and let $n \in \mathbb{N} \cup \{0\}$. We introduce the $\L^{2}(E)$-orthogonal projection as follows:
\[
 \Pi^{0,E}_{n}: \L^2(E) \rightarrow \mathbb{P}_n(E):
 \quad
 (q_n, v  - \Pi^{0,E}_{n}v )_{\L^2(E)} = 0 \,\, \forall q_{n} \in \mathbb{P}_{n}(E).
\]
We denote by $\boldsymbol{\Pi}^{0,E}_{n}$ the vector-valued couterpart of $\Pi^{0,E}_{n}$. We also introduce the $\H^{1}(E)$-seminorm projection, which is defined as follows:
\[
 \Pi^{\nabla,E}_{n}: \H^1(E) \rightarrow \mathbb{P}_n(E):
 \quad
 (\nabla q_n, \nabla(v  - \Pi^{\nabla,E}_{n}v) )_{\L^2(E)} = 0 \,\, \forall q_{n} \in \mathbb{P}_{n}(E)
\]
and $\int_{\partial E} (v  - \Pi^{\nabla,E}_{n}v) = 0$. We denote by $\boldsymbol{\Pi}^{\nabla,E}_{n}$ the vector-valued couterpart of $\Pi^{0,E}_{n}$.

\subsection{Virtual element spaces for the velocity variable}

Following \cite{MR3626409,MR3796371}, we introduce a finite-dimensional local virtual space for $E \in \CT_h$:
\begin{multline*}
\bWh(E):= \{\bv_h \in [\H^{1}(E)]^{2} \cap [\mathcal{C}^{0}(\partial E)]^{2} \colon -\Delta \bv_h - \nabla \textsf{q} \in \mathbf{x}^{\perp}\mathbb{P}_{k-1}(E),\; \text{for}\,\text{some} \, \textsf{q} \in \L_0^2(E), \\
\div \bv_h \in \mathbb{P}_{k-1}(E), \; \bv_h|_{e} \in [\mathbb{P}_k(e)]^{2} \; \forall e \subset \partial E\},
\qquad
\mathbf{x}^{\perp} := (x_{2},-x_{1}).
\end{multline*}
With this space at hand, we introduce $\bVh(E)$ as a restriction of the space $\bWh(E)$:
\begin{equation*}
\bVh(E) := \{\bv_h \in \bWh(E):\, (\bv_h - \bPi^{\nabla,E}_k\bv_h,\mathbf{x}^{\perp}\textbf{q}_{k-1})_{\mathbf{L}^2(E)} = 0 \, \forall \textbf{q}_{k-1} \in \mathbb{P}_{k-1\backslash k-3}(E)\}.
\end{equation*}
A similar definition of $\bV_h(E)$ can be found in \cite[definition (3.7)]{AVV}; compare with \cite[definition (17)]{MR3796371}.

We recall some properties of the virtual space $\bVh(E)$ \cite{MR3626409,MR3737081}, which are based on the description in \cite[Section 3.2]{AVV} and \cite[Proposition 3.1]{MR3796371}.

\begin{itemize}
 \item [\textbf{(P1)}] \textbf{Polynomial inclusion:} $[\mathbb{P}_k(E)]^2 \subseteq \bVh(E)$;
 \item [\textbf{(P2)}] \textbf{Degrees of freedom (DoFs):} the following linear operators $\mathbf{D}_{\mathbf{V}}$ constitute a set of DoFs for the virtual element space $\bVh(E)$:
 \begin{itemize}
 \item [$\mathbf{D}_{\mathbf{V}} \mathbf{1}$] the values of $\mathbf{v}_h$ at the vertices of $E$,
 \item [$\mathbf{D}_{\mathbf{V}} \mathbf{2}$] the values of $\bv_h$ at $k-1$ distinct points of every edge $e \in \partial E$,
 \item [$\mathbf{D}_{\mathbf{V}} \mathbf{3}$] the moments of $\bv_h$
 \begin{equation*}
\dfrac{1}{|E|}\displaystyle \int_{E} \bv_h\cdot\mathbf{m}^{\perp} \mathbf{m}_{\boldsymbol{\alpha}} \quad \forall \mathbf{m}_{\boldsymbol{\alpha}} \in \mathbb{M}_{k-3}(E), 
\qquad 
\mathbf{m}^{\perp} = \dfrac{\mathbf{x}^{\perp}-\mathbf{x}_E^{\perp}}{h_E}.
\end{equation*}
\item [$\mathbf{D}_{\mathbf{V}} \mathbf{4}$] the moments of $\div \bv_h$
\begin{equation*}
\dfrac{h_{E}}{|E|} \displaystyle \int_{E} \div\bv_h \mathbf{m}_{\boldsymbol{\alpha}}
\quad 
\forall \mathbf{m}_{\boldsymbol{\alpha}} \in \mathbb{M}_{k-1}(E), 
\quad
|\boldsymbol{\alpha}| = \alpha_{1}+\alpha_{2} > 0.
\end{equation*}
 \end{itemize}
\end{itemize}

As a final ingredient, we introduce $\bXi_{k-1}^{0,E}: \boldsymbol{\nabla} \bVh(E) \rightarrow [\mathbb{P}_{k-1}(E)]^{2 \times 2}$ such that
\begin{equation*}
(\textbf{q}_{k-1},\nabla \bv_h - \bXi^{\nabla,E}_{k-1} \nabla \bv_h)_{\mathbf{L}^2(E)} = 0 \,\, \forall \textbf{q}_{k-1} \in [\mathbb{P}_{k-1}(E)]^{2 \times 2}.
\end{equation*}

The following remark is in order.

\begin{remark}[computability]\rm
The DoFs $\mathbf{D}_{\mathbf{V}}$ allow an exact computability of
\[
\bPi^{0,E}_k: \bVh(E) \rightarrow [\mathbb{P}_k(E)]^{2},
\qquad
\bXi_{k-1}^{0,E}: \boldsymbol{\nabla} \bVh(E) \rightarrow [\mathbb{P}_{k-1}(E)]^{2 \times 2}
\]
in the following sense: Given $\mathbf{v}_h \in \bVh(E)$, $\bPi^{0,E}_k \mathbf{v}_h$ and $\bXi_{k-1}^{0,E} \nabla \mathbf{v}_h$ can be computed using only, as unique information, the DoFs values $\mathbf{D}_{\mathbf{V}}$ of $\mathbf{v}_h$ \cite[Proposition 3.2]{MR3796371}.
\label{rk:computability_1}
\end{remark}

Finally, we define the \emph{global velocity space} $\bVh$ as follows \cite{MR3626409,MR3796371,AVV}:
\begin{equation*}
\bVh := \{\bv_h \in \bV : \, \bv_h|_{E} \in \bVh(E) \, \forall E \in \CT_h\}.
\end{equation*}

\subsection{Finite element space for the pressure variable}

We introduce the following finite element space for approximating a pressure variable:
\begin{equation*}
\Q_h := \{\textsf{q}_h \in \Q \colon \, \textsf{q}_h|_{E} \in \mathbb{P}_{k-1}(E) \, \forall E \in \CT_h\}.
\end{equation*}
It is important to note that the pair $(\bVh,\Q_h)$ satisfies the following discrete inf-sup condition: there exists a constant $\tilde{\beta}>0$ independent of $h$
such that
\begin{equation}\label{eq:infsupdiscreta}
\underset{\bv_h \in \bVh}{\sup} \; \dfrac{b(\bv_h,\textsf{q}_h)}{|\bv_h|_{1,\O}} \geq \tilde{\beta} \|\textsf{q}_h\|_{0,\O} \quad \forall \textsf{q}_h \in \Q_h;
\end{equation}
\cite[Proposition 3.4]{MR3796371}. Let us now introduce the discrete kernel
\begin{equation}
\bZh := \{ \bv_h \in \bVh: \, b(\bv_h,\textsf{q}_h) = 0 \,\, \forall \textsf{q}_h \in \Q_h\}. 
\label{eq:Z_h}
\end{equation}

The following observation is important:

\begin{remark}[divergence-free]\rm
Let $\bv_h \in \bVh$. Given that $\div \bv_h|_E \in \mathbb{P}_{k-1}(E)$, we deduce that $\mathbf{Z}_h \subseteq \mathbf{Z}$: the functions in the discrete kernel are exactly divergence-free.
\label{rem:divergence_free}
\end{remark}

\subsection{Virtual element spaces for the temperature variable}

To introduce a virtual element space for approximating a temperature variable, we first introduce, for each $E \in \CT_h$, the finite-dimensional local virtual space \cite[definition (16)]{MR4567234}
\begin{align*}
\bUh(E):= \{S_h \in \H^{1}(E) \cap \mathcal{C}^{0}(\partial E) \colon \Delta S_h \in\mathbb{P}_{k}(E),\; S_h|_{e} \in \mathbb{P}_k(e) \; \forall e \in \partial E\}.
\end{align*}
With this space at hand, on each $E \in \CT_h$ we introduce \cite[definition (3.13)]{AVV}
\begin{equation*}
\bVVh(E) := \{S_h \in \bUh(E): \, (S_h - \Pi^{\nabla,E}_k S_h, q_k)_{\L^2(E)} = 0 \,\, \forall q_k \in \mathbb{P}_{k\setminus k-2}(E) \}.
\end{equation*}

We recall some properties of the space $\bVVh(E)$ based on the presentation in \cite[Section 3.3]{AVV} and \cite[Section 2]{MR4567234}:

\begin{itemize}
\item[\textbf{(P3)}] \textbf{Polynomial inclusion}: $\mathbb{P}_k(E) \subseteq \bVVh(E)$;
\item[\textbf{(P4)}] \textbf{Degrees of freedom}: the following linear operators $D_V$ constitute a set of DoFs for the virtual element space $\bVVh(E)$:
\begin{itemize}
\item[$D_V1$] the values of $S_h$ at the vertices of $E$,
\item[$D_V2$] the values of $S_h$ at $k-1$ distinct points of every edge $e \in \partial E$,
\item[$D_V3$] the moments of $S_h$
\begin{equation*}
\dfrac{1}{|E|} \displaystyle \int_{E} S_h \textrm{m}_{\boldsymbol{\alpha}} \quad \forall \textrm{m}_{\boldsymbol{\alpha}} \in \mathbb{M}_{k-2}(E).
\end{equation*}
\end{itemize}
\end{itemize}

\begin{remark}[computability]\rm
The DoFs $D_V$ allow an exact computability of 
\[
 \Pi^{0,E}_k: \bVV_h(E) \rightarrow \mathbb{P}_k(E),
 \qquad
 \bPi^{0,E}_{k-1}: \nabla \bVV_h(E) \rightarrow [\mathbb{P}_{k-1}(E)]^2,
\]
\cite[Proposition 3.2]{MR3796371}; compare with Remark \ref{rk:computability_1}.
\end{remark}

Finally, we introduce a global virtual space to approximate a temperature variable
\begin{equation*}
\bVVh := \{S_h \in \bVV: \, S_h|_{E} \in \bVVh(E) \, \forall E \in \CT_h\}.
\end{equation*}

\subsection{Virtual element forms}

With the discrete spaces $\bVh$, $\Q_h$, and $\bVVh$  in hand, and following \cite{AVV,MR3796371,NMTMA-17-210} we now introduce discrete versions of the continuous forms involved in the weak problem \eqref{def:weak_BFH}.
\\
\begin{itemize}[leftmargin=*]
\item The discrete counterpart of $a(\cdot;\cdot,\cdot)$ (cf. \eqref{eq:bilinear_form_a}) is defined by 
\begin{equation*}
a_h(\cdot;\cdot,\cdot):\bVVh\times\bVh\times\bVh \rightarrow \mathbb{R}, 
\qquad
a_h(X_h;\bv_h,\bw_h) := \ds 
\sum_{E\in\CT_h} a_h^E(X_h;\bv_h,\bw_h),
\end{equation*}
where $a_h^{E}(\cdot;\cdot,\cdot):\bVVh(E)\times\bVh(E)\times\bVh(E) \rightarrow \mathbb{R}$ is given by
\begin{multline*}
a_h^E(X_h;\bv_h,\bw_h) := 
\displaystyle 
\int_E \nu(\Pi_k^{0,E}X_h)\bXi^{0,E}_{k-1}\nabla\bv_h:\bXi^{0,E}_{k-1}\nabla\bw_h 
\\
+ 
\nu(\Pi^{0,E}_0X_h)S_{\textit{V}}^{E}((\mathbf{I}-\bPi^{0,E}_k)\bv_h,(\mathbf{I}-\bPi^{0,E}_k)\bw_h).
\end{multline*}
Here, $S_{\textit{V}}^{E}(\cdot,\cdot):\bVh(E)\times\bVh(E) \rightarrow \mathbb{R}$ is a computable symmetric form that satisfies
\begin{equation}
|\bv_h|_{1,E}^2 \lesssim S_{\textit{V}}^{E}(\bv_h,\bv_h) \lesssim |\bv_h|_{1,E}^2 \quad \forall \bv_h \in \bVh(E) \cap \ker(\bPi^{0,E}_k).
\label{eq:S_v}
\end{equation}
\item The discrete counterpart of $\mathfrak{a}(\cdot;\cdot,\cdot)$ (cf. \eqref{eq:bilinear_form_fraka}) is defined by
\begin{equation}
\mathfrak{a}_h(\cdot;\cdot,\cdot):\bVVh\times\bVVh\times\bVVh \rightarrow \mathbb{R}, 
\qquad 
\mathfrak{a}_h(X_h;R_h,S_h) := \ds \sum_{E\in\CT_h} \mathfrak{a}_h^E(X_h;R_h,S_h),
\label{eq:bilinear_form_fraka_h}
\end{equation}
where $\mathfrak{a}_h^{E}(\cdot;\cdot,\cdot):\bVVh(E)\times\bVVh(E)\times\bVVh(E) \rightarrow \mathbb{R}$ is given by
\begin{multline*}
\mathfrak{a}_h^E(X_h;R_h,S_h) := \displaystyle \int_E \kappa(\Pi^{0,E}_k X_h)\bPi^{0,E}_{k-1}\nabla R_h\cdot \bPi^{0,E}_{k-1}\nabla S_h \\
+ \kappa(\Pi^{0,E}_0 X_h)S_{\textit{T}}^{E}((\I-\Pi^{0,E}_k)R_h,(\I-\Pi^{0,E}_k)S_h).
\end{multline*}
Here, $S_{\textit{T}}^{E}(\cdot,\cdot):\bVVh(E)\times\bVVh(E) \rightarrow \mathbb{R}$ is a computable symmetric form that satisfies
\begin{equation}
|X_h|_{1,E}^2 \lesssim S_{\textit{T}}^{E}(X_h,X_h) \lesssim |X_h|_{1,E}^2 \quad \forall X_h \in \bVVh(E) \cap \ker(\Pi^{0,E}_k).
\label{eq:S_T}
\end{equation}
\item The discrete counterpart of $c_N(\cdot;\cdot,\cdot)$ (cf. \eqref{eq:bilinear_form_aL_cN}) is defined by
\begin{equation*}
c_{N,h}(\cdot;\cdot,\cdot):\bVh\times\bVh\times\bVh \rightarrow \mathbb{R}, 
\qquad 
c_{N,h}(\bz_h;\bv_h,\bw_h) := \ds \sum_{E\in\CT_h} c_{N,h}^E(\bz_h;\bv_h,\bw_h),
\end{equation*}
where $c_{N,h}^{E}(\cdot;\cdot,\cdot):\bVh(E)\times\bVh(E)\times\bVh(E) \rightarrow \mathbb{R}$ is given by
\begin{equation*}
c_{N,h}^E(\bz_h;\bv_h,\bw_h) := 
\displaystyle 
\int_E [\bXi^{0,E}_{k-1}(\nabla \bv_h) \bPi^{0,E}_k \bz_h]\cdot \bPi^{0,E}_k \bw_h.
\end{equation*}
Following \cite[equation (38)]{MR3796371} and \cite[equation (3.24)]{AVV}, we define
\begin{equation*}
c_{N,h}^{\textit{skew},E}(\bz_h;\bv_h,\bw_h):=\tfrac{1}{2}(c_{N,h}^{E}(\bz_h;\bv_h,\bw_h)-c_{N,h}^{E}(\bz_h;\bw_h,\bv_h)),
\end{equation*}
and $c_{N,h}^{\textit{skew}}(\bz_h;\bv_h,\bw_h) := \ds \sum_{E\in\CT_h} c_{N,h}^{\textit{skew},E}(\bz_h;\bv_h,\bw_h)$.
\item The discrete counterpart of $c_{F}(\cdot;\cdot,\cdot)$ (cf. \eqref{eq:bilinear_form_cF_d}) is defined by
\begin{equation*}
c_{F,h}(\cdot;\cdot,\cdot):\bVh\times\bVh\times\bVh \rightarrow \mathbb{R}, 
\qquad 
c_{F,h}(\bz_h;\bv_h,\bw_h) := \ds \sum_{E\in\CT_h} c_{F,h}^{E}(\bz_h;\bv_h,\bw_h),
\end{equation*}
where $c_{F,h}^{E}(\cdot;\cdot,\cdot):\bVh(E)\times\bVh(E)\times\bVh(E) \rightarrow \mathbb{R}$ is given by 
\begin{equation*}
c_{F,h}^E(\bz_h;\bv_h,\bw_h) := c_{F}^E(\bPi^{0,E}_k\bz_h;\bPi^{0,E}_k\bv_h,\bPi^{0,E}_k\bw_h).
\end{equation*}
\item The discrete counterpart of $\mathfrak{c}(\cdot;\cdot,\cdot)$ (cf. \eqref{eq:bilinear_form_frakc}) is defined by
\begin{equation*}
\mathfrak{c}_{h}(\cdot;\cdot,\cdot):\bVh\times\bVVh\times\bVVh \rightarrow \mathbb{R}, 
\qquad 
\mathfrak{c}_{h}(\bv_h;R_h,S_h) := \ds \sum_{E\in\CT_h} \mathfrak{c}_{h}^{E}(\bv_h;R_h,S_h),
\end{equation*}
where $\mathfrak{c}_{h}^{E}(\cdot;\cdot,\cdot):\bVh(E)\times\bVVh(E)\times\bVVh(E) \rightarrow \mathbb{R}$ is given by
\begin{equation*}
\mathfrak{c}_{h}^E(\bv_h;R_h,S_h) := \displaystyle \int_E (\bPi^{0,E}_k\bv_h \cdot \bPi^{0,E}_{k-1}\nabla R_h)\Pi^{0,E}_{k}S_h.
\end{equation*}
Following \cite[equation (3.25)]{AVV}, we define 
\begin{equation*}
\mathfrak{c}_{h}^{\textit{skew},E}(\bv_h;R_h,S_h):=\displaystyle\tfrac{1}{2}(\mathfrak{c}_{h}^{E}(\bv_h;R_h,S_h)-\mathfrak{c}_{h}^{E}(\bv_h;S_h,R_h)),
\end{equation*}
and $\mathfrak{c}_{h}^{\textit{skew}}(\bv_h;R_h,S_h) := \ds \sum_{E\in\CT_h} \mathfrak{c}_{h}^{\textit{skew},E}(\bv_h;R_h,S_h)$.
\item The discrete counterpart of $d(\cdot,\cdot)$ (cf. \eqref{eq:bilinear_form_cF_d}) is defined by
\begin{equation*}
d_h(\cdot,\cdot) : \bVh\times\bVh \rightarrow \mathbb{R}, 
\qquad d_h(\bv_h,\bw_h) := \ds \sum_{E\in\CT_h} d_h^{E}(\bv_h,\bw_h),
\end{equation*}
where $d_h^{E}(\cdot,\cdot):\bVh(E)\times\bVh(E) \rightarrow \mathbb{R}$ is given by
\begin{equation*}
d_h^E(\bv_h,\bw_h) := \int_E \bPi^{0,E}_k\bv_h \cdot \bPi^{0,E}_k\bw_h.
\end{equation*}
\end{itemize}

\begin{remark}[skew-symmetry]\rm
We note that $c_{N,h}^{\textit{skew}}$ and $\mathfrak{c}_{h}^{\textit{skew}}$ are skew-symmetric by construction, i.e., for $\bz_h \in \bV_h$, 
\[
 c_{N,h}^{\textit{skew},E}(\bz_h;\bv_h,\bv_h) = 0,
 \qquad
 \mathfrak{c}_{h}^{\textit{skew}}(\bz_h;R_h,R_h) = 0,
\]
for every $\bv_h \in \bV_h$ and every $R_h \in \bVV_h$. As usual, this property simplifies the analysis of the corresponding discrete problem \cite{MR0548867,MR0769654}.
\end{remark}

\begin{remark}[the no discretization of $b$]\rm
As mentioned in \cite{MR3626409,MR3796371}, we do not introduce any approximation of the bilinear form $b$. We note that $b(\bv_h,\textsf{q}_h)$ for $\bv_h \in \bV_h$ and $\textsf{q}_h\in \Q_h$ is computable from the DoFs $\mathbf{D}_{\mathbf{V}}\mathbf{1}, \mathbf{D}_{\mathbf{V}}\mathbf{2}, \mathbf{D}_{\mathbf{V}}\mathbf{4}$ because $\textsf{q}_h$ is a polynomial on each element $E \in \CT_h$.
\end{remark}

\subsection{Virtual element forcing}

From now on we will assume that $\boldsymbol{f} \in \mathbf{L}^2(\Omega)$ and $g \in \L^2(\Omega)$. Within this framework, we define the discrete sources $\boldsymbol{f}_h \in [\mathbb{P}_k(\CT_h)]^2$ and $g_h \in \mathbb{P}_k(\CT_h)$ to be such that, for each $E \in \CT_h$,
\begin{equation*}
\boldsymbol{f}_h|_{E} := \bPi^{0,E}_k\boldsymbol{f}, 
\qquad 
g_h|_{E} := \Pi^{0,E}_k g.
\end{equation*}

\subsection{The virtual element method}
With all the ingredients and definitions introduced in the previous sections, we finally design a virtual element discretization for the nonlinear system \eqref{def:weak_BFH}: Find  $(\bu_h,\textsf{p}_h,T_h)\in \bVh\times \Q_h\times\bVVh$ such that

\begin{equation}\label{def:weak_BFH_discrete}
\left\{\begin{array}{rcl}
a_h(T_h;\bu_h,\bv_h)
+
c_{N,h}^{\textit{skew}}(\bu_h;\bu_h,\bv_h)
+
c_{F,h}(\bu_h;\bu_h,\bv_h)
\\
+
d_h(\bu_h,\bv_h) + b(\bv_h,\textsf{p}_h)&=& (\boldsymbol{f}_h,\bv_h)_{0,\O},
\\
b(\bu_h,\textsf{q}_h) &=& 0, 
\\
\mathfrak{a}_h(T_h;T_h,S_h)
+
\mathfrak{c}_h^{\textit{skew}}(\bu_h;T_h,S_h) 
&=& (g_h,S_h)_{0,\O},
\end{array}\right.
\end{equation}
for all $(\bv_h,\textsf{q}_h,S_h) \in \bVh\times\Q_h\times\bVVh$.

\begin{remark}[reduced formulation]\rm
Using the space $\bZh$ defined in \eqref{eq:Z_h}, the problem \eqref{def:weak_BFH_discrete} can be reformulated as follows: Find  $(\bu_h,T_h)\in \bZh\times\bVVh$ such that
\begin{equation}\label{def:weak_BFH_kernel_discrete}
\left\{\begin{array}{rcl}
a_h(T_h;\bu_h,\bv_h)+c_{N,h}^{\textit{skew}}(\bu_h;\bu_h,\bv_h)+c_{F,h}(\bu_h;\bu_h,\bv_h)\\
+ d_h(\bu_h,\bv_h) &=& (\boldsymbol{f}_h,\bv_h)_{0,\O},
\\
\mathfrak{a}_h(T_h;T_h,S_h)+\mathfrak{c}_h^{\textit{skew}}(\bu_h;T_h,S_h) &=& (g_h,S_h)_{0,\O},
\end{array}\right.
\end{equation}
for all $(\bv_h,S_h) \in \bZh\times\bVVh$.
\end{remark}

To provide an analysis for the discrete problem \eqref{def:weak_BFH_discrete}, we list some properties that the discrete forms satisfy:

\begin{itemize}[leftmargin=*]
\item For any $X_h \in \bVV_h$, $a_h(X_h;\cdot,\cdot)$ is a coercive and continuous bilinear form: There exist $\alpha_*>0$ and $\alpha^*>0$ such that, for every $\bv_h,\bw_h \in \bVh$ we have
\begin{equation}\label{eq:ah_bound}
a_h(X_h;\bv_h,\bw_h) \leq \alpha^{*}\nu^*|\bv_h|_{1,\O}|\bw_h|_{1,\O},
\qquad
a_h(X_h;\bv_h,\bv_h) \geq \alpha_{*}\nu_{*}|\bv_h|_{1,\O}^2.
\end{equation}
These inequalities result from standard properties of $\boldsymbol{\Pi}^{0,E}_{n}$ and $\bXi_{k-1}^{0,E}$ in conjunction with basic inequalities and the properties \textbf{A1} and \eqref{eq:S_v} that $\nu$ and $S_V$ satisfy.

\item For any $X_h \in \bVV_h$, $\mathfrak{a}_h(X_h;\cdot,\cdot)$ is a coercive and continuous bilinear form: There exist $\beta_*>0$ and $\beta^*>0$ such that, for every $R_h, S_h \in \bVVh$, we have
\begin{equation}\label{eq:ahtemp_bound}
\mathfrak{a}_h(X_h;R_h,S_h) \leq \beta^{*}\kappa^*|R_h|_{1,\O}|S_h|_{1,\O},
\qquad
\mathfrak{a}_h(X_h;S_h,S_h) \geq \beta_{*}\kappa_{*}|S_h|_{1,\O}^2.
\end{equation}
These inequalities follow from very similar arguments to those that lead to \eqref{eq:ah_bound}.

\item $c_{N,h}^{\textit{skew}}(\cdot;\cdot,\cdot)$ is continuous: For every $\bz_h,\bv_h,\bw_h \in \bVh$, we have
\begin{equation}\label{eq:ch_bound}
c_{N,h}^{\textit{skew}}(\bz_h;\bv_h,\bw_h) \leq \widehat{C}_{N}|\bz_h|_{1,\O}|\bv_h|_{1,\O}|\bw_h|_{1,\O}.
\end{equation}
A proof of this estimate is essentially contained in \cite[Proposition 3.3]{MR3796371}.

\item $c_{F,h}(\cdot;\cdot,\cdot)$ is continuous: For every $\bz_h,\bv_h,\bw_h \in \bVh$, we have
\begin{equation}\label{eq:cFh_bound}
c_{F,h}(\bz_h;\bv_h,\bw_h) \leq \widehat{C}_{F}|\bz_h|_{1,\O}^{r-2}|\bv_h|_{1,\O}|\bw_h|_{1,\O}.
\end{equation}
A proof of this estimate can be found in \cite[Lemma 4.1]{NMTMA-17-210}.

\item $\mathfrak{c}_{h}^{\textit{skew}}(\cdot;\cdot,\cdot)$ is continuous: For every $\bv_h \in \bV_h$ and $R_h,S_h \in \bVV_h$, we have
\begin{equation}\label{eq:chtemp_bound}
\mathfrak{c}_h^{\textit{skew}}(\bv_h;R_h,S_h) \leq \widehat{\mathfrak{C}}|\bv_h|_{1,\O} |R_h |_{1,\O}|S_h|_{1,\O},
\end{equation}
This bound can be obtained with arguments similar to those in the proof of \cite[Proposition 3.3]{MR3796371}.

\item $d_h(\cdot,\cdot)$ is continuous: For every $\bv_h, \bw_h \in \bV_h$, we have
\begin{equation}\label{eq:dh_bound}
d_h(\bv_h,\bw_h) \leq \|\bv_h\|_{0,\O}\|\bw_h\|_{0,\O}.
\end{equation}
This bound follows directly from the Cauchy-Schwarz inequality in $\mathbf{L}^2(E)$, the $\mathbf{L}^2$-stability of $\bPi_k^{0,E}$, and the Cauchy-Schwarz inequality in $\mathbb{R}^{\# \CT_h}$.

\item $b(\cdot,\cdot)$ is continuous: For every $\textsf{q}_h \in \Q_h$ and $\bv_h \in \bV_h$, we have
\begin{equation}\label{eq:bh_bound}
b(\bv_h,\textsf{q}_h) \leq |\bv_h|_{1,\O}\|\textsf{q}_h\|_{0,\O}.
\end{equation} 

\end{itemize} 

\section{Existence and uniqueness of solutions for our VEM}
\label{sec:ex_uniq_sol_vem}
In this section, we derive an existence result for the discrete problem \eqref{def:weak_BFH_discrete} without restriction on the problem data by using a fixed point strategy. Moreover, we obtain a global uniqueness result when the problem data is suitably restricted.

\begin{theorem}[existence of discrete solutions]\label{eq:existence}
Under the data assumptions \textbf{A0}) and \textbf{A1}), the nonlinear system \eqref{def:weak_BFH_discrete} admits at least one solution $(\bu_h,\textsf{p}_h,T_h) \in \bVh\times\Q_h\times\bVVh$. Moreover, the following estimates hold uniformly in $h$:
\begin{equation}\label{exi1}
\begin{split}
|\bu_h|_{1,\O} &\leq \alpha_*^{-1}\nu_*^{-1} \textsf{C}\|\boldsymbol{f}_h\|_{0,\O},
\quad 
\|\textsf{p}_h\|_{0,\O} \leq \tilde{\beta}^{-1}\Gamma
\textsf{C}
\|\boldsymbol{f}_h\|_{0,\O}, \\
|T_h|_{1,\O} &\leq \beta_*^{-1}\kappa_*^{-1}\textsf{C}\|g_h\|_{0,\O},
\end{split}
\end{equation}
Here,
$\Gamma :=
1
+
(\alpha_*\nu_*)^{-1}(
\alpha^*\nu^*
+
\textsf{C}\widehat{C}_N\alpha_{*}^{-1}\nu_*^{-1}\|\boldsymbol{f}_h\|_{0,\O}
+
\textsf{C}^{r-2}\widehat{C}_{F}\alpha_*^{2-r}\nu_*^{2-r}\|\boldsymbol{f}_h\|_{0,\O}^{r-2}
+
\textsf{C}^2).
$
\label{thm:existence}
\end{theorem}
\begin{proof} We proceed in two steps.

\emph{Step 1. Existence of solutions:} Let us first analyze the existence of solutions to the reduced problem \eqref{def:weak_BFH_kernel_discrete}. To this purpose, we
define $\mathrm{A}_{(\bu_h,T_h)}: \bZh \times \bVVh \rightarrow \mathbb{R}$, for $(\bu_h,T_h) \in \bZh \times \bVVh$, as follows:
\begin{multline*}
\mathrm{A}_{(\bu_h,T_h)}(\bv_h,S_h) := a_h(T_h;\bu_h,\bv_h)+c_{N,h}^{\textit{skew}}(\bu_h;\bu_h,\bv_h) + c_{F,h}(\bu_h;\bu_h,\bv_h) + d_h(\bu_h,\bv_h)
\\
- (\boldsymbol{f}_h,\bv_h)_{0,\Omega} + \mathfrak{a}_h(T_h;T_h,S_h) + \mathfrak{c}_h^{\textit{skew}}(\bu_h;T_h,S_h) - (g_h,S_h)_{0,\Omega}.
\end{multline*}
Set $(\bv_h,S_h) = (\bu_h,T_h)$ and use the skew-symmetry of $c_{N,h}^{\textit{skew}}(\cdot;\cdot,\cdot)$ and $\mathfrak{c}_h^{\textit{skew}}(\cdot;\cdot,\cdot)$ and the coercivity properties in \eqref{eq:ah_bound} and \eqref{eq:ahtemp_bound} to obtain
\begin{multline*}
\mathrm{A}_{(\bu_h,T_h)}(\bu_h,T_h) \geq  \alpha_{*}\nu_*|\bu_h|_{1,\O}^2 + \beta_*\kappa_*|T_h|_{1,\O}^2 + c_{F,h}(\bu_h;\bu_h,\bu_h) + d_h(\bu_h,\bu_h) \\
- \|\boldsymbol{f}_h\|_{0,\O}\|\bu_h\|_{0,\O}-\|g_h\|_{0,\O}\|T_h\|_{0,\O},
\end{multline*}
Since $d_h(\bu_h,\bu_h)\geq 0$ and $c_{F,h}(\bu_h;\bu_h,\bu_h) \geq 0$, we thus obtain that 
\begin{multline*}
\mathrm{A}_{(\bu_h,T_h)}(\bu_h,T_h) \geq  \alpha_{*}\nu_*|\bu_h|_{1,\O}^2 + \beta_*\kappa_*|T_h|_{1,\O}^2 - \textsf{C}\left( \|\boldsymbol{f}_h\|_{0,\O}|\bu_h|_{1,\O} - \|g_h\|_{0,\O}|T_h|_{1,\O} \right)
\\
\geq 
(\min\{\alpha_*,\beta_*\}\min\{\nu_*,\kappa_*\})
|(\bu_h,T_h)|_{1,\O}^2 - \textsf{C} \|(\boldsymbol{f}_h,g_h)\|_{0,\O}
|(\bu_h,T_h)|_{1,\O},
\end{multline*}
where $|(\bv_h,S_h)|_{1,\O} := (|\bv_h|_{1,\O}^2 + |S_h|_{1,\O}^2)^{\frac{1}{2}}$ and $\|(\boldsymbol{f},g)\|_{0,\O} := (\|\boldsymbol{f}\|_{0,\O}^2 + \|g\|_{0,\O}^2)^{\frac{1}{2}}$ for every $(\bv_h,S_h) \in \bV_h \times \bVV_h$ and $(\boldsymbol{f},g) \in \mathbf{L}^{2}(\Omega) \times \L^2(\Omega)$. Hence, we obtain that $\mathrm{A}_{(\bu_h,T_h)}(\bu_h,T_h) \geq 0$ for every $(\bu_h,T_h) \in \bZh \times \bVVh$ such that
\begin{equation*}\label{condiexi1}
|(\bu_h,T_h)|_{1,\O} ^2 =\dfrac{\textsf{C}^2\|(\boldsymbol{f}_h,g_h)\|_{0,\O}^2}{(\min\{\alpha_*,\beta_*\}\min\{\nu_*,\kappa_*\})^2} =: \mu^2.
\end{equation*} 
With reference to \cite[Chap. IV, Corollary 1.1]{MR851383}, there thus exists $(\bu_h,T_h) \in \bZ_h \times \bVV_h$ such that $\mathrm{A}_{(\bu_h,T_h)}(\bv_h,S_h) = 0$ for all $(\bv_h,S_h) \in \bZ_h \times \bVV_h$ and $|(\bu_h,T_h)|_{1,\O} \leq \mu \lesssim \|(\boldsymbol{f}_h,g_h)\|_{0,\Omega}$; i.e., the nonlinear system \eqref{def:weak_BFH_kernel_discrete} has at least one solution $(\bu_h,T_h)\in\bZh\times\bVVh$. Finally, the existence of a solution for the system \eqref{def:weak_BFH_discrete} follows immediately from the inf-sup condition \eqref{eq:infsupdiscreta}.

\emph{Step 2. Stability bounds:} Let us first derive the bound for the velocity field $\bu_h$ in \eqref{exi1}. To to this, we set $(\bv_h,\textsf{q}_h)=(\bu_h,0)$ in \eqref{def:weak_BFH_discrete} and use the coercivity property in \eqref{eq:ah_bound}, the skew-symmetry of $c_{N,h}^{\textit{skew}}(\cdot;\cdot,\cdot)$, $d(\bu_h,\bu_h)\geq 0$, and $c_{F,h}(\bu_h;\bu_h,\bu_h)\geq 0$ to obtain the following bounds:
\begin{multline*}
\alpha_*\nu_*|\bu_h|_{1,\O}^2 \leq a_h(T_h;\bu_h,\bu_h)+c_{F,h}(\bu_h;\bu_h,\bu_h)+d_h(\bu_h,\bu_h) \\
= ( \boldsymbol{f}_h,\bu_h)_{0,\O}
\leq \textsf{C} \| \boldsymbol{f}_h \|_{0,\Omega} |\bu_h|_{1,\O}.
\end{multline*}
This immediately yields the the desire bound for the velocity field. The estimate for the temperature follows similar arguments. Finally, the estimate for the pressure can be obtained using \eqref{eq:infsupdiscreta}, \eqref{def:weak_BFH_discrete}, \eqref{eq:ah_bound}, \eqref{eq:ch_bound}, and \eqref{eq:cFh_bound}. This concludes the proof.
\end{proof}

In the following, we present two elementary results that are useful to show the uniqueness of the solutions of system \eqref{def:weak_BFH_discrete}.

\begin{lemma}[Lipschitz property]
\label{estimavel}
Let $X_h, S_h \in \bVVh$, and let $\bv_h \in \bVh$ such that $\bv_h \in \mtW_{q}^1(\CT_h)$ for some $q>2$. Then, there exists $C_{V}>0$ depending on $\O$, the polynomial degree $k$, and the shape regularity constant $\varrho$ such that, for any $\bw_h \in \bVh$,
\begin{equation*}
\left\lvert a_h(X_h;\bv_h,\bw_h) - a_h(S_h;\bv_h,\bw_h)\right\lvert \leq C_{V}\nu_{lip}|X_h-S_h|_{1,\O}|\bv_h|_{\mtW_q^1(\CT_h)}|\bw_h|_{1,\O}.
\end{equation*}
\end{lemma}

\begin{proof}
See \cite[Lemma 4.3]{AVV}.
\end{proof}

\begin{lemma}[Lipschitz property]
\label{estimatemp}
Let $X_h, S_h \in \bVVh$, and let $R_h \in \bVVh$ such that $R_h \in \W_{q}^1(\CT_h)$ for some $q>2$. Then, there exists $C_{T}>0$ depending on $\O$, the polynomial degree $k$, and the shape regularity constant $\varrho$ such that, for any $P_h \in \bVVh$,
\begin{equation*}
\left\lvert \mathfrak{a}_h(X_h;R_h,P_h) - \mathfrak{a}_h(S_h;R_h,P_h)\right\lvert \leq C_{T}\kappa_{lip}|X_h-S_h|_{1,\O}|R_h|_{\W_q^1(\CT_h)}|P_h|_{1,\O}.
\end{equation*}
\end{lemma}
\begin{proof}
The proof is similar to that of Lemma \ref{estimavel}, so we skip the details.
\end{proof}

We are now in a position to show the uniqueness of solutions for system \eqref{def:weak_BFH_discrete}.

\begin{theorem}[uniqueness of discrete solutions]\label{eq:uniqueness} 
Let us assume that assumptions \textbf{A0}) and \textbf{A1}) hold. If
\begin{align}
\label{eq:uniq}
\dfrac{C_{V}\nu_{lip}\widehat{\mathfrak{C}}\textsf{C}\|g_h\|_{0,\O} |\bu_h|_{\mtW_q^1(\CT_h)} }{\alpha_*\nu_*\beta_*\kappa_*(\beta_*\kappa_{*}-C_T\kappa_{lip}|T_h|_{\W_q^1(\CT_h)})}
+
\dfrac{\widehat{C}_{N}\textsf{C}}{\alpha_*^2\nu_*^2}\|\boldsymbol{f}_h\|_{0,\O} & < 1,
\\
\label{eq:uniq_2}
C_T\kappa_{lip}|T_h|_{\W_q^1(\CT_h)} &< \beta_*\kappa_{*},
\end{align}
and the nonlinear system \eqref{def:weak_BFH_discrete} admits a solution $(\bu_h,\textsf{p}_h,T_h) \in \bZh \times \Q_h \times \bVV_h$ such that $\bu_h \in \mtW_q^1(\CT_h)$ and $T_h \in \W_q^1(\CT_h)$ for some $q>2$, then this solution is unique.
\label{thm:uniqueness}
\end{theorem}

\begin{proof}
We begin the proof with the assumption that there is another solution $(\widehat{\bu}_h,\widehat{\textsf{p}}_h,\widehat{T}_h) \in \bZh\times\Q_h\times\bVVh$ of system \eqref{def:weak_BFH_discrete}. Define
\begin{equation*}
\overline{\bu}_h := \bu_h-\widehat{\bu}_h \in \bZh,
\qquad
\overline{T}_h := T_h-\widehat{T}_h \in \bVVh,
\qquad
\overline{\textsf{p}}_h := \textsf{p}_h-\widehat{\textsf{p}}_h \in \Q_h.
\end{equation*}

\emph{Step 1}. A bound for $|\overline{T}_h|_{1,\O}$. We begin by noting that it follows directly from the definition of $\mathfrak{c}_h^{\textit{skew}}(\cdot;\cdot,\cdot)$ that
\begin{equation}\label{eq:skew1}
\mathfrak{c}_h^{\textit{skew}}(\widehat{\bu}_h;\overline{T}_h,\overline{T}_h) = 0 \implies \mathfrak{c}_h^{\textit{skew}}(\widehat{\bu}_h;T_h,\overline{T}_h) = \mathfrak{c}_h^{\textit{skew}}(\widehat{\bu}_h;\widehat{T}_h,\overline{T}_h).
\end{equation}
We now use the fact that $(\bu_h,T_h)$ and $(\widehat{\bu}_h,\widehat{T}_h)$ solve the system \eqref{def:weak_BFH_kernel_discrete} to set $\bv_h = \mathbf{0}$ and $S_h = \overline{T}_h \in \bVV_h$ in the corresponding systems and obtain
\begin{equation*}
\mathfrak{a}_h(T_h;T_h,\overline{T}_h) + \mathfrak{c}_h^{\textit{skew}}(\bu_h;T_h,\overline{T}_h) = \mathfrak{a}_h(\widehat{T}_h;\widehat{T}_h,\overline{T}_h) + \mathfrak{c}_h^{\textit{skew}}(\widehat{\bu}_h;\widehat{T}_h,\overline{T}_h).
\end{equation*}
We add and subtract the term $\mathfrak{a}_h(\widehat{T}_h;T_h,\overline{T}_h)$ and use \eqref{eq:skew1} to obtain
\begin{equation*}
\mathfrak{a}_h(\widehat{T}_h;\overline{T}_h,\overline{T}_h) = \left[\mathfrak{a}_h(\widehat{T}_h;T_h,\overline{T}_h)-\mathfrak{a}_h(T_h;T_h,\overline{T}_h)\right] - \mathfrak{c}_h^{\textit{skew}}(\overline{\bu}_h;T_h,\overline{T}_h).
\end{equation*}
We then use the coercivity property in \eqref{eq:ahtemp_bound}, the bound \eqref{eq:chtemp_bound}, and the estimate of Lemma \ref{estimatemp} to obtain
\begin{equation}\label{temperatura}
\beta_{*}\kappa_{*}|\overline{T}_h|_{1,\O}^{2} \leq C_{T}\kappa_{lip}|\overline{T}_h|_{1,\O}^{2}|T_h|_{\W_q^1(\CT_h)} + \widehat{\mathfrak{C}}|\overline{\bu}_h|_{1,\O}|T_{h}|_{1,\O}|\overline{T}_h|_{1,\O}.
\end{equation}
This yields
$
(\beta_{*}\kappa_* - C_{T}\kappa_{lip}|T_h|_{\W_q^1(\CT_h)})|\overline{T}_h|_{1,\O} \leq \widehat{\mathfrak{C}}|\overline{\bu}_h|_{1,\O}|T_h|_{1,\O}.
$
Note that $\beta_{*}\kappa_* - C_{T}\kappa_{lip}|T_h|_{\W_q^1(\CT_h)}>0$. We now invoke the stability bound for $T_h$ in \eqref{exi1} to obtain
\begin{equation}\label{uniq1}
|\overline{T}_h|_{1,\O} \leq \dfrac{\widehat{\mathfrak{C}}\textsf{C}\|g_h\|_{0,\Omega}}{\beta_*\kappa_*(\beta_{*}\kappa_* - C_T\kappa_{lip}|T_h|_{\W_q^1(\CT_h)})}|\overline{\bu}_h|_{1,\O}.
\end{equation}

\emph{Step 2. We now obtain a bound for $\overline{\bu}_h$.} To do this, we first note that, due to the skew-symmetry of $c_{N,h}^{\textit{skew}}(\cdot;\cdot,\cdot)$, we have the following
\begin{equation}\label{eq:skew2}
c_{N,h}^{\textit{skew}}(\widehat{\bu}_h;\overline{\bu}_h,\overline{\bu}_h) = 0 
\implies c_{N,h}^{\textit{skew}}(\widehat{\bu}_h;\bu_h,\overline{\bu}_h) = c_{N,h}^{\textit{skew}}(\widehat{\bu}_h;\widehat{\bu}_h,\overline{\bu}_h).
\end{equation}
We now use the fact that $(\bu_h,T_h)$ and $(\widehat{\bu}_h,\widehat{T}_h)$ solve system \eqref{def:weak_BFH_kernel_discrete} to set $\bv_h = \overline{\bu}_h \in \bV_h$ in the first equation of the corresponding systems and obtain
\begin{multline*}
a_h(T_h;\bu_h,\overline{\bu}_h) + c_{N,h}^{\textit{skew}}(\bu_h;\bu_h,\overline{\bu}_h) + c_{F,h}(\bu_h;\bu_h,\overline{\bu}_h) + d_h(\bu_h,\overline{\bu}_h) \\
= a_h(\widehat{T}_h;\widehat{\bu}_h,\overline{\bu}_h) + c_{N,h}^{\textit{skew}}(\widehat{\bu}_h;\widehat{\bu}_h,\overline{\bu}_h) + c_{F,h}(\widehat{\bu}_h;\widehat{\bu}_h,\overline{\bu}_h) + d_h(\widehat{\bu}_h,\overline{\bu}_h).
\end{multline*}
We add and subtract the term $a_h(\widehat{T}_h;\bu_h,\overline{\bu}_h)$ and use \eqref{eq:skew2} to obtain
\begin{multline*}
a_h(\widehat{T}_h;\overline{\bu}_h,\overline{\bu}_h) + \left[c_{F,h}(\bu_h;\bu_h,\overline{\bu}_h) - c_{F,h}(\widehat{\bu}_h;\widehat{\bu}_h,\overline{\bu}_h)\right] + d_h(\overline{\bu}_h,\overline{\bu}_h)\\
 = \left[a_h(\widehat{T}_h;\bu_h,\overline{\bu}_h) - a_h(T_h;\bu_h,\overline{\bu}_h)\right] - c_{N,h}^{\textit{skew}}(\overline{\bu}_h;\bu_h,\overline{\bu}_h). 
\end{multline*}
We then use the left-hand side bound in \eqref{eq:ah_bound}, \cite[Chapter I, Lemma 4.4]{dibenedetto2012degenerate}, the bound in Lemma \ref{estimavel}, \eqref{eq:ch_bound}, and the fact that $d_h(\overline{\bu}_h,\overline{\bu}_h)\geq 0$ to obtain
\begin{equation}\label{uniq2}
\alpha_{*}\nu_*|\overline{\bu}_h|_{1,\O}^{2} \leq C_{V}\nu_{lip}|\overline{T}_h|_{1,\O}|\bu_h|_{\mtW_q^1(\CT_h)}|\overline{\bu}_h|_{1,\O} + \widehat{C}_{N}|\overline{\bu}_h|_{1,\O}^{2}|\bu_h|_{1,\O}.
\end{equation}
We now use the bound for $\bu_h$ in Theorem \ref{eq:existence} and \eqref{uniq1} to arrive at
\begin{align*}
\alpha_*\nu_{*}|\overline{\bu}_h|^2_{1,\O} &\leq \dfrac{C_{V}\nu_{lip}\widehat{\mathfrak{C}}\textsf{C}\|g_h\|_{0,\O}|\bu_h|_{\mtW_q^1(\CT_h)}}{\beta_*\kappa_*(\beta_*\kappa_{*}-C_T\kappa_{lip}|T_h|_{\W_q^1(\CT_h)})}|\overline{\bu}_h|^2_{1,\O} + \dfrac{\widehat{C}_{N}\textsf{C}}{\alpha_*\nu_*}\|\boldsymbol{f}_h\|_{0,\O}|\overline{\bu}_h|^2_{1,\O}.
\end{align*}
The previous inequality allows us to conclude that 
\begin{equation*}
\left(1-\dfrac{C_{V}\nu_{lip}\widehat{\mathfrak{C}}\textsf{C}\|g_h\|_{0,\O} |\bu_h|_{\mtW_q^1(\CT_h)} }{\alpha_*\nu_*\beta_*\kappa_*(\beta_*\kappa_{*}-C_T\kappa_{lip}|T_h|_{\W_q^1(\CT_h)})}
-
\dfrac{\widehat{C}_{N}\textsf{C}}{\alpha_*^2\nu_*^2}\|\boldsymbol{f}_h\|_{0,\O}\right)|\overline{\bu}_h|^2_{1,\O} \leq 0.
\end{equation*}

\emph{Step 3}. In view of \eqref{eq:uniq}, we immediately conclude that $\overline{\bu}_h = 0$. We now invoke \eqref{uniq1} to obtain that $\overline{T}_h = 0$. Finally, the inf-sup condition \eqref{eq:infsupdiscreta} and the arguments developed in Step 2 show that $\overline{\textsf{p}}_h = 0$. This concludes the proof.
\end{proof}


\section{Error estimates for our VEM}
\label{sec:error}
In the following, we derive error bounds for the formulation with virtual elements \eqref{def:weak_BFH_discrete}. For this purpose, we will make the following regularity assumptions:

\begin{itemize}
\item[\textbf{A4})] The solutions $(\bu,\textsf{p},T) \in \bV\times\Q\times\bVV$ and the data $\boldsymbol{f},g,\kappa,\nu$ of system \eqref{def:weak_BFH_discrete} satisfy the following regularity properties for some $0<s\leq k$:
\begin{itemize}
\item[i)] $\bu \in \mathbf{H}^{s+1}(\O)$, $\textsf{p} \in \H^s(\O)$, and $T\in \H^{s+1}(\O)$.
\item[ii)] $\boldsymbol{f} \in \mathbf{H}^{s+1}(\O)$ and $g \in \H^{s+1}(\O)$.
\item[iii)] $\nu(T)$ and $\kappa(T)$ belong to $\W^{s}_{\infty}(\O)$.
\end{itemize}
\end{itemize}

Before deriving a priori error estimates, we recall some preliminary approximation properties.

\begin{lemma}[an approximation property for $\bVh$]
\label{eq:interpolantvel}
Assume that \textbf{A2}) and \textbf{A3}) hold. Let $\bv \in \bV \cap \mathbf{H}^{s+1}(\O)$. Then, there exists $\bv_{I} \in \bVh$ such that
\begin{equation*}
\|\bv - \bv_{I}\|_{0,E} + h|\bv - \bv_{I}|_{1,E} \lesssim h_E^{s+1}|\bv|_{s+1,D(E)}.
\end{equation*}
for all $E \in \CT_h$. Here, $D(E)$ denote the union of the polygons in $\CT_h$ intersecting $E$ and $0<s\leq k$. Moreover, if $\bv\in\bZ$, then $\bv_I \in \bZh$.
\end{lemma}
\begin{proof}
A proof of the error bound follows from the arguments in the proof of \cite[Theorem 4.1]{MR3796371} in combination with the error bound of \cite[Proposition 4.2]{MR3626409}.
\end{proof}

\begin{lemma}[an approximation property for $\bVVh$]
\label{eq:interpolanttemp}
Assume that \textbf{A2)} and \textbf{A3)} hold. Let $S \in \bVV \cap \H^{s+1}(\O)$. Then, there exists $S_{I} \in \bVVh$ such that
\begin{equation*}
\|S - S_{I}\|_{0,E} + h|S - S_{I}|_{1,E} \lesssim h_E^{s+1}|S|_{s+1,D(E)}.
\end{equation*}
for all $E \in \CT_h$. Here, $D(E)$ denote the union of the polygons in $\CT_h$ intersecting $E$ and $0<s\leq k$.
\end{lemma}
\begin{proof}
A proof of this error bound can be found in \cite[Theorem 11]{CETS}.
\end{proof}

We will also make use of a Bramble--Hilbert Lemma \cite[Lemma 4.3.8]{MR2373954}: Let $0 \leq t \leq s \leq \ell+1$ and $1 \leq q,p \leq \infty$ such that $s-2/p > t-2/q$. Then, for any $E\in\CT_h$,
\begin{equation}\label{eq:Bramblehilbert}
\underset{q_{\ell} \in \mathbb{P}_\ell(E) }{\inf} |S - q_{\ell} |_{W_{q}^{t}(E)} \lesssim h_E^{s-t+2/q-2/p}|S|_{W_{p}^s(E)} \qquad \forall S \in W_p^s(E).
\end{equation}



To simplify the presentation of the material, we will write the continuous and discrete stability bounds for $(\bu,T)$ and $(\bu_h,T_h)$ as in \cite{AVV}. This is
\begin{equation}
\label{eq:stability_bounds_continuous_discrete}
|\bu|^2_{1,\O}+|T|^2_{1,\O} \leq C^2_{\mathrm{est}}C^{-2}_{\mathrm{data}},
\qquad
|\bu_h|^2_{1,\O}+|T_h|^2_{1,\O} \leq \widehat{C}^2_{\mathrm{est}}\widehat{C}^{-2}_{\mathrm{data}}.
\end{equation}
As in \cite{AVV},  we will also use the following notation here: We denote by $\mathfrak{A}(\cdot)$, $\mathfrak{B}(\cdot)$, $\mathfrak{C}(\cdot)$, $\mathfrak{D}(\cdot)$, etc. generic constants that are independent of $h$, but may depend on appropriate norms of $\bu$, $p$, $T$, $\nu$, $\kappa$, $\boldsymbol{f}$, $g$, $\Omega$, or the polynomial degree $k$.

We are now ready to state and prove the main result of this section.

\begin{theorem}[error estimates]
Let us assume that the assumptions \textbf{A0}), \textbf{A1}), \textbf{A2}), \textbf{A3}), and \textbf{A4}) hold. Let us also assume that the smallness assumptions of Proposition \ref{pro:existence3} and Theorem \ref{eq:uniqueness} hold. Let $(\bu,\textsf{p},T) \in \bV \times \Q \times \bVV$ and $(\bu_h,\textsf{p}_h,T_h) \in \bV_h\times\Q_h\times\bVV_h$ be the unique solutions of problems \eqref{def:weak_BFH} and \eqref{def:weak_BFH_discrete}, respectively. If
\begin{equation*}
\label{eq:condicionerror}
\begin{split}
C_{\mathrm{final},T} := \frac{\widehat{\frak{C}}C_{\mathrm{est}}}{C_{\mathrm{data}}}
\left(
\beta_{*}\kappa_{*} - \frac{C_{\mathrm{lip}}(T)}{C_{\mathrm{sol}}(T)}
\right
)^{-1} > 0,\\
C_{\mathrm{final},\bu}^{-1}:= \left( \alpha_*\nu_* - \frac{\widehat{C}_N C_{\mathrm{est}}}{C_{\mathrm{data}}} - C_{\mathrm{for}} -  \textsf{C}^2 - \frac{C_{\mathrm{lip}}(\bu)}{C_{\mathrm{sol}}(\bu)}C_{\mathrm{final},T} \right) > 0,
\end{split}
\end{equation*}
then the following a priori error estimates hold
\begin{equation*}
\begin{aligned}
|\bu-\bu_h|_{1,\O}+|T-T_h|_{1,\O}
& \leq
\mathfrak{M}(\nu,\kappa,\bu,T)h^s
+
\mathfrak{N}(\boldsymbol{f},g)h^{s+2}
\\
\| \textsf{p}- \textsf{p}_h\|_{0,\O}
& \leq
\mathfrak{O}(\nu,\kappa,\bu,T,p)h^s
+
\mathfrak{N}(\boldsymbol{f},g)h^{s+2}.
\end{aligned}
\end{equation*}
\label{thm:error_estimates}
\end{theorem}

\begin{proof}
We follow the proof of \cite[Proposition 5.3]{AVV} and proceed in several steps.

\emph{Step 1.} \emph{Interpolation error estimates.}
We introduce $\bu_I \in \bV_h$ and $T_I \in \bVV_h$ as the interpolants of $\bu \in \bV$ and $T \in \bVV$ given by Lemmas \ref{eq:interpolantvel} and \ref{eq:interpolanttemp}, respectively. We also introduce $ \textsf{p}_I \in \Q_h$ as follows: for each $E \in \CT_h$, $ \textsf{p}_I|_{E} := \Pi^{0,E}_{k-1} \textsf{p}$. As a direct consequence of Lemmas \ref{eq:interpolantvel} and \ref{eq:interpolanttemp} and the Bramble-Hilbert bound \eqref{eq:Bramblehilbert}, we obtain
\begin{equation}\label{eq:int1}
|\bu-\bu_I|_{1,\O} \lesssim h^{s}|\bu|_{s+1,\O}, \quad |T-T_I|_{1,\O} \lesssim h^{s}|T|_{s+1,\O}, \quad \| \textsf{p}- \textsf{p}_I\|_{0,\O} \lesssim h^{s}| \textsf{p}|_{s,\O}.
\end{equation}

Define $\boldsymbol{e}_h := \bu_I-\bu_h$, $E_h := T_I-T_h$, and $\textsf{e}_h :=  \textsf{p}_I- \textsf{p}_h$. Since $\bu \in \bZ$, we have that $\bu_I \in \bZ_h$ and thus that $\boldsymbol{e}_h \in \bZ_h$. In the following, we bound $\boldsymbol{e}_h$, $E_h$, and $\textsf{e}_h$.

\emph{Step 2. An estimate for $E_h$.} We start with the coercivity bound in \eqref{eq:ahtemp_bound}, the definition of $E_h$, namely $E_h := T_I-T_h$, and add and subtract $\mathfrak{a}(T;T,E_h)$ to obtain
\begin{equation}
 \begin{aligned}
\label{eqtemp1}
\beta_{*}\kappa_{*}|E_h|_{1,\O}^{2} & \leq \mathfrak{a}_h(T_h;E_h,E_h) = \mathfrak{a}_h(T_h;T_I,E_h) - \mathfrak{a}_h(T_h;T_h,E_h)
\\
& = \mathfrak{a}_h(T_h;T_I,E_h) - \mathfrak{a}(T;T,E_h) + \mathfrak{a}(T;T,E_h) - \mathfrak{a}_h(T_h;T_h,E_h).
\end{aligned}
\end{equation}
We now use the third equations of the continuous and discrete systems, \eqref{def:weak_BFH} and \eqref{def:weak_BFH_discrete}, respectively, to arrive at the following estimate:
\begin{multline}
 \beta_{*}\kappa_{*}|E_h|_{1,\O}^{2} \leq \left[\mathfrak{a}_h(T_h;T_I,E_h) - \mathfrak{a}(T;T,E_h)\right]
\\
+
\left[ \mathfrak{c}_h^{\textit{skew}}(\bu_h;T_h,E_h) - \mathfrak{c}(\bu;T,E_h)\right]
+ ( g-g_h,E_h)_{0,\Omega} =: \I + \I\I + \I\I\I.
\label{eq:I+II+III}
\end{multline}

\emph{Step 2.1.} We estimate $\I$. To do so, we first rewrite $\I$ using the definitions of $\mathfrak{a}$ and $\mathfrak{a}_h$, given in \eqref{eq:bilinear_form_fraka} and \eqref{eq:bilinear_form_fraka_h}, respectively, and a localization argument and obtain
\begin{multline*}
\I
=
\displaystyle
\sum_{E\in\CT_h} \bigg\{ \int_{E} \kappa(\Pi^{0,E}_k T_h)
(\bPi^{0,E}_{k-1} \nabla T_I - \nabla T)
\cdot
\bPi^{0,E}_{k-1}\nabla E_h
-  \int_{E}\kappa(T)\nabla T\cdot\nabla E_h
\\
+ \int_{E} \kappa(\Pi^{0,E}_{k}T_h)\nabla T \cdot \bPi^{0,E}_{k-1}\nabla E_h
+
\displaystyle
\kappa(\Pi^{0,E}_0 T_h) S_T^{E}((\I-\Pi^{0,E}_k)T_I,(\I-\Pi^{0,E}_k) E_h)\bigg\}.
\end{multline*}
We construct further differences as follows:
\begin{equation*}
\begin{aligned}
\I
& = \displaystyle
\sum_{E\in\CT_h}
\bigg\{
\int_{E} \kappa(\Pi^{0,E}_k T_h)(\bPi^{0,E}_{k-1} \nabla T_I - \nabla T)\cdot \bPi^{0,E}_{k-1}\nabla E_h
\\
& \quad + \displaystyle  \int_E
\left(
\kappa(\Pi^{0,E}_k T_h) - \kappa(T)
\right)\nabla T
\cdot \bPi^{0,E}_{k-1} \nabla E_h
-
\displaystyle  \int_E \kappa(T)\nabla T \cdot ( \nabla E_h - \bPi^{0,E}_{k-1} \nabla E_h)
\\
& \quad
+
\displaystyle \kappa(\Pi^{0,E}_0 T_h) S_T^{E}((\I-\Pi^{0,E}_k)T_I,(\I-\Pi^{0,E}_k) E_h)
\bigg\}
=: \sum_{E\in\CT_h} ( \I_1^E + \I_2^E - \I_3^E + \I_4^E).
\end{aligned}
\end{equation*}
Using the definition of $\bPi^{0,E}_{k-1}$, $\I_3^E$ can be rewritten as $\I_3^E = \int_E (\mathbf{I} - \bPi^{0,E}_{k-1})(\kappa(T)\nabla T) \cdot \nabla E_h$. As in \cite[page 3417]{AVV}, the terms $\I_1^E$, $\I_3^E$, and $\I_4^E$ can be controlled simultaneously under the assumption \textbf{A0}) and \eqref{eq:S_T}. In fact, we have
\begin{multline}
\label{eq:aux_temp_0}
 \sum_{E\in\CT_h}  (\I_1^E - \I_3^E + \I_4^E)
 \leq C
 \sum_{E\in\CT_h}
 \left( \kappa^{*}
 \| \bPi^{0,E}_{k-1} \nabla T_I - \nabla T \|_{0,E}
 \right.
 \\
 \left.
 +
 \| (\mathbf{I} - \bPi^{0,E}_{k-1})(\kappa(T)\nabla T) \|_{0,E}
 +
 \kappa^{*} \| \nabla (\I - \Pi^{0,E}_{k})T_I \|_{0,E}
 \right)
 \|  \nabla E_h \|_{0,E},
\end{multline}
where we used $\| \bPi^{0,E}_{k-1} \nabla E_h \|_{0,E} \leq \| \nabla E_h \|_{0,E}$ and $\| \nabla ( \I  - \Pi^{0,E}_{k}) E_h \|_{0,E} \lesssim \| \nabla E_h \|_{0,E}$. We now use the triangle inequality, Lemma \ref{eq:interpolanttemp}, and a basic error bound for the $\mathbf{L}^2(E)$-projection (see, e.g. \cite[Theorem 7]{CETS}) to obtain
\begin{equation}
 \| \bPi^{0,E}_{k-1} \nabla T_I - \nabla T \|_{0,E}
 \leq
 \| \nabla (T - T_I) \|_{0,E}
 +
 \| (\mathbf{I} - \bPi^{0,E}_{k-1}) \nabla T \|_{0,E}
 \lesssim h_E^s |T|_{1+s,D(E)}.
 \label{eq:aux_temp_1}
\end{equation}
An estimate for $\| \nabla (\I - \Pi^{0,E}_{k})T_I \|_{0,E}$ can be derived in view of similar arguments:
\begin{multline}
 \| \nabla (\I - \Pi^{0,E}_{k})T_I \|_{0,E}
 \leq
 \| \nabla (\I - \Pi^{0,E}_{k}) (T_I-T)  \|_{0,E}
 +
 \| \nabla (\I - \Pi^{0,E}_{k}) T \|_{0,E}
 \\
 \lesssim
 \| \nabla (T_I-T)  \|_{0,E}
 +
\| \nabla (\I - \Pi^{0,E}_{k}) T \|_{0,E}
\lesssim h_E^s |T|_{1+s,D(E)}.
\label{eq:aux_temp_2}
\end{multline}
Finally, we bound $\| (\mathbf{I} - \bPi^{0,E}_{k-1})(\kappa(T)\nabla T) \|_{0,E}$ as follows:
\begin{equation}
\| (\mathbf{I} - \bPi^{0,E}_{k-1})(\kappa(T)\nabla T) \|_{0,E}
\lesssim
h_E^s | \kappa(T)\nabla T |_{s,E}
\lesssim
h_E^s \| \kappa(T) \|_{W_{\infty}^s(E)} | T |_{s+1,E}.
\label{eq:aux_temp_3}
\end{equation}
If we substitute the estimates \eqref{eq:aux_temp_1}, \eqref{eq:aux_temp_2}, and \eqref{eq:aux_temp_3} into \eqref{eq:aux_temp_0} we arrive at the bound
\begin{equation}
\label{eq:aux_temp_4}
 \sum_{E\in\CT_h}  (\I_1^E - \I_3^E + \I_4^E)
 \leq \mathfrak{A}(\kappa,T) h^s \|  \nabla E_h \|_{0,\Omega},
 \end{equation}
 where $\mathfrak{A} = \mathfrak{A}(\kappa,T)$ is a constant that depends on $\kappa$ and $T$.

 We now turn to the derivation of a bound for $\I_2^E$. For this purpose, we invoke H\"{o}lder's inequality ($1/p + 1/q +1/2 =1 $, where $q = 2 + \varepsilon >2$ is given by the uniqueness assumptions of Proposition \ref{pro:existence3}) and the stability of the $\L^2(E)$-projection in $\L^2(E)$ and $\L^p(E)$ to obtain
 \begin{multline}
\I_2^E
\leq
\kappa_{lip} \|T-\Pi^{0,E}_{k}T_h\|_{\L^p(E)}|T|_{\W_q^1(E)}\|\nabla E_h\|_{0,E}
\leq
\kappa_{lip} ( \|T- \Pi^{0,E}_{k}T \|_{\L^p(E)}
\\
+ \| \Pi^{0,E}_{k}(T - T_h) \|_{\L^p(E)})|T|_{\W_q^1(E)}\|\nabla E_h\|_{0,E}
\lesssim
\kappa_{lip}
( \|T- \Pi^{0,E}_{k}T \|_{\L^p(E)}
\\
+ \| T - T_h \|_{\L^p(E)})|T|_{\W_q^1(E)}\|\nabla E_h\|_{0,E}.
\label{eq:aux_temp_5}
\end{multline}
A standard error bound yields $\|T- \Pi^{0,E}_{k}T \|_{\L^p(E)} \lesssim h_E^s |T|_{\W_s^p(E)}$. On the other hand, we have $\| T - T_h \|_{\L^p(E)} \leq \| T - T_I \|_{\L^p(E)} + \| T_I - T_h \|_{\L^p(E)}$. Substituting these estimates into \eqref{eq:aux_temp_5}, we obtain
\[
\I_2^E \lesssim
\kappa_{lip}
(
h_E^s |T|_{\W_s^p(E)} + \| T - T_I \|_{\L^p(E)} + \| E_h \|_{\L^p(E)}
)|T|_{\W_q^1(E)}\|\nabla E_h\|_{0,E},
\]
where we also used that $E_h = T_I - T_h$. If we sum over all the elements in $E$ in $\CT_h$ and apply H\"{o}lder's inequality ($1/p + 1/q +1/2 =1$), we obtain
\begin{multline*}
\I_2 := \sum_{E\in\CT_h} \I_2^E \lesssim
\kappa_{lip}
(
h^s |T|_{\W_s^p(\Omega)} +  h^s| T |_{s+1,\Omega}+ \| E_h \|_{\L^p(\Omega)}
)|T|_{\W_q^1(\Omega)}\|\nabla E_h\|_{0,\Omega}
\end{multline*}
where we have used that $\| T - T_I \|_{\L^p(\Omega)} \lesssim \| \nabla(T - T_I) \|_{0,\Omega} \lesssim h^s  | T |_{s+1,\Omega}$, which follows from a basic Sobolev embedding and the error bound in Lemma \ref{eq:interpolanttemp}. As a result,
\begin{equation}
\I_2 \leq \frac{C_{\mathrm{lip}}(T)}{C_{\mathrm{sol}}(T)}\|\nabla E_h\|^2_{0,\Omega} + \mathfrak{A}(\kappa,T)h^s\|\nabla E_h\|_{0,\Omega}.
\label{eq:aux_temp_6}
\end{equation}
$C_{\mathrm{sol}}(T)$ comes from Prop. \ref{pro:existence3}: $|T|_{\W_q^1(\Omega)} < C_{\mathrm{sol}}(T)^{-1} := (\kappa_{lip} \mathcal{C}_q/\kappa_{*})^{-1}$,
and $C_{\mathrm{lip}}(T)>0$ is a suitable constant. In a final step, we combine \eqref{eq:aux_temp_4} with \eqref{eq:aux_temp_6} and obtain
\begin{equation}
\label{eq:IT}
\I
=
\sum_{E\in\CT_h}  (\I_1^E + \I_2^E - \I_3^E + \I_4^E) \leq
\frac{C_{\mathrm{lip}}(T)}{C_{\mathrm{sol}}(T)}\|\nabla E_h\|^2_{0,\Omega} + \mathfrak{A}(\kappa,T)h^s\|\nabla E_h\|_{0,\Omega}.
\end{equation}

\emph{Step 2.2.} Let us now control $\I\I$ in \eqref{eq:I+II+III}. As a first step, we note that according to Remark \ref{rem:skew-symmetry}, $\mathfrak{c}(\bu,T;E_h) = \mathfrak{c}^{skew}(\bu,T;E_h)$ because $\bu \in \bZ$. We use this property, the fact that $\mathfrak{c}_{h}^{\textit{skew}}(\bu_h;T_h,E_h) = \mathfrak{c}_{h}^{\textit{skew}}(\bu_h;T_h -T_I,E_h) + \mathfrak{c}_{h}^{\textit{skew}}(\bu_h;T_I,E_h) = \mathfrak{c}_{h}^{\textit{skew}}(\bu_h;T_I,E_h)$, and add subtract $\mathfrak{c}_{h}^{\textit{skew}}(\bu;T,E_h)$ to rewrite $\I\I$ as follows:
\begin{equation*}
 \I \I
 = \mathfrak{c}_{h}^{\textit{skew}}(\bu_h;T_I - T,E_h)
 +
 \mathfrak{c}_{h}^{\textit{skew}}(\bu_h - \bu;T,E_h)
 +
 \mathfrak{c}_{h}^{\textit{skew}}(\bu;T,E_h)
 -
 \mathfrak{c}^{\textit{skew}}(\bu;T,E_h).
\end{equation*}
Define $\I\I_1:= \mathfrak{c}_{h}^{\textit{skew}}(\bu_h;T_I - T,E_h)
 +
 \mathfrak{c}_{h}^{\textit{skew}}(\bu_h - \bu;T,E_h)$. To bound $\I\I_1$, we first use the estimate \eqref{eq:chtemp_bound} and obtain
\begin{equation*}
\textrm{II}_1
\leq \widehat{\mathfrak{C}}
\left(
|\bu_h|_{1,\O}|T-T_I|_{1,\O} + |\bu-\bu_h|_{1,\O}|T|_{1,\O}
\right)
|E_h|_{1,\O}.
\end{equation*}
We now use the interpolation error bounds of Lemmas \ref{eq:interpolantvel} and \ref{eq:interpolanttemp} to arrive at
\begin{equation}
\label{eq:II_1}
\textrm{II}_1
\leq
\frac{\widehat{\mathfrak{C}} C_{\mathrm{est}}}{C_{\mathrm{data}}}
|\boldsymbol{e}_h|_{1,\O}|E_h|_{1,\O}
+
\mathfrak{B}(\boldsymbol{u},T)h^s|E_h|_{1,\O},
\end{equation}
where we also used \eqref{eq:stability_bounds_continuous_discrete}. Define $\I\I_2^E:= \mathfrak{c}_h^{\textit{skew},E}(\bu;T,E_h)-\mathfrak{c}^{\textit{skew},E}(\bu;T,E_h)$, for $E \in \CT_h$, and $\I\I_2:= \mathfrak{c}_{h}^{\textit{skew}}(\bu;T,E_h)
 - \mathfrak{c}^{\textit{skew}}(\bu;T,E_h)$. We note that
\begin{multline}
\I\I_2^{E}
=
\frac{1}{2} \left[
\int_E ( \bPi^{0,E}_k \bu \cdot \bPi^{0,E}_{k-1}\nabla T ) \Pi^{0,E}_k E_h
-
\int_E  ( \bu \cdot \nabla T ) E_h
\right]
\\
-
\frac{1}{2}
\left[
\int_E ( \bPi^{0,E}_k \bu \cdot \bPi^{0,E}_{k-1} \nabla E_h ) \Pi^{0,E}_k T
-
\int_E  ( \bu \cdot \nabla E_h ) T
\right]
=:
\frac{1}{2} \I\I_{2,a}^{E} - \frac{1}{2} \I\I_{2,b}^{E}.
\end{multline}
The control of $\I\I_{2,a}^{E}$ and $\I\I_{2,b}^{E}$ follow from the arguments given in the proof of \cite[Lemma 4.3]{MR3796371}. If we sum the obtained bounds over all elements in $E$ in $\CT_h$ and apply a suitable H\"{o}lder's inequality, we arrive at
\begin{equation}
\textrm{II}_2 \leq \mathfrak{B}(\bu,T)h^s|E_h|_{1,\O}.
\label{eq:II_2}
\end{equation}
A collection of the bounds \eqref{eq:II_1} and \eqref{eq:II_2} shows that
\begin{equation}
\label{eq:IIT}
\textrm{II}
\leq
\frac{\widehat{\mathfrak{C}} C_{\mathrm{est}}}{C_{\mathrm{data}}}
|\boldsymbol{e}_h|_{1,\O}|E_h|_{1,\O}
+
\mathfrak{B}(\boldsymbol{u},T)h^s|E_h|_{1,\O}
\end{equation}

\emph{Step 2.3.} We have now arrived at the estimation of term $\I\I\I$. The bound for this term follows from the arguments in \cite[Section 4.7]{MR2997471}:
\begin{equation}\label{eq:IIIT}
\I\I\I
\lesssim
\mathfrak{G}(g) h^{s+2}|E_h|_{1,\O},
\end{equation}

\emph{Step 2.4.} Finally, if we substitute the bounds for $\I$, $\I\I$, and $\I\I\I$ obtained in \eqref{eq:IT}, \eqref{eq:IIT} and \eqref{eq:IIIT}, respectively, into \eqref{eq:I+II+III} we obtain
\begin{equation}\label{eq:semifinal_estimate_T}
|E_h|_{1,\O} \leq
C_{\mathrm{final},T}
|\boldsymbol{e}_h|_{1,\O}
+
\mathfrak{C}(\kappa,\bu,T)h^s
+
\mathfrak{G}(g) h^{s+2}.
\end{equation}

\emph{Step 3. An estimate for $\boldsymbol{e}_h$.} We start with the coercivity bound in \eqref{eq:ah_bound} for $a_h(\cdot;\cdot,\cdot)$, the definition of $\boldsymbol{e}_h$, namely $\boldsymbol{e}_h = \boldsymbol{u}_I - \boldsymbol{u}_h$, and add and subtract the term $a(T;\bu,\boldsymbol{e}_h)$ to obtain
\begin{equation}
\begin{aligned}
\alpha_*\nu_*|\boldsymbol{e}_h|_{1,\O}^2
&
\leq a_h(T_h;\boldsymbol{e}_h,\boldsymbol{e}_h) = a_h(T_h;\bu_I,\boldsymbol{e}_h)-a_h(T_h;\bu_h,\boldsymbol{e}_h)
\\
&
=
a_h(T_h;\bu_I,\boldsymbol{e}_h)
-
a(T;\bu,\boldsymbol{e}_h)
+
a(T;\bu,\boldsymbol{e}_h)
-
a_h(T_h;\bu_h,\boldsymbol{e}_h).
\end{aligned}
\end{equation}
We now use the first equations of the continuous and discrete systems, \eqref{def:weak_BFH} and \eqref{def:weak_BFH_discrete}, respectively, to obtain
\begin{multline*}
\alpha_*\nu_*|\boldsymbol{e}_h|_{1,\O}^2
\leq
[
a_h(T_h;\bu_I,\boldsymbol{e}_h)
-
a(T;\bu,\boldsymbol{e}_h)
]
+
[
c_{N,h}^{\textit{skew}}(\bu_h;\bu_h,\boldsymbol{e}_h)
-
c_N(\bu;\bu,\boldsymbol{e}_h)
]
\\
+
[
c_{F,h}(\bu_h;\bu_h,\boldsymbol{e}_h)
-
c_{F}(\bu;\bu,\boldsymbol{e}_h)
]
+
[d_h(\bu_h,\boldsymbol{e}_h)-d(\bu,\boldsymbol{e}_h)]
+
(\boldsymbol{f}-\boldsymbol{f}_h,\boldsymbol{e}_h)_{0,\Omega},
\end{multline*}
where we have also used that $\boldsymbol{e}_h = \boldsymbol{u}_I - \boldsymbol{u}_h \in \bZ_h \subseteq \mathbf{Z}$; see Remark \ref{rem:divergence_free}.

\emph{Step 3.1.} Define $\mathfrak{I}_1 := [
a_h(T_h;\bu_I,\boldsymbol{e}_h)
-
a(T;\bu,\boldsymbol{e}_h)
]$. The control of $\mathfrak{I}_1$ follows from \cite[estimate (5.16)]{AVV}:
\begin{equation}
\label{eq:mathfrakI_1}
\mathfrak{I}_1
\leq
\frac{C_{\mathrm{lip}}(\bu)}{C_{\mathrm{sol}}(\bu)} |E_h|_{1,\O}|\boldsymbol{e}_h|_{1,\O}
+
\mathfrak{D}(\nu,\bu,T)h^s|\boldsymbol{e}_h|_{1,\O}.
\end{equation}
Here, the constant $C_{\mathrm{sol}}(\bu)$ comes from Prop. \ref{pro:existence3}: $C_{\mathrm{sol}}(\bu)|\boldsymbol{u}|_{\mathbf{W}_q^1(\Omega)} < 1$, and $C_{\mathrm{lip}}(\bu)$ denotes a suitable positive constant.

\emph{Step 3.2.} Define $\mathfrak{I}_2:= c_{N,h}^{\textit{skew}}(\bu_h;\bu_h,\boldsymbol{e}_h) - c_N(\bu;\bu,\boldsymbol{e}_h)$.
Note that $c^{skew}_N(\bu;\bu,\boldsymbol{e}_h) = c_N(\bu;\bu,\boldsymbol{e}_h)$ because $\bu \in \bZ$; see Remark \ref{rem:skew-symmetry}. So we rewrite the term $\mathfrak{I}_2$ as follows:
\begin{equation*}
 \mathfrak{I}_2 =
 [c_{N,h}^{\textit{skew}}(\bu_h;\bu_h,\boldsymbol{e}_h)
 -
 c_{N,h}^{\textit{skew}}(\bu;\bu,\boldsymbol{e}_h)]
 +
 [c_{N,h}^{\textit{skew}}(\bu;\bu,\boldsymbol{e}_h)
 -
 c^{skew}_N(\bu;\bu,\boldsymbol{e}_h)].
\end{equation*}
Since $\bu \in \bV \cap \mathbf{H}^{s+1}(\Omega)$ and $\boldsymbol{e}_h \in \bV$, we can apply  \cite[Lemma 4.3]{MR3796371} to obtain
\[
  \mathfrak{I}_{2,b}
 :=
 c_{N,h}^{\textit{skew}}(\bu;\bu,\boldsymbol{e}_h)
 -
 c^{skew}_N(\bu;\bu,\boldsymbol{e}_h) \lesssim \mathfrak{E}(\bu)h^s|\boldsymbol{e}_h|_{1,\O}.
\]
Define $\mathfrak{I}_{2,a}:= c_{N,h}^{\textit{skew}}(\bu_h;\bu_h,\boldsymbol{e}_h)
 -
 c_{N,h}^{\textit{skew}}(\bu;\bu,\boldsymbol{e}_h)$. Apply \cite[Lemma 4.4]{MR3796371} to obtain
\begin{equation*}
 \mathfrak{I}_{2,a}
 \leq
 \frac{\widehat{C}_N C_{\mathrm{est}}}{C_{\mathrm{data}}}|\boldsymbol{e}_h|^2_{1,\O}
 +
 \widehat{C}_N
 \left(
 \frac{C_{\mathrm{est}}}{C_{\mathrm{data}}}
 +
 \frac{\widehat{C}_{\mathrm{est}}}{\widehat{C}_{\mathrm{data}}}
 \right)
 \mathfrak{E}(\bu)h^s|\boldsymbol{e}_h|_{1,\O}.
\end{equation*}
If we combine the previously derived bounds, we arrive at
\begin{equation}
\mathfrak{I}_2
\leq
\frac{\widehat{C}_N C_{\mathrm{est}}}{C_{\mathrm{data}}}
|\boldsymbol{e}_h|^2_{1,\O}
 +
 \mathfrak{E}(\bu)h^s|\boldsymbol{e}_h|_{1,\O}.
\label{eq:mathfrakI_2}
\end{equation}

\emph{Step 3.3.} Define $\mathfrak{I}_3:= c_{F,h}(\bu_h;\bu_h,\boldsymbol{e}_h)
-
c_{F}(\bu;\bu,\boldsymbol{e}_h)$. We add and subtract $c_{F,h}(\bu;\bu,\boldsymbol{e}_h)$ to obtain $\mathfrak{I}_3 = [c_{F,h}(\bu_h;\bu_h,\boldsymbol{e}_h)- c_{F,h}(\bu;\bu,\boldsymbol{e}_h)] + [c_{F,h}(\bu;\bu,\boldsymbol{e}_h) - c_{F}(\bu;\bu,\boldsymbol{e}_h)] =: \mathfrak{I}_{3,a} + \mathfrak{I}_{3,b}$. We bound $\mathfrak{I}_{3,a}$ with the help of \cite[Lemma 5.3]{NMTMA-17-210}:
\begin{equation}
\label{eq:I_3a}
\mathfrak{I}_{3,a}
\leq
C_{\textrm{for}}\left(\mathfrak{E}(\bu)h^s|\boldsymbol{e}_h|_{1,\O} + |\boldsymbol{e}_h|^2_{1,\O} \right),
\end{equation}
where $C_{\textrm{for}} = C [ (C_{\textrm{est}}/{C_{\textrm{data}}})^{r-2} +  (\widehat{C}_{\textrm{est}}/{\widehat{C}_{\textrm{data}}})^{r-2}]$ and $C$ is the constant in \cite[Lemma 5.3]{NMTMA-17-210}. Since $\boldsymbol{u} \in \bV \cap \mathbf{H}^{s+1}(\Omega)$, a bound for $\mathfrak{I}_{3,b}$ follows from \cite[Lemma 5.2]{NMTMA-17-210}:
\begin{equation}
\label{eq:I_3b}
\mathfrak{I}_{3,b}
\leq
Ch^s (|\boldsymbol{u}|_{s+1,\Omega} + |\boldsymbol{u}|_{s,\Omega}) |\boldsymbol{u}|^{r-2}_{1,\O}|\boldsymbol{e}_h|_{1,\O}
\leq
\mathfrak{E}(\bu)h^s|\boldsymbol{e}_h|_{1,\O}.
\end{equation}
If we combine the bounds \eqref{eq:I_3a} and \eqref{eq:I_3b}, we obtain
\begin{equation}
\label{eq:mathfrakI_3}
\mathfrak{I}_{3} \leq C_{\textrm{for}} |\boldsymbol{e}_h|^2_{1,\O} + \mathfrak{E}(\bu)h^s|\boldsymbol{e}_h|_{1,\O}.
\end{equation}

\emph{Step 3.4.} Define $\mathfrak{I}_4 := d_h(\bu_h,\boldsymbol{e}_h)-d(\bu,\boldsymbol{e}_h)$. The control of $\mathfrak{I}_4$ is standard. To derive an estimate, we first analyze the local term $\mathfrak{I}_4^E := d_h^E(\bu_h,\boldsymbol{e}_h)-d^E(\bu,\boldsymbol{e}_h)$. In view of $(\bPi^{0,E}_k\bu_h, \bPi^{0,E}_k\boldsymbol{e}_h - \boldsymbol{e}_h)_{0,E} = 0$, we rewrite $\mathfrak{I}_4^E$ as follows:
\begin{multline}
\mathfrak{I}_4^E
 =
 (\bPi^{0,E}_k\bu_h - \bu, \bPi^{0,E}_k\boldsymbol{e}_h)_{0,E}
 +
 (\bu, \bPi^{0,E}_k\boldsymbol{e}_h - \boldsymbol{e}_h)_{0,E}
 \\
 =
 (\bPi^{0,E}_k\bu_h - \bu, \bPi^{0,E}_k\boldsymbol{e}_h)_{0,E}
 +
 (\bu - \bPi^{0,E}_k\bu_h, \bPi^{0,E}_k\boldsymbol{e}_h - \boldsymbol{e}_h)_{0,E}
 =
 (\bPi^{0,E}_k\bu_h - \bu,\boldsymbol{e}_h)_{0,E}.
 \label{eq:I_4E}
\end{multline}
As a result,
$
\mathfrak{I}_4^E
\leq
\|\bPi^{0,E}_k\bu_h - \bu \|_{0,E} \|\boldsymbol{e}_h \|_{0,E}.
$
We control $\|\bPi^{0,E}_k\bu_h - \bu \|_{0,E}$ as follows:
\begin{multline}
\|\bPi^{0,E}_k\bu_h - \bu \|_{0,E}
\leq
\|\bu_h - \bu \|_{0,E} + \|\bPi^{0,E}_k\bu - \bu \|_{0,E}
\\
\leq
\|\bu - \bu_I \|_{0,E} + \|\boldsymbol{e}_h \|_{0,E}
+
\|\bPi^{0,E}_k\bu - \bu \|_{0,E}.
\end{multline}
A basic bound for $\|\bPi^{0,E}_k\bu - \bu \|_{0,E}$ and an application of Lemma \ref{eq:interpolantvel} thus show that
\begin{equation*}
\mathfrak{I}_4^E
\
\leq
\|\boldsymbol{e}_h \|^2_{0,E}
+
Ch_E^{s+1} |\boldsymbol{u} |_{s+1,D(E)} \|\boldsymbol{e}_h \|_{0,E}.
\end{equation*}
If we sum over all the elements in $E$ in $\CT_h$ and apply H\"{o}lder's inequality for sums, we obtain an estimate for $ \mathfrak{I}_4 $:
\begin{equation}
\mathfrak{I}_4
\leq
\textsf{C}^2 |\boldsymbol{e}_h|_{1,\Omega}^2
 +
\mathfrak{E}(\bu)h^{s+1}|\boldsymbol{e}_h|_{1,\Omega},
\label{eq:mathfrakI_4}
\end{equation}
where we have also used the Poincar\'e inequality \eqref{eq:Poincare}.

\emph{Step 3.5.} Define $\mathfrak{I}_5 := (\boldsymbol{f}-\boldsymbol{f}_h,\boldsymbol{e}_h)_{0,\Omega}$. An estimate for $\mathfrak{I}_5$ is direct:
\begin{equation}
 \label{eq:mathfrakI_5}
\mathfrak{I}_5
\lesssim \mathfrak{F}(\boldsymbol{f}) h^{s+2} | \boldsymbol{e}_h |_{1,\Omega}.
\end{equation}

\emph{Step 3.6.} \emph{A final estimate for $|\boldsymbol{e}_h|_{1,\Omega}$.} Replace the estimates \eqref{eq:mathfrakI_1}, \eqref{eq:mathfrakI_2}, \eqref{eq:mathfrakI_3}, \eqref{eq:mathfrakI_4}, and \eqref{eq:mathfrakI_5} into the bounded derived for $|\boldsymbol{e}_h|_{1,\O}^2$ to obtain
\begin{multline*}
\left( \alpha_*\nu_* - \frac{\widehat{C}_N C_{\mathrm{est}}}{C_{\mathrm{data}}} - C_{\textrm{for}} - \textsf{C}^2 \right)
|\boldsymbol{e}_h|^2_{1,\O}
\leq
\frac{C_{\mathrm{lip}}(\bu)}{C_{\mathrm{sol}}(\bu)} |E_h|_{1,\O}|\boldsymbol{e}_h|_{1,\O}
\\
+
\mathfrak{D}(\nu,\bu,T) h^{s}|\boldsymbol{e}_h|_{1,\O} +
\mathfrak{F}(\boldsymbol{f})h^{s+2} |\boldsymbol{e}_h|_{1,\O},
\end{multline*}
We now replace the bound \eqref{eq:semifinal_estimate_T} for $|E_h|_{1,\O}$ into the previous estimate to obtain
\begin{equation}
|\boldsymbol{e}_h|_{1,\O} \leq C_{\mathrm{final},\bu} \left( \mathfrak{M}(\nu,\kappa,\bu,T) h^{s} + \mathfrak{N}(\boldsymbol{f},g)h^{s+2} \right).
 \label{eq:final_estimate_u}
\end{equation}

\emph{Step 4.} \emph{A final estimate for $|E_h|_{1,\Omega}$.} We replace \eqref{eq:final_estimate_u} into the bound \eqref{eq:semifinal_estimate_T} to finally arrive at
\begin{multline}
|E_h|_{1,\Omega} \leq C_{\mathrm{final},T}C_{\mathrm{final},u}\left( \mathfrak{M}(\nu,\kappa,\bu,T) h^{s} + \mathfrak{N}(\boldsymbol{f},g)h^{s+2} \right)
\\
+
\mathfrak{C}(\kappa,\bu,T)h^s
+
\mathfrak{G}(g) h^{s+2}
\leq  \mathfrak{M}(\nu,\kappa,\bu,T) h^{s} + \mathfrak{N}(\boldsymbol{f},g)h^{s+2}.
\label{eq:final_estimate_T}
\end{multline}

\emph{Step 5}. \emph{An estimate for $\textsf{e}_h$.} In this last step, we derive an error estimate for the pressure error $\|\textsf{e}_h\|_{L^2(\Omega)}$. To do so, we let $\bv_h \in \bV_h$ and use the definition of $\textsf{e}_h$, namely $\textsf{e}_h = \textsf{p}_I - \textsf{p}_h$ and the first equations of the continuous and discrete systems \eqref{def:weak_BFH} and \eqref{def:weak_BFH_discrete}, respectively, to obtain
\begin{multline*}
b(\bv_h,\textsf{e}_h)=b(\bv_h, \textsf{p}_I)-b(\bv_h, \textsf{p}_h)
=
b(\bv_h, \textsf{p}_I -\textsf{p})
+
b(\bv_h, \textsf{p})-b(\bv_h, \textsf{p}_h)
=
b(\bv_h, \textsf{p}_I- \textsf{p})
\\
+
\left[a_h(T_h;\bu_h,\bv_h)-a(T;\bu,\bv_h)\right]
+
\left[c_{N,h}^{\textit{skew}}(\bu_h;\bu_h,\bv_h)-c_{N}^{\textit{skew}}(\bu;\bu,\bv_h)\right]
\\
+
\left[c_{F,h}(\bu_h;\bu_h,\bv_h)-c_{F}(\bu;\bu,\bv_h)\right]
+
\left[d_h(\bu_h,\bv_h)-d(\bu,\bv_h)\right]
+
(\boldsymbol{f}-\boldsymbol{f}_h),\bv_h)_{0,\O}.
\end{multline*}

\emph{Step 5.1}. Define $\mathfrak{K}_1:=\left[a_h(T_h;\bu_h,\bv_h)-a(T;\bu,\bv_h)\right]$. The control of the term $\mathfrak{K}_1$ can be found in \cite[page 3419]{AVV}:
\begin{multline}
\mathfrak{K}_1 \lesssim |\bu-\bu_h|_{1,\O}|\bv_h|_{1,\O}
+
|T-T_h|_{1,\O}|\bv_h|_{1,\O}
+
\mathfrak{D}(\nu,\boldsymbol{u},T)h^s|\bv_h|_{1,\O}
\\
\leq
\mathfrak{M}(\nu,\kappa,\bu,T)h^{s}|\bv_h|_{1,\O}
+
\mathfrak{N}(\boldsymbol{f},g)h^{s+2}|\bv_h|_{1,\O},
\label{eq:K_1}
\end{multline}
where we have used \eqref{eq:final_estimate_u} and \eqref{eq:final_estimate_T} to obtain the last bound.

\emph{Step 5.2.} Define $\mathfrak{K}_2:= c_{N,h}^{\textit{skew}}(\bu_h;\bu_h,\bv_h)-c_{N}^{\textit{skew}}(\bu;\bu,\bv_h)$. An estimate for the term $\mathfrak{K}_2$ can be obtained as follows. First, we rewrite $\mathfrak{K}_2$ as
\begin{multline*}
\mathfrak{K}_2
=
[
c_{N,h}^{\textit{skew}}(\bu_h;\bu_h,\bv_h)
-
c_{N,h}^{\textit{skew}}(\bu;\bu,\bv_h)
]
+
[
c_{N,h}^{\textit{skew}}(\bu;\bu,\bv_h)
-
c_{N}^{\textit{skew}}(\bu;\bu,\bv_h)]
\\
=
[
c_{N,h}^{\textit{skew}}(\bu_h - \bu;\bu_h,\bv_h)
-
c_{N,h}^{\textit{skew}}(\bu;\bu - \bu_h,\bv_h)
]
+
[
c_{N,h}^{\textit{skew}}(\bu;\bu,\bv_h)
-
c_{N}^{\textit{skew}}(\bu;\bu,\bv_h)].
\end{multline*}
Define
$\mathfrak{K}_{2,a}:=
c_{N,h}^{\textit{skew}}(\bu_h - \bu;\bu_h,\bv_h)
-
c_{N,h}^{\textit{skew}}(\bu;\bu - \bu_h,\bv_h)
$. A bound for $\mathfrak{K}_{2,a}$ follows from the use of the bound \eqref{eq:ch_bound} and the error estimate \eqref{eq:final_estimate_u}. In fact, we have
\begin{equation}
\begin{aligned}
\mathfrak{K}_{2,a}
& \leq
\widehat{C}_N( |\bu_h|_{1,\O} + |\bu|_{1,\O} )
|\bu - \bu_h|_{1,\O}
|\bv_h|_{1,\O}
\\
& \leq
\mathfrak{M}(\nu,\kappa,\bu,T)h^{s}|\bv_h|_{1,\O}
+
\mathfrak{N}(\boldsymbol{f},g)h^{s+2}|\bv_h|_{1,\O}.
\end{aligned}
\label{eq:K_2a}
\end{equation}
It thus suffices to bound $\mathfrak{K}_{2,b}:= c_{N,h}^{\textit{skew}}(\bu;\bu,\bv_h) - c_{N}^{\textit{skew}}(\bu;\bu,\bv_h)$. Since $\bu \in \mathbf{H}^{s+1}(\Omega) \cap \bV$, we can use the bound in \cite[Lemma 4.3]{MR3796371} and obtain
$
\mathfrak{K}_{2,b}
\leq \mathfrak{E}(\bu)h^s|\bv_h|_{1,\O}.
$
This bound and \eqref{eq:K_2a} yield the control of $\mathfrak{K}_{2}:$
\begin{equation}
 \mathfrak{K}_{2} \leq \mathfrak{M}(\nu,\kappa,\bu,T)h^{s}|\bv_h|_{1,\O}
+
\mathfrak{N}(\boldsymbol{f},g)h^{s+2}|\bv_h|_{1,\O}.
\label{eq:K_2}
\end{equation}

\emph{Step 5.3.} Define $\mathfrak{K}_3:= c_{F,h}(\bu_h;\bu_h,\bv_h)-c_{F}(\bu;\bu,\bv_h)$. We add and subtract the term $c_{F,h}(\bu;\bu,\bv_h)$ and rewrite $\mathfrak{K}_3$ as follows:
\begin{equation*}
\mathfrak{K}_3 = [c_{F,h}(\bu_h;\bu_h,\bv_h)-c_{F,h}(\bu;\bu,\bv_h)]
+
[c_{F,h}(\bu;\bu,\bv_h)-c_{F}(\bu;\bu,\bv_h)].
\end{equation*}
Define $\mathfrak{K}_{3,a} = c_{F,h}(\bu_h;\bu_h,\bv_h)-c_{F,h}(\bu;\bu,\bv_h)$. A direct application of the bound in \cite[Lemma 5.3]{NMTMA-17-210} shows that
\begin{multline*}
\mathfrak{K}_{3,a}
\leq
C(|\bu_h|_{1,\O}^{r-2} + |\bu|_{1,\O}^{r-2})|\bu-\bu_h|_{1,\O}|\bv_h|_{1,\O}
\leq
C_{\mathrm{for}}|\bu-\bu_h|_{1,\O}|\bv_h|_{1,\O}
\\
\leq
\mathfrak{M}(\nu,\kappa,\bu,T)h^{s}|\bv_h|_{1,\O}
+
\mathfrak{N}(\boldsymbol{f},g)h^{s+2}|\bv_h|_{1,\O}.
\end{multline*}
Define $\mathfrak{K}_{3,b} :=c_{F,h}(\bu;\bu,\bv_h)-c_{F}(\bu;\bu,\bv_h)$. Since $\bu \in \bV \cap \mathbf{H}^{s+1}(\Omega)$, we can directly apply the bound from \cite[Lemma 5.2]{NMTMA-17-210} to obtain
\begin{equation*}
\mathfrak{K}_{3,b} \leq C h^{s}(|\bu|_{s,\O}+|\bu|_{1+s,\O})|\bu|_{1,\O}^{r-2}|\bv_h|_{1,\O}
\leq
\mathfrak{E}(\bu) h^{s} |\bv_h|_{1,\O}.
\end{equation*}
The two estimates derived above allow us to control the term $\mathfrak{K}_3$:
\begin{equation}
\mathfrak{K}_3
\leq
\mathfrak{M}(\nu,\kappa,\bu,T)h^{s}|\bv_h|_{1,\O}
+
\mathfrak{N}(\boldsymbol{f},g)h^{s+2}|\bv_h|_{1,\O}.
\label{eq:K_3}
\end{equation}

\emph{Step 5.4.} Define $\mathfrak{K}_4 := d_h(\bu_h,\bv_h)-d(\bu,\bv_h)$ and $\mathfrak{K}_4^E := d_h^E(\bu_h,\bv_h)-d^E(\bu,\bv_h)$ for $E \in \CT_h$. Following the arguments that lead to \eqref{eq:I_4E}, we deduce that
\begin{equation*}
\begin{aligned}
\mathfrak{K}_4^E =
(\bPi^{0,E}_k\bu_h - \bu,\boldsymbol{v}_h)_{0,E}
&
\leq
\textsf{C}
\| \bPi^{0,E}_k\bu_h - \bu \|_{0,E} |\boldsymbol{v}_h|_{1,E}
\\
&
\leq
\textsf{C}
(
\| \bu_h - \bu \|_{0,E}
+
\| \bPi^{0,E}_k \bu - \bu \|_{0,E}
)
|\boldsymbol{v}_h|_{1,E}.
\end{aligned}
\end{equation*}
If we sum over all the elements $E \in \CT_h$ and apply H\"{o}lder's inequality, we obtain
\begin{equation}
\begin{aligned}
\mathfrak{K}_4
& \lesssim
 (
| \bu_h - \bu |_{1,\Omega}
+
\| \bPi^{0,E}_k \bu - \bu \|_{0,\Omega}
)
|\boldsymbol{v}_h|_{1,\Omega}
\\
& \leq
\mathfrak{M}(\nu,\kappa,\bu,T)h^{s}|\bv_h|_{1,\O}
+
\mathfrak{N}(\boldsymbol{f},g)h^{s+2}|\bv_h|_{1,\O}.
\end{aligned}
\label{eq:K_4}
\end{equation}

\emph{Step 5.5.} Define $\mathfrak{K}_5 := (\boldsymbol{f}-\boldsymbol{f}_h,\bv_h)_{0,\O}$ and $\mathfrak{K}_6 := b(\bv_h, \textsf{p}_I - \textsf{p})$. The term $\mathfrak{K}_5$ was already estimated in \eqref{eq:mathfrakI_5}. A bound for the term $\mathfrak{K}_6$ follows from the definition of $\textsf{p}_I$ and the Bramble--Hilbert bound \eqref{eq:Bramblehilbert}: $\mathfrak{K}_6 \lesssim \mathfrak{P}(\textsf{p}) h^s |\bv_h|_{1,\O}$.

\emph{Step 5.6.} In view of the discrete inf-sup condition \eqref{eq:infsupdiscreta}, the bounds \eqref{eq:K_1}, \eqref{eq:K_2}, \eqref{eq:K_3}, and \eqref{eq:K_4}, we conclude
\begin{equation}
\| \textsf{p} - \textsf{p}_h\|_{0,\O} \leq \mathfrak{O}(\nu,\kappa,\bu,T,\textsf{p})h^{s}
+
\mathfrak{N}(\boldsymbol{f},g)h^{s+2}.
\end{equation}
This concludes the proof.
\end{proof}

\section{Numerical experiments}
\label{sec:numericos}

In this section, we report some numerical experiments to evaluate the performance of the proposed VEM. Our goal is to compute the experimental convergence rates in the norms used in the theoretical analysis. The results of this section were obtained using a MATLAB code with $k=2$, where the nonlinear problem \eqref{def:weak_BFH_discrete} was solved using the fixed-point iteration described in \textbf{Algorithm 1}. The refinement parameter $N$ used to characterize each mesh is the number of elements on each edge of $\Omega$. In Figure \ref{FIG:meshes} we present examples of the meshes we will use for our tests.

\subsection{Fixed-point iteration}
Let us describe the fixed-point iteration used to solve the coupled problem \eqref{def:weak_BFH_discrete}. \\

\noindent \textbf{Algorithm 1: Fixed-point iteration}

\noindent\rule{13cm}{0.4pt} \\
\noindent Input: Initial mesh $\CT_h$, initial guess $(\bu_h^0,\textsf{p}_h^0,T_h^0)\in \bV_h \times\Q_h\times\bVV_h$, the coefficients $\nu(\cdot)$ and $\kappa(\cdot)$, $\boldsymbol{f}_h \in [\L^2(\O)]^2$, $g_h\in\L^{2}(\O)$, $r\in[3,4]$, and $\textsf{tol}=10^{-6}$.
\begin{enumerate}
\item[1:] For $n\geq0$, find $(\bu_h^{n+1},\textsf{p}_{h}^{n+1})\in\bV_h\times\Q_h$ such that
\begin{equation*}
\begin{split}
a_h(T_h^n;\bu_h^{n+1},\bv_h) + c_{N,h}^{\textit{skew}}(\bu_h^n;\bu_h^{n+1},\bv_h)
+
c_{F,h}(\bu_h^n;\bu_h^{n+1},\bv_h) \\
+
d_h(\bu_h^{n+1},\bv_h)
+
b(\bv_h,\textsf{p}_h^{n+1})&= (\boldsymbol{f}_h,\bv_h)_{0,\O},
\\
b(\bu_h^{n+1},\textsf{q}_h) &= 0,
\end{split}
\end{equation*}
for every $(\bv_h,\textsf{q}_h)\in\bV_h\times\Q_h$. Then, $T_h^{n+1} \in \bVV_h$ is found as the solution of
\begin{equation*}
\mathfrak{a}_h(T_h^{n};T_h^{n+1},S_h)+\mathfrak{c}_h^{\textit{skew}}(\bu_h^{n+1};T_h^{n+1},S_h)=(g_h,S_h)_{0,\O},
\end{equation*}
for every $S_h\in\bVV_h$.
\item[2:] If $|(\bu_h^{n+1},\textsf{p}_h^{n+1},T_h^{n+1})-(\bu_h^{n},\textsf{p}_h^n,T_h^n)|>\textsf{tol}$, set $n\leftarrow n+1$ and go to step 1. Otherwise, return $(\bu_h,\textsf{p}_h,T_h)=(\bu_h^{n+1},\textsf{p}_h^{n+1},T_h^{n+1})$. Here, $|\cdot|$ denotes the Euclidean norm.
\end{enumerate}
\noindent\rule{13cm}{0.4pt}

\begin{figure}[H]
	\begin{center}
		\begin{minipage}{13cm}
			\centering\includegraphics[height=3.7cm, width=3.7cm]{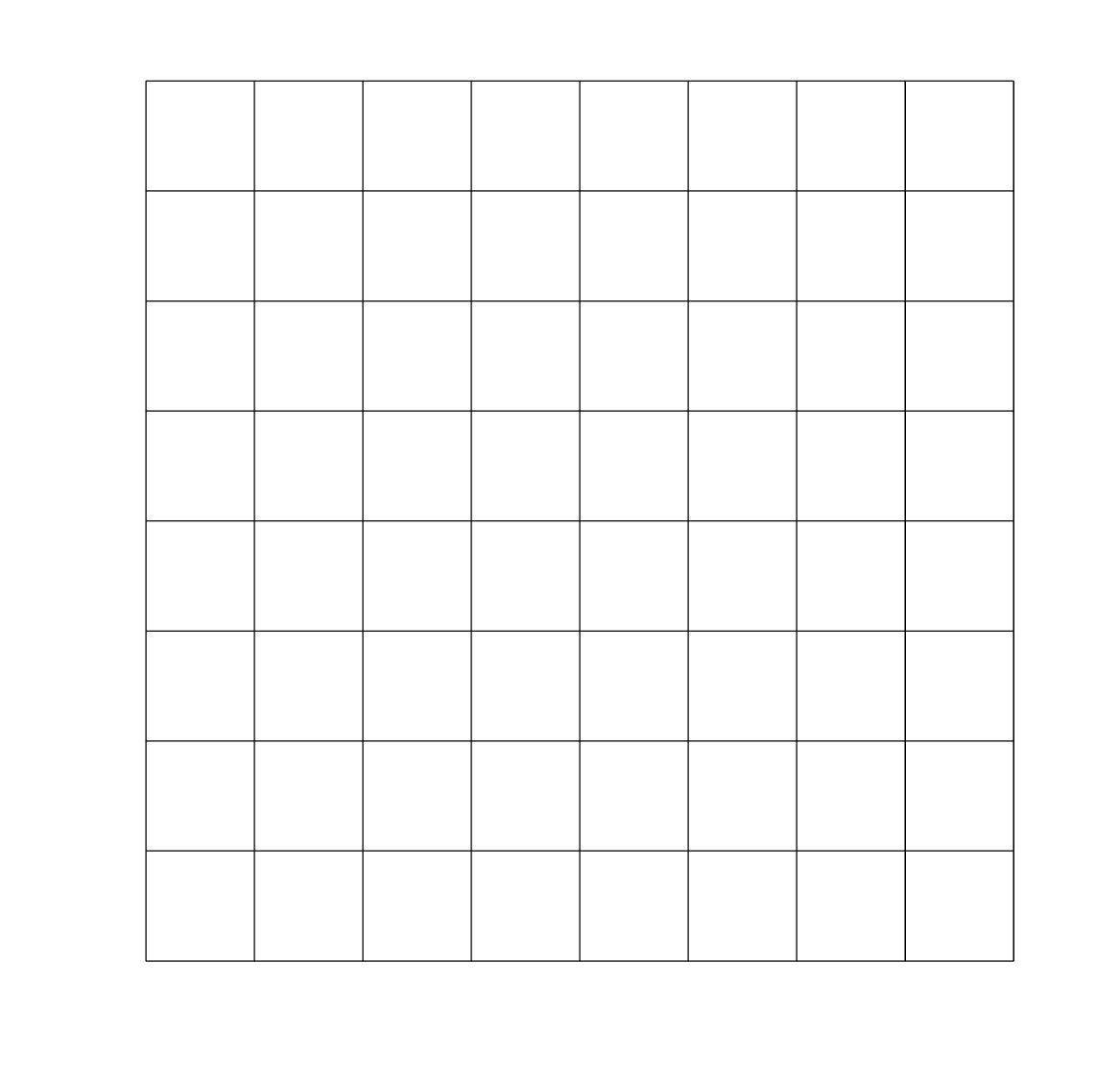}
                          \centering\includegraphics[height=3.7cm, width=3.7cm]{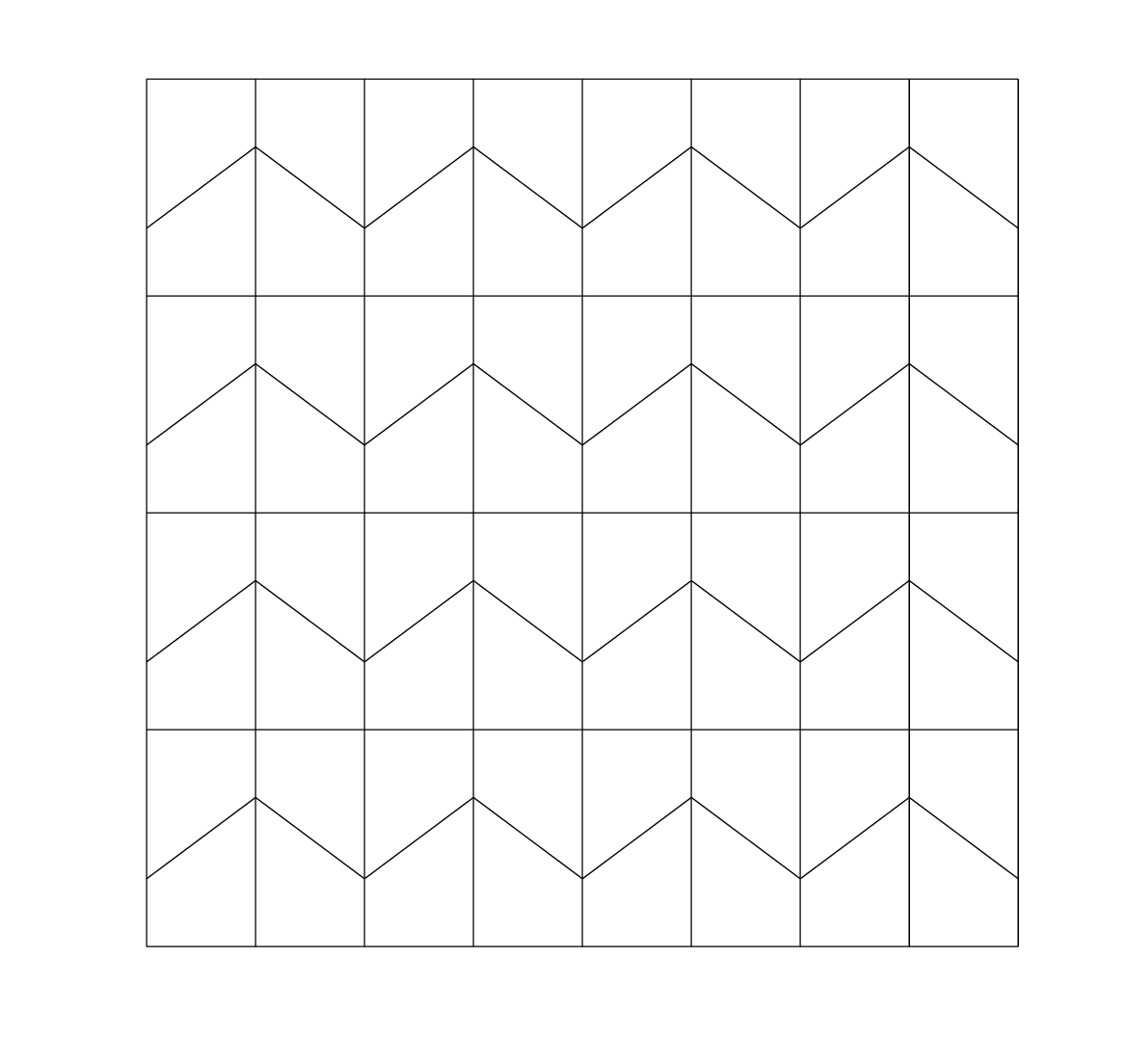}
                         \centering\includegraphics[height=3.7cm, width=3.7cm]{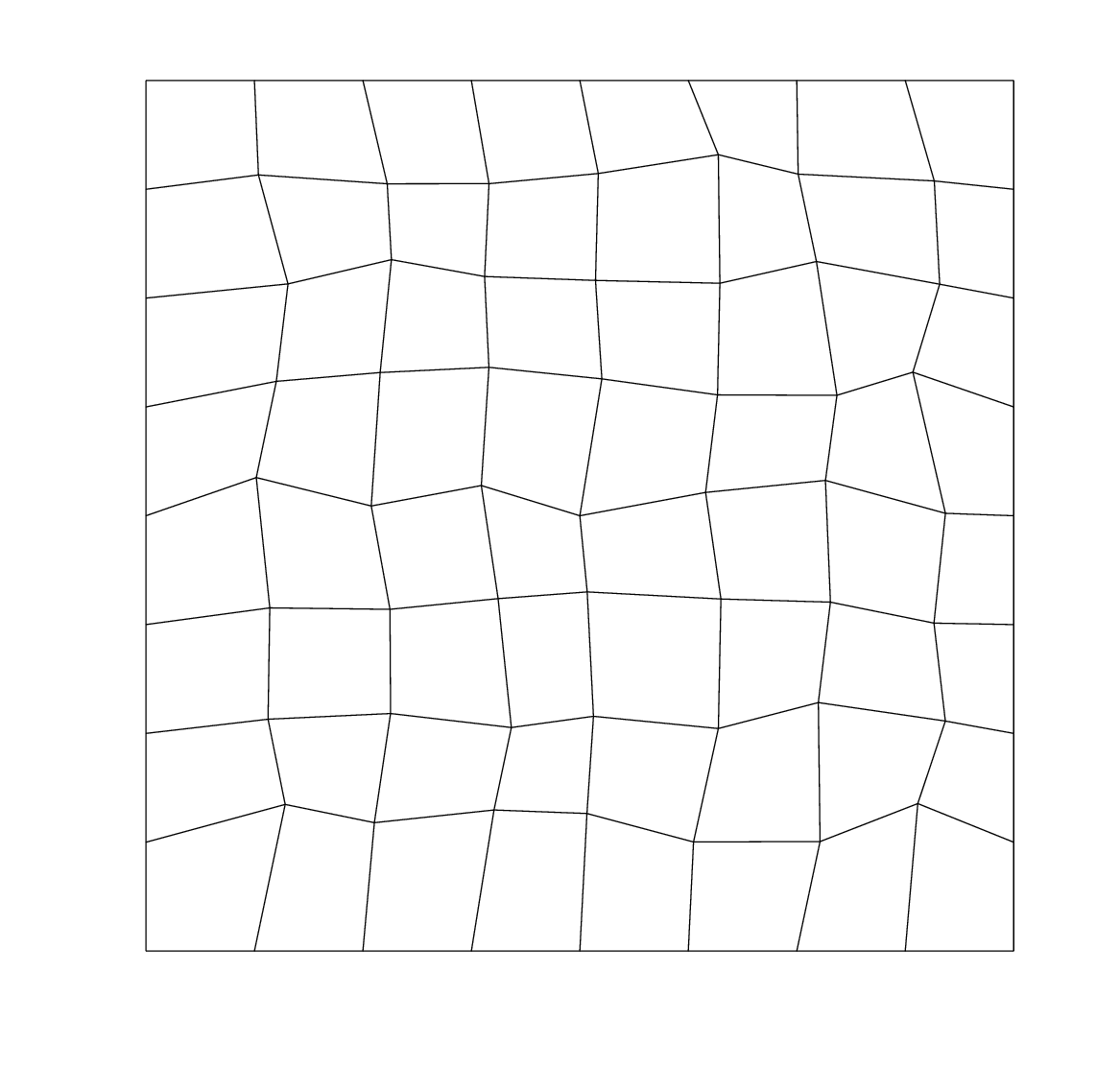}\\
                          \centering\includegraphics[height=3.7cm, width=3.7cm]{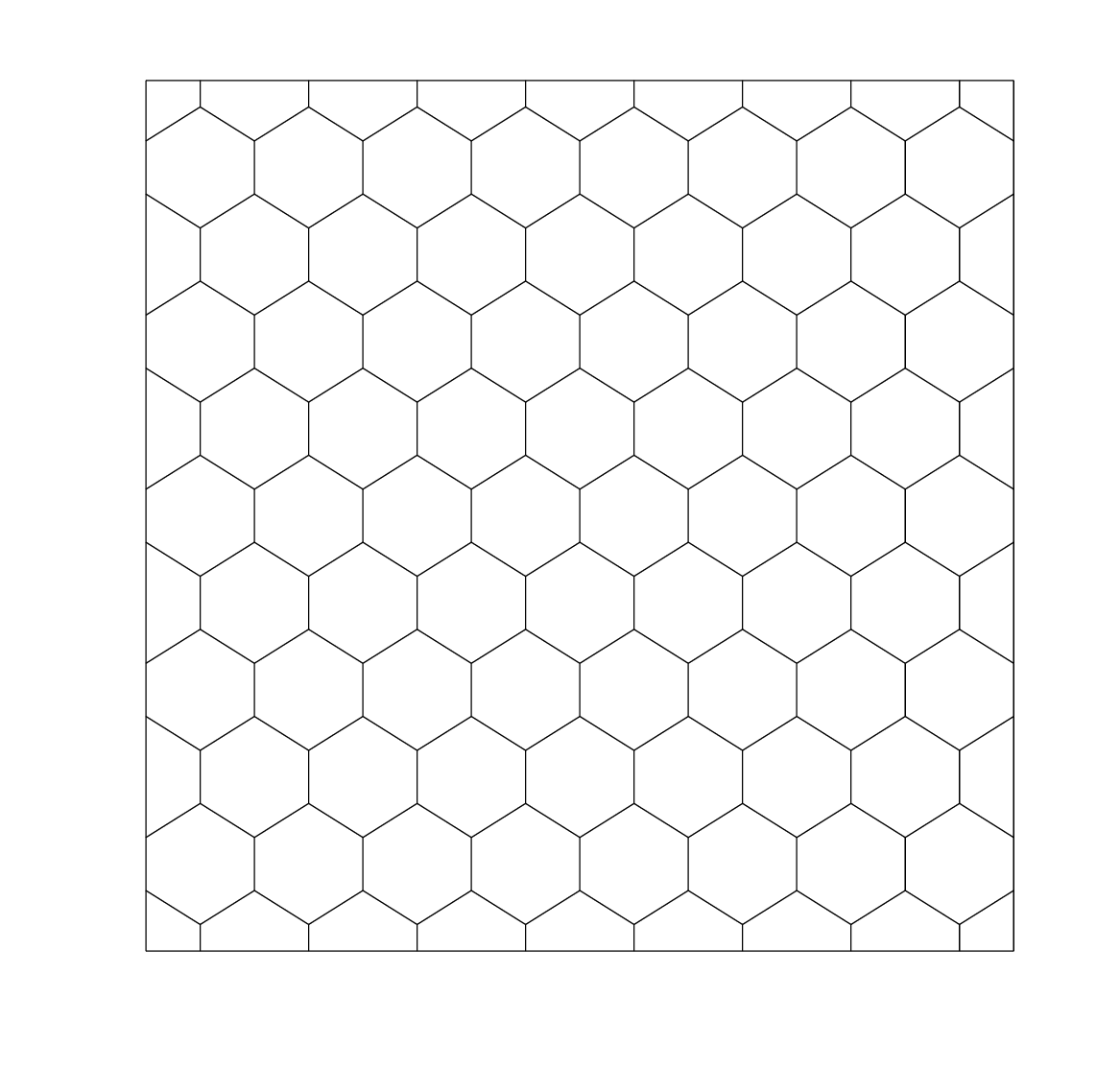}
                          \centering\includegraphics[height=3.7cm, width=3.7cm]{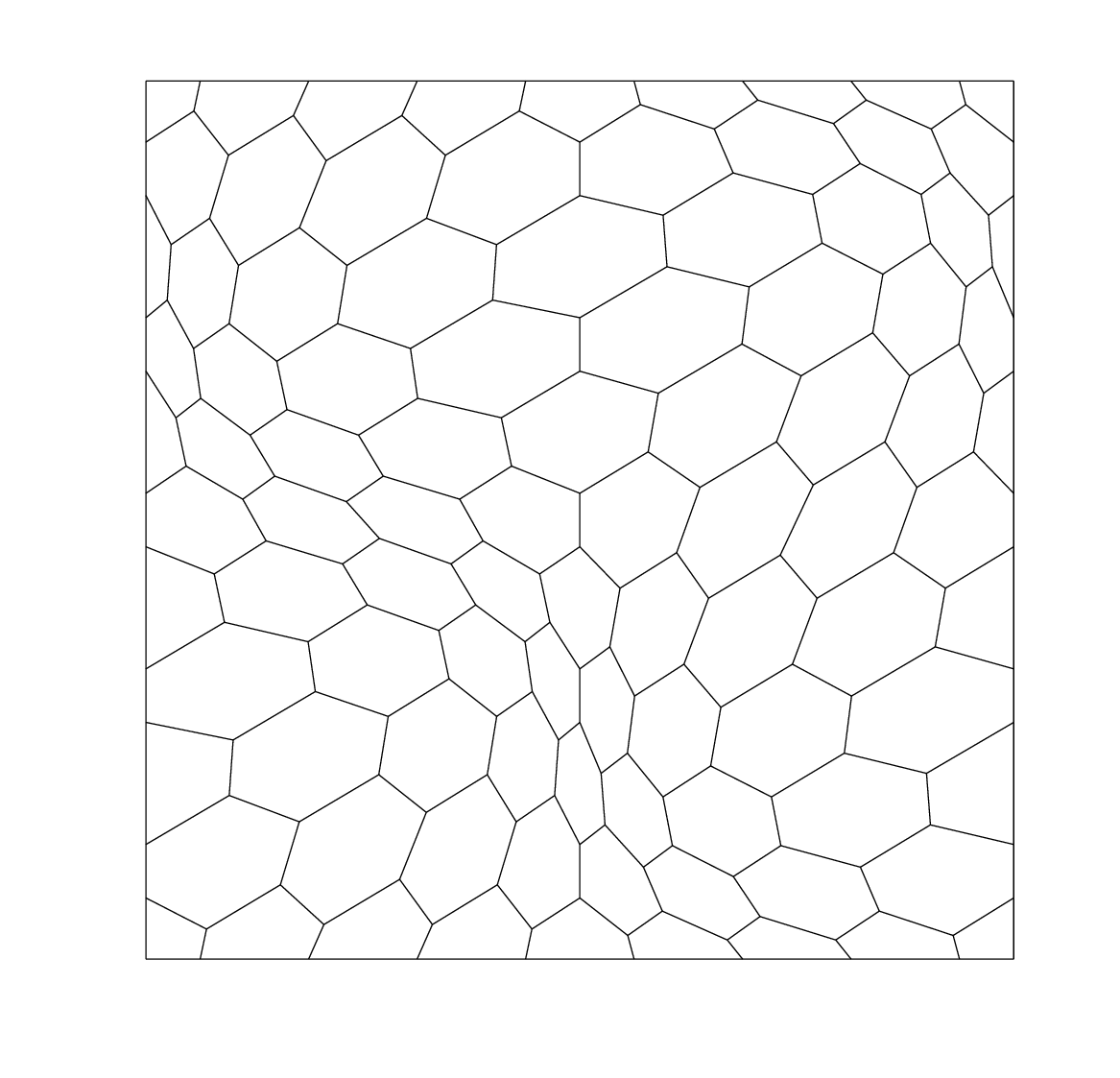}
                           \centering\includegraphics[height=3.7cm, width=3.7cm]{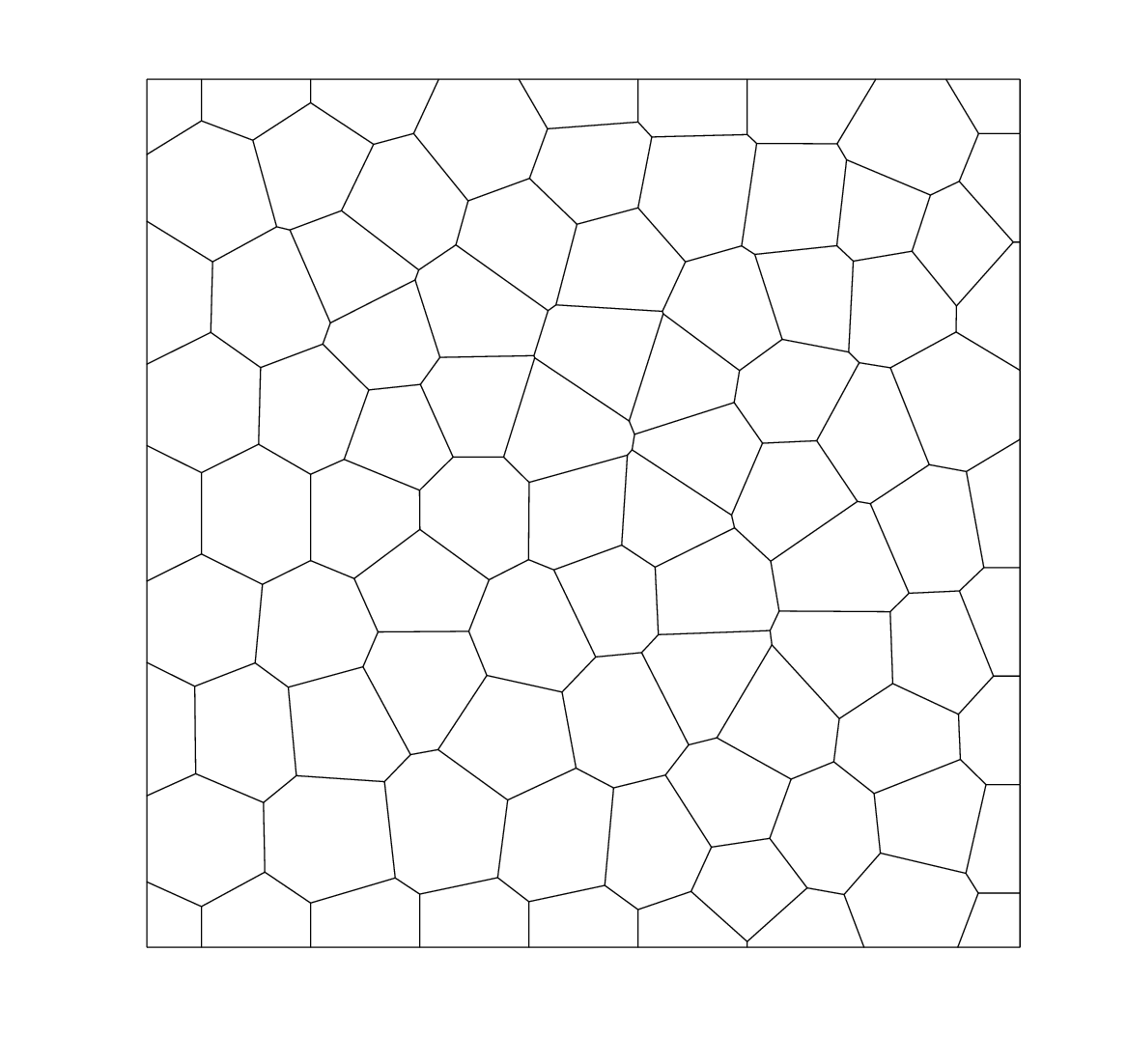}
                   \end{minipage}
		\caption{Initial meshes. From top left to bottom right: $\CT_h^1$, $\CT_h^2$, $\CT_h^3$, $\CT_h^4$, $\CT_h^5$, and $\CT_h^6$, with $N=8$.
}
		\label{FIG:meshes}
	\end{center}
\end{figure}

To complete the proposed VEM \eqref{def:weak_BFH_discrete}, we need to describe the bilinear forms $S_{\textit{V}}^{E}(\cdot,\cdot)$ and $S_{\textit{T}}^{E}(\cdot,\cdot)$ that satisfy \eqref{eq:S_v} and \eqref{eq:S_T}, respectively. For this purpose, we proceed as in \cite{AVV,MR2997471} and use the so-called  dofi--dofi stabilization, which is defined as follows: for each $E$, we denote by $\vec{\bu}_h$, $\vec{\bv}_h$ and $\vec{T}_h$, $\vec{R}_h$, the real-valued vectors containg the values of the local degrees of freedom associated with $\bu_h$, $\bv_h$ in $\bVh$ and with $T_h$, $R_h$ in $\V_h$, respectively. Then, we set
\[
S_{\textit{V}}^{E}(\bu_h,\bv_h):=\vec{\bu}_h\cdot \vec{\bv_h},
\qquad
S_{\textit{T}}^{E}(T_h,R_h):=\vec{T_h}\cdot\vec{R_h}.
\]

To calculate the error between the exact velocity component $\bu$ of the solution and the corresponding approximation obtained with our VEM, namely, $\bu_h$, we will use the following computable quantity
\[
|e_{\bu}|_{1,\O}:=\displaystyle\left(\displaystyle\sum_{E\in\CT_h}|\bu-\bPi_2^{0,E}\bu_h|_{1,E}^2\right)^{\frac{1}{2}}.
\]
A similar expression is used to calculate the error between $T$ and $T_h$. Finally, to calculate the pressure error we consider $\|e_\textsf{p}\|_{0,\O}:=\|\textsf{p}-\textsf{p}_h\|_{0,\O}$.

\subsection{Unit square} Let us first test our method on the unit square $\O=(0,1)^2$. The viscosity coefficient $\nu$, the thermal diffusivity coefficient $\kappa$, and the parameter $r$ are as follows \cite{campana2024finite}:
\begin{enumerate}
\item Test 1: $\nu(T)=1+T$, $\kappa(T)=1+\sin(T)$, and $r=3$.
\item Test 2: $\nu(T)=1+e^{-T}$, $\kappa(T)=2+\sin(T)$, and $r=4$.
\end{enumerate}
The data $\boldsymbol{f}$ and $g$ are chosen so that the exact solution to problem \eqref{def:BFH} is
\[
 \bu(x_1,x_2)=(-x_1^2(x_1-1)^2x_2(x_2-1)(2x_2-1),x_2^2(x_2-1)^2x_1(x_1-1)(2x_1-1)),
\]
$\textsf{p}(x_1,x_2)=x_1x_2(1-x_1)(1-x_2)-1/36$, and $T(x_1,x_2)=x_1^2x_2^2(1-x_1)^2(1-x_2)^2$.¨

In Figures \ref{FIG:error_test1} and \ref{FIG:error_test2}, we present the experimental convergence rates in the semi-norm for the velocity error and the temperature error and in the $\L^2$-norm for the pressure error for tests 1 and 2. We have calculated these errors for the six families of meshes shown in Figure \ref{FIG:meshes} with different levels of refinement. The plots also include a reference line with slope $-2$, which indicates the optimal convergence rate of the method.
\begin{figure}[h]
	\begin{center}
			\centering\includegraphics[height=4.7cm, width=4cm]{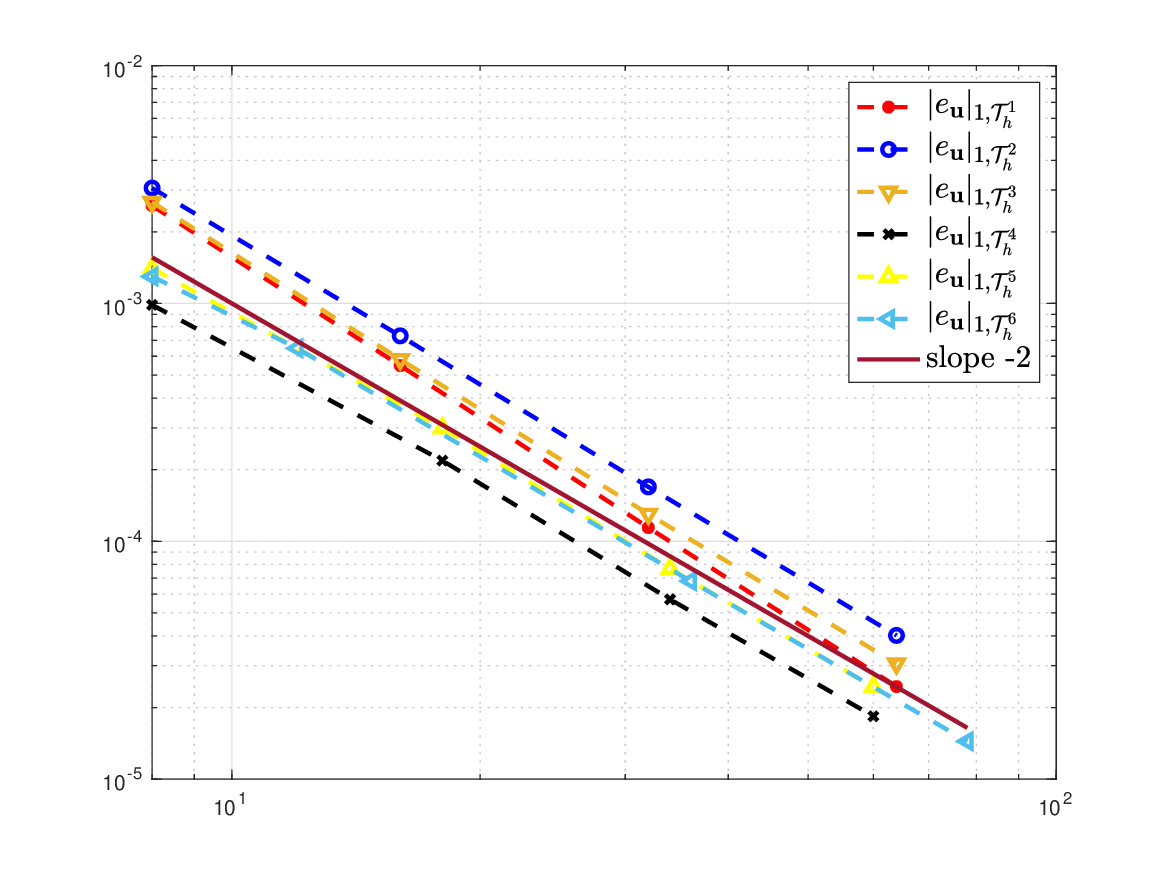}
                          \centering\includegraphics[height=4.7cm, width=4cm]{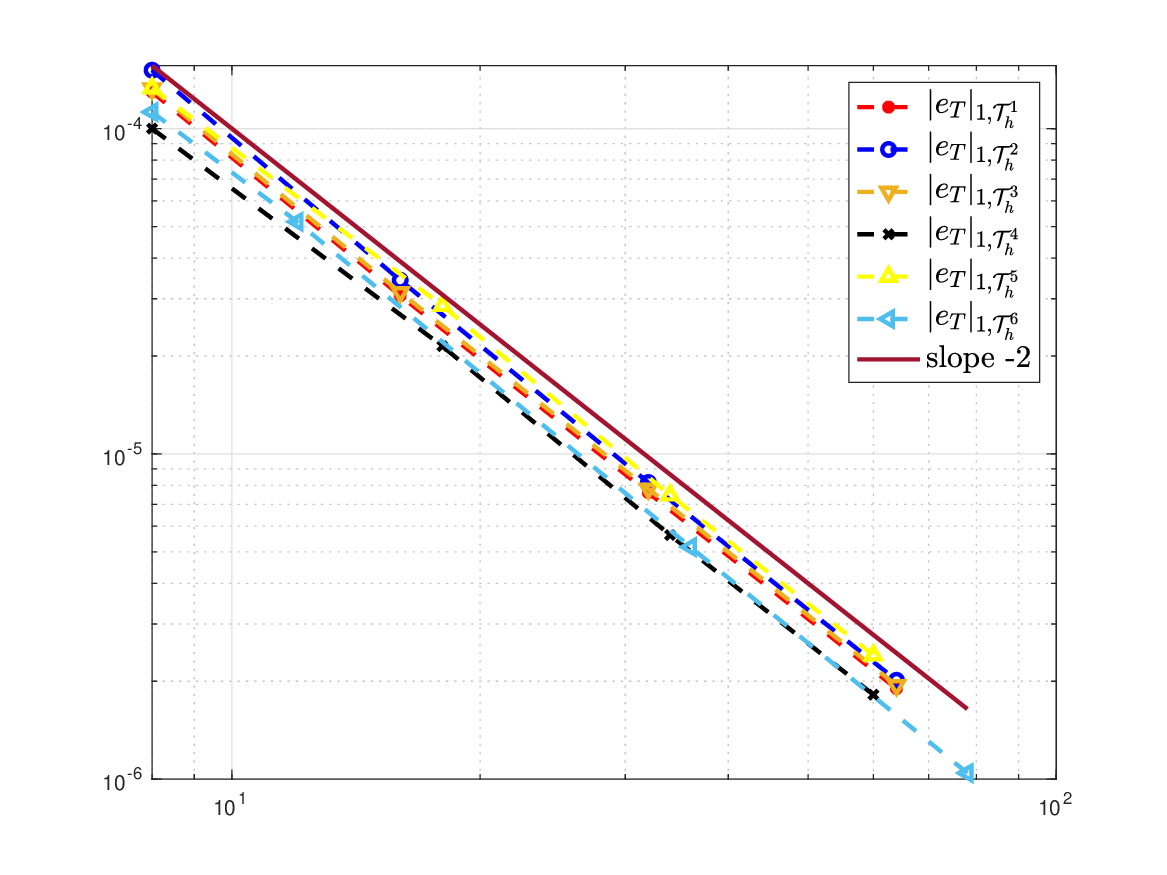}
                         \centering\includegraphics[height=4.7cm, width=4cm]{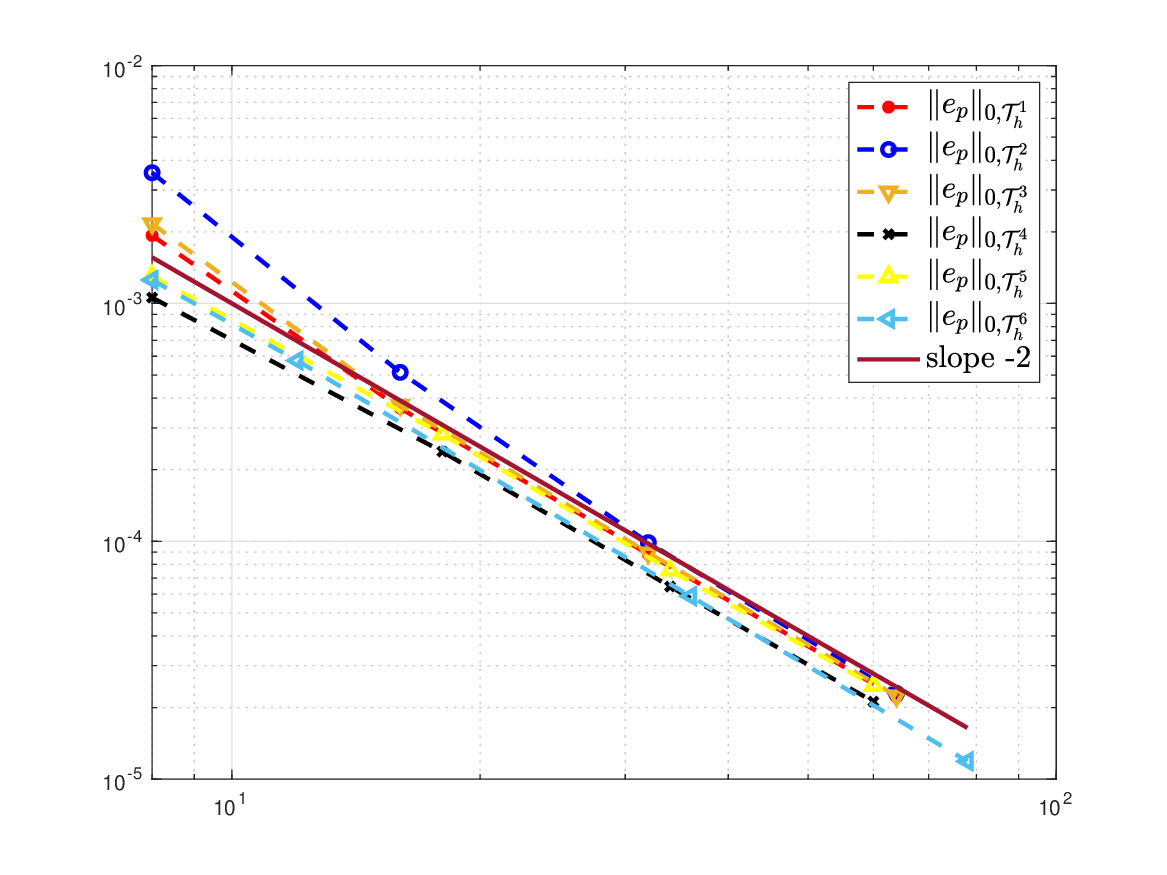}\\
		\caption{Test 1: Experimental convergence rates for $|e_{\bu}|_{1,\O}$, $|e_{T}|_{1,\O}$, and $\|e_{p}\|_{0,\O}$ on different polygonal meshes.}
		\label{FIG:error_test1}
	\end{center}
\end{figure}

\begin{figure}[h]
	\begin{center}
			\centering\includegraphics[height=4.7cm, width=4cm]{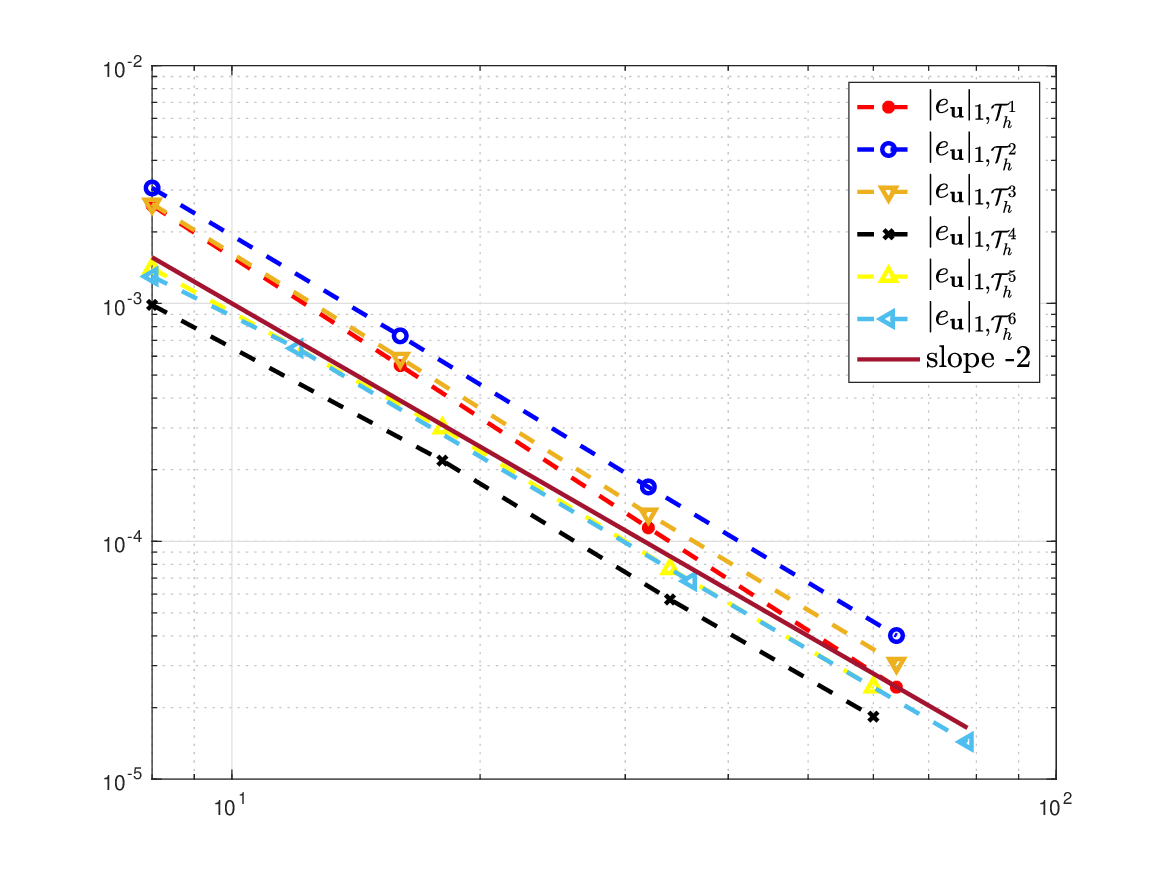}
                          \centering\includegraphics[height=4.7cm, width=4cm]{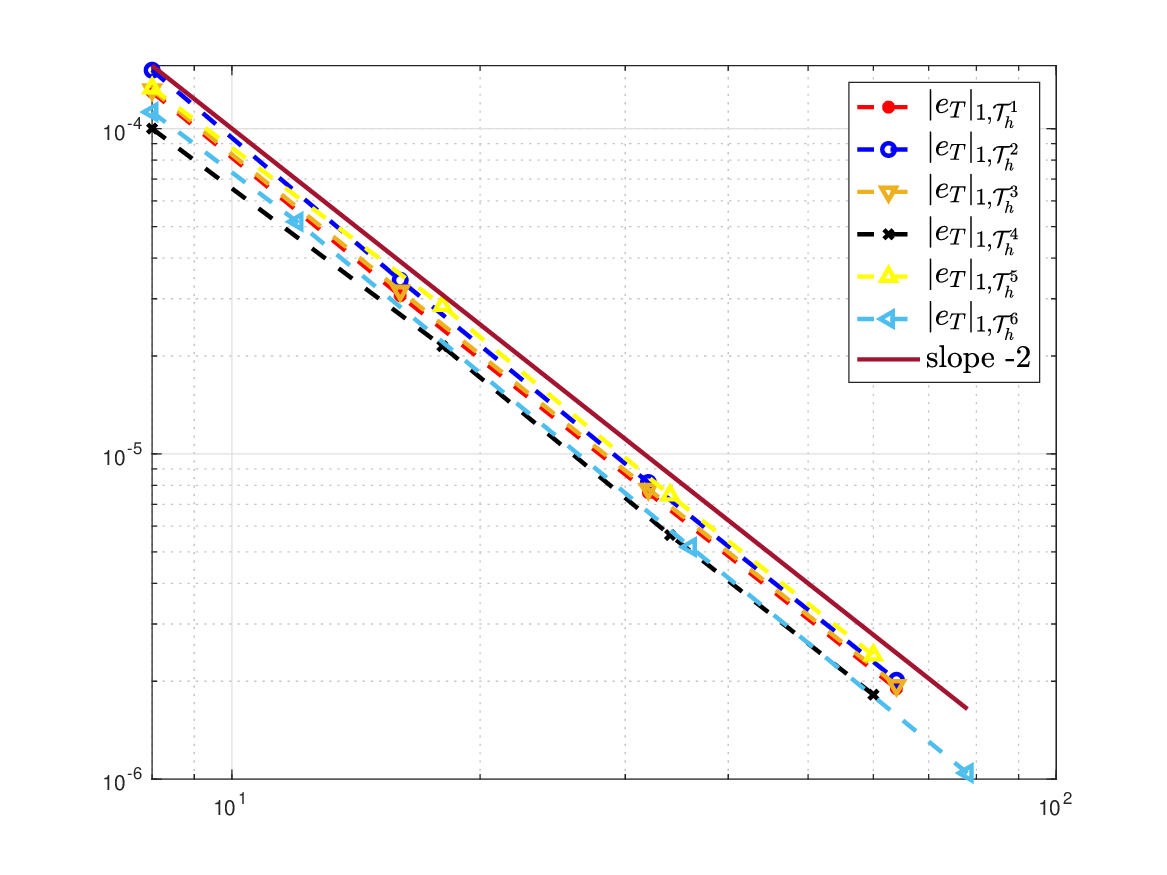}
                         \centering\includegraphics[height=4.7cm, width=4cm]{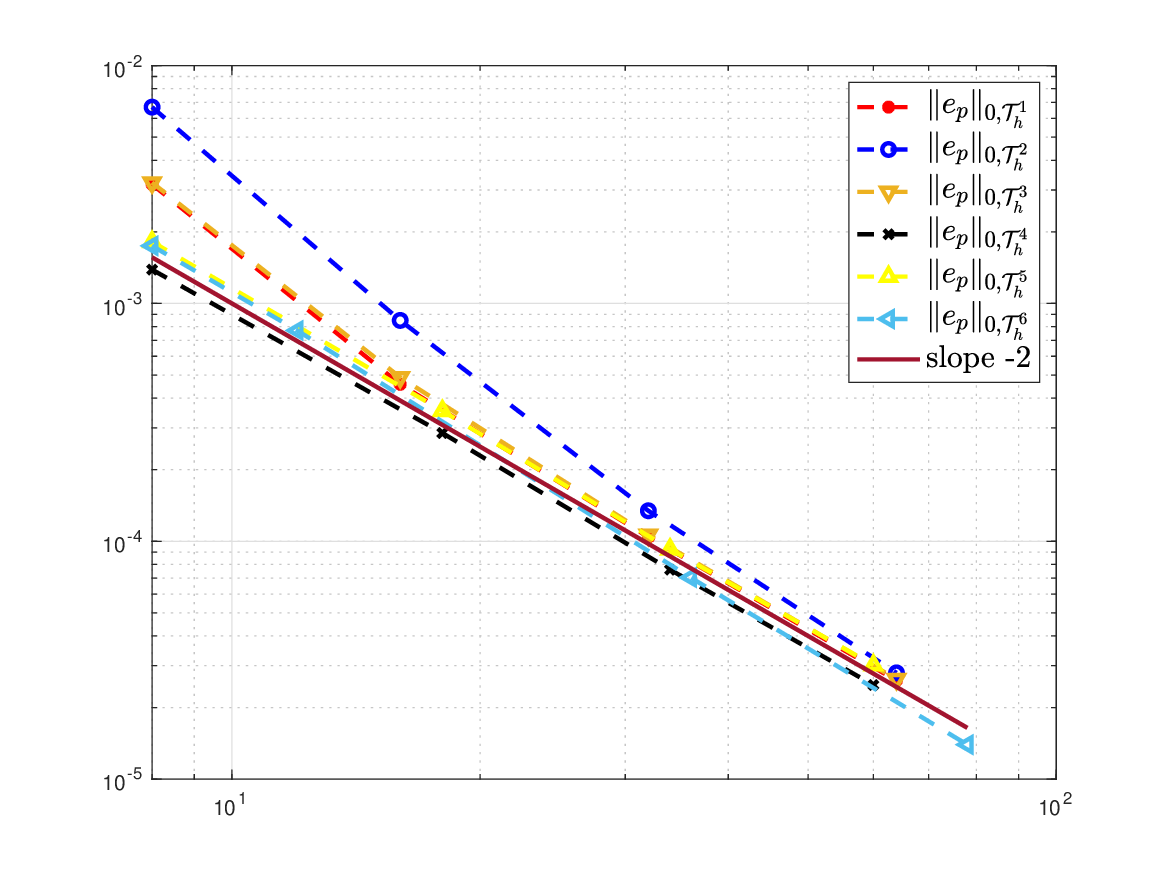}\\
		\caption{Test 2: Experimental convergence rates for $|e_{\bu}|_{1,\O}$, $|e_{T}|_{1,\O}$, and $\|e_{p}\|_{0,\O}$ on different polygonal meshes.}
		\label{FIG:error_test2}
	\end{center}
\end{figure}

Figures \ref{FIG:error_test1} and \ref{FIG:error_test2} show that the theoretical predictions of Theorem \ref{thm:error_estimates} are confirmed. We also note that optimal experimental convergence rates are obtained for each of the mesh families considered. This is computational evidence of the robustness of the VEM with respect to the geometry of the mesh. Figure \ref{FIG:solution} shows the solution obtained with our method using a particular family of polygonal meshes.
\begin{figure}[h]
	\begin{center}
		\begin{minipage}{15cm}
		  \centering\includegraphics[height=5cm, width=6cm]{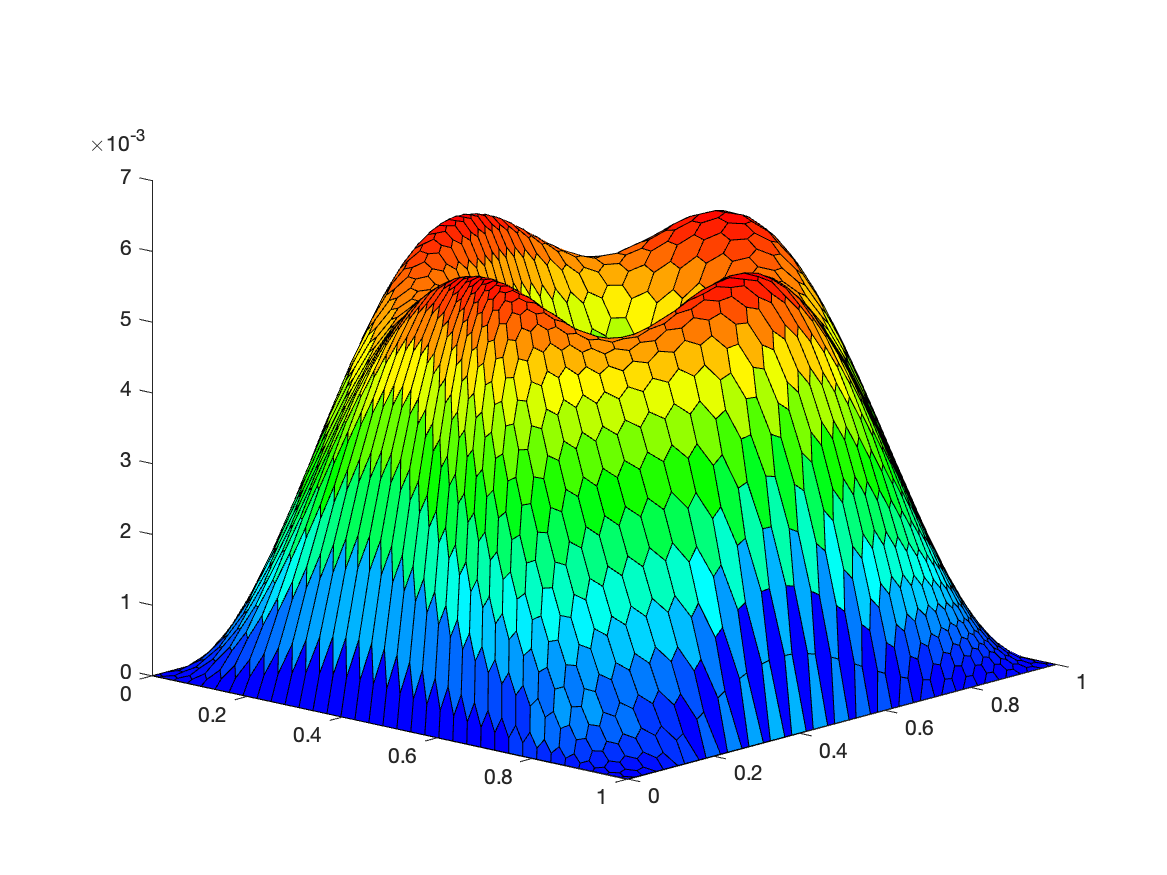}
                         \centering\includegraphics[height=5cm, width=6cm]{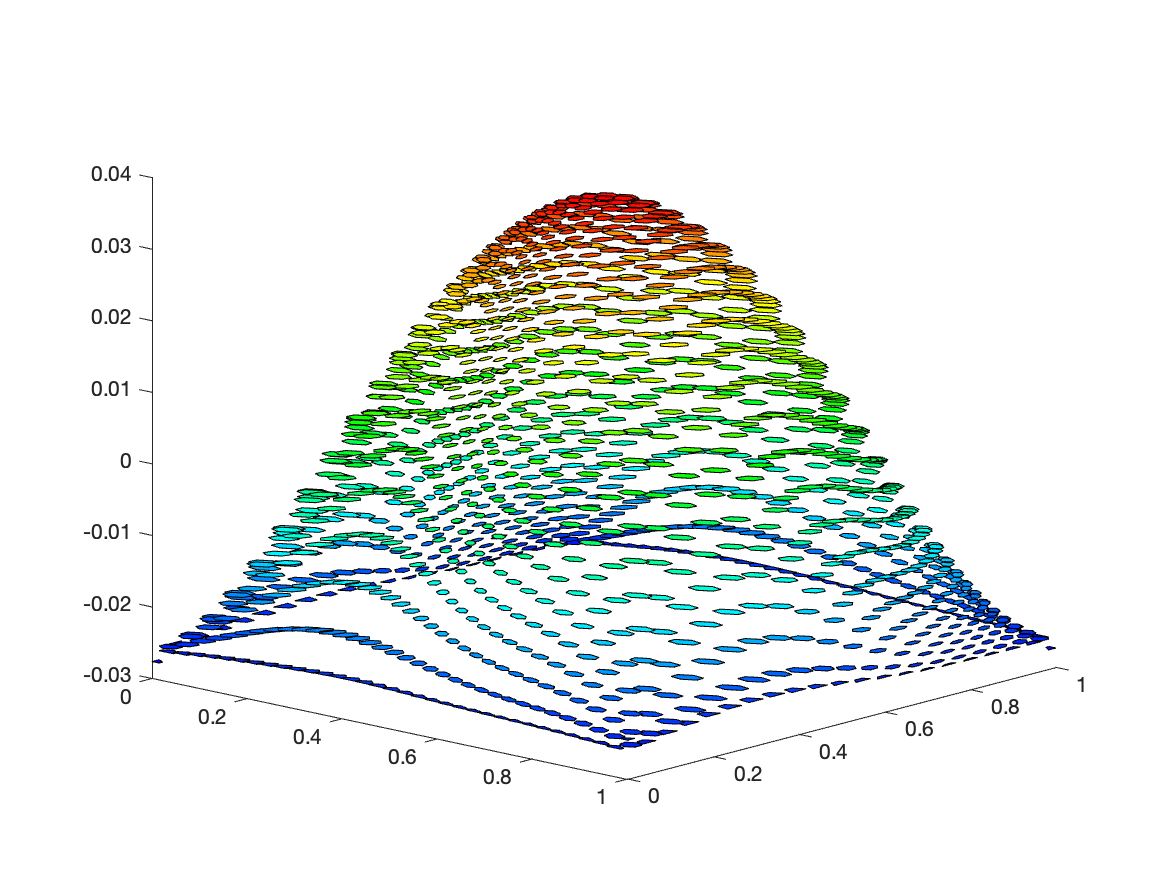}\\
                          \centering\includegraphics[height=5cm, width=6cm]{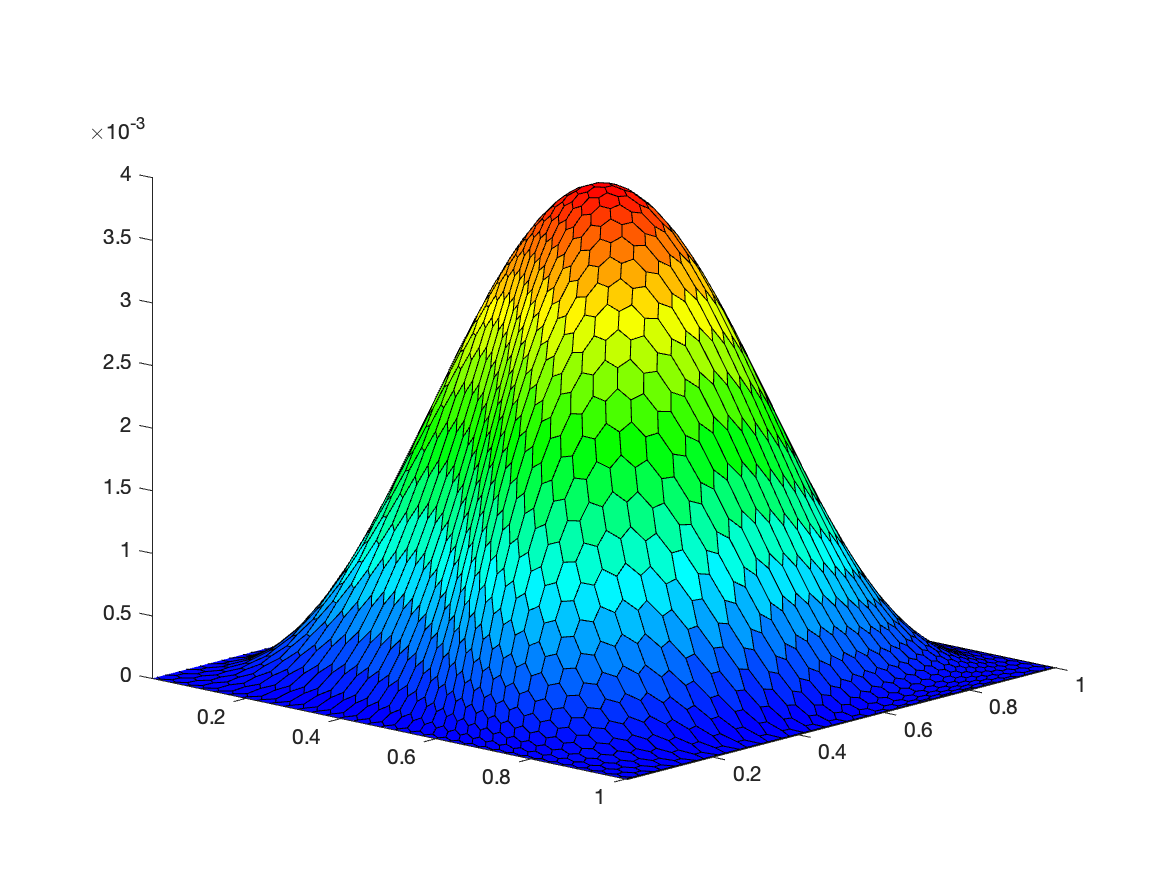}
                   \end{minipage}
		\caption{Virtual element approximatiopn: magnitude of the velocity $\bu_h$ (left above), pressure $p_h$ (right above), and temperature $T_h$ (below) in $\CT_h^5$}
		\label{FIG:solution}
	\end{center}
\end{figure}

The next goal is to prove computationally that our VEM is divergence-free. To accomplish this task, we calculate the value of the $\L^2$-norm of the discrete divergence of the velocity field $\bu_h$. In Table \ref{tabla1}, we present the values of $\|\div \bu_h\|_{0,\O}$ that we computed for five mesh families at different refinement levels. As the results in Table \ref{tabla1} show, the method is indeed divergence-free, and this again is independent of the polygonal mesh used.
\begin{table}[h]
\caption{ Test 2: $\|\div \bu_h\|_{0,\O}$ computed for five mesh families.}
\label{tabla1}
\begin{center}
\begin{tabular}{|c|c|c|c|c|c|c|c|} \hline
\multicolumn{6}{ |c| }{$\|\div \bu_h\|_{0,\O}$} \\ \hline 
 $N$ & $\CT_{h}^{1}$ & $\CT_{h}^{2}$ & $\CT_{h}^{3}$ & $\CT_{h}^{4}$  &$\CT_{h}^{5}$. \\ \hline 
8 &4.3865e-19 &  4.7045e-19 &  6.9960e-19   &6.3543e-19  & 8.0372e-19\\
 16  &4.6768e-19&   4.5652e-19 &  6.3817e-19 &  6.2734e-19&   7.4630e-19\\
 32  &4.8928e-19  & 5.5585e-19  & 5.6303e-19  & 5.0117e-19   &5.8633e-19\\
   64  &       4.6112e-19  & 5.5279e-19  & 5.6886e-19 &  5.0243e-19  & 5.7873e-19\\\hline
\end{tabular}
\end{center}
\end{table}
\subsection{Influence of the viscosity}
In this test, we investigate the influence of the viscosity coefficient $\nu(\cdot)$ on the calculation of the error in the approximation of the velocity, pressure, and temperature variables. It is well-known that most standard approximation techniques are not robust with respect to this coefficient. To perform this study, we choose different values of $\nu$ and investigate the experimental convergence rates on different polygonal meshes. We use a similar configuration as in the previous tests: $\O=(0,1)^2$, $s=3$, and the solution $(\bu, p, T)$ as previously described. For simplicity, we assume that $\kappa=1$ and consider constant values for the viscosity parameter: $\nu=10^{-1}, 10^{-4}, 10^{-8}$.

\begin{figure}[h]
	\begin{center}
			\centering\includegraphics[height=4.7cm, width=4cm]{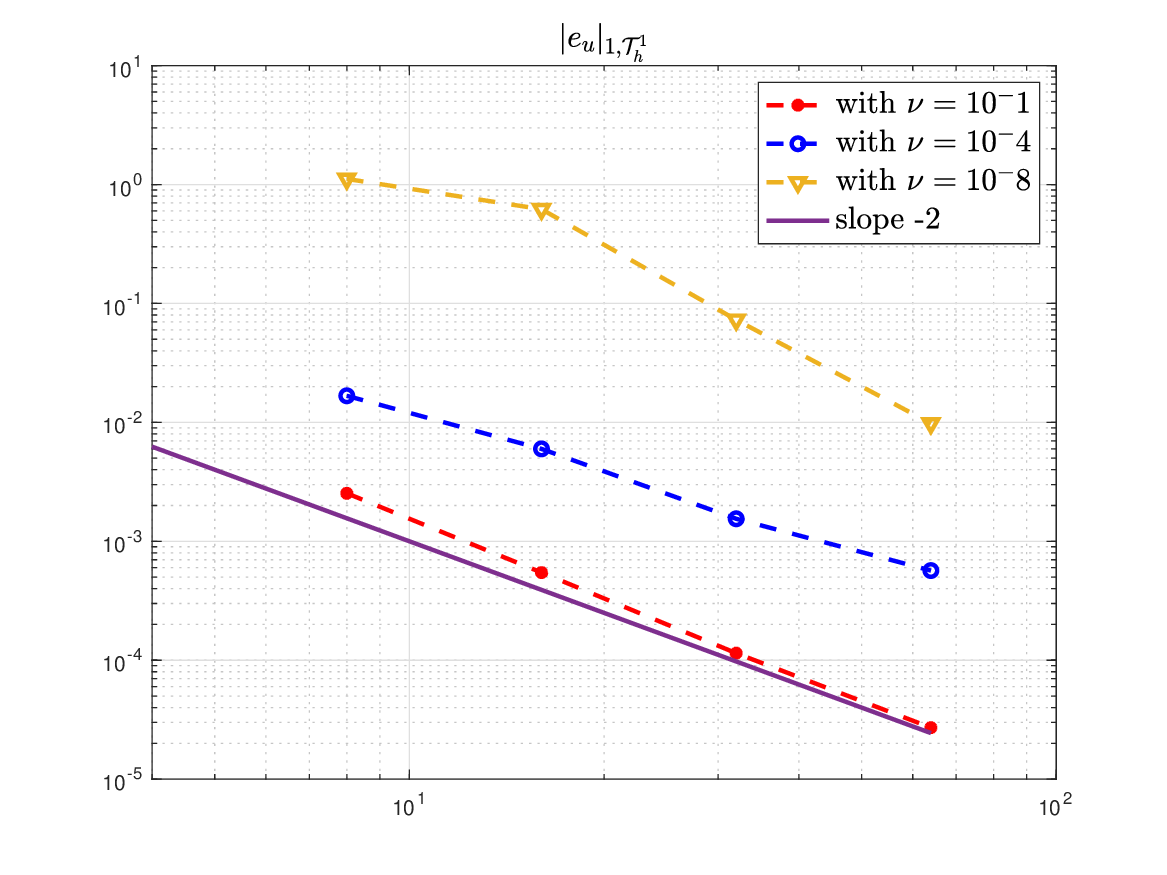}
                          \centering\includegraphics[height=4.7cm, width=4cm]{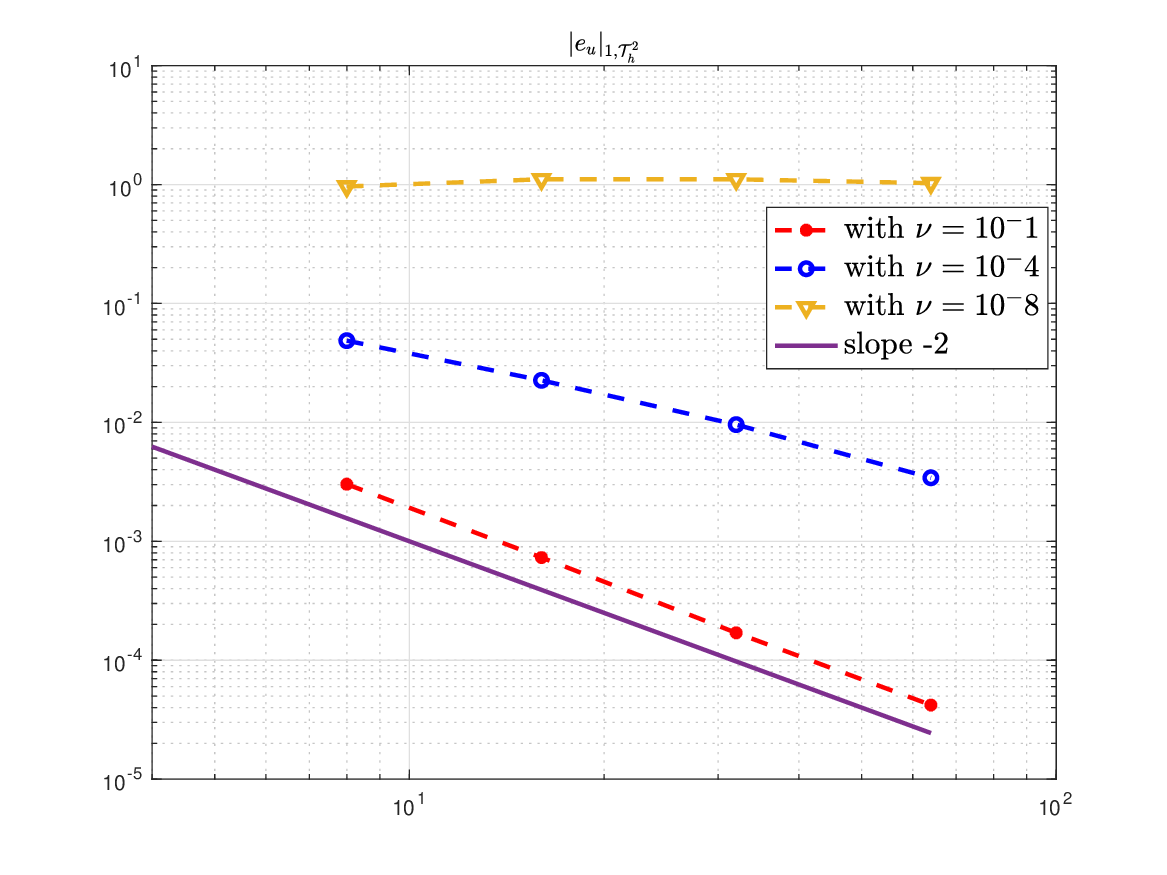}
                         \centering\includegraphics[height=4.7cm, width=4cm]{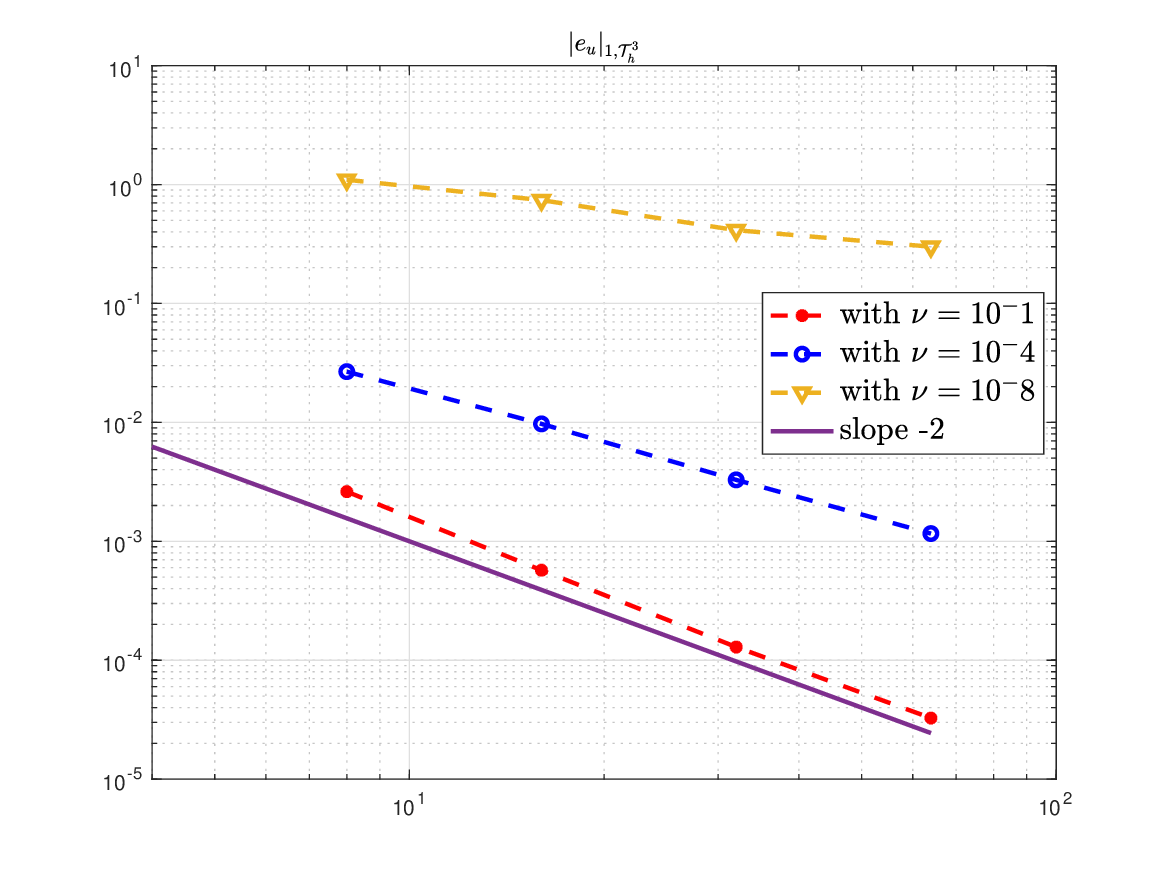}\\
		\caption{Test 3: Velocity error for $\nu =10^{-1},10^{-4},10^{-8}$ on the meshes $\CT_h^1, \CT_h^2$, and $\CT_h^3$.}
		\label{FIG:error_nu1}
	\end{center}
\end{figure}

\begin{figure}[h]
	\begin{center}
			\centering\includegraphics[height=4.7cm, width=4cm]{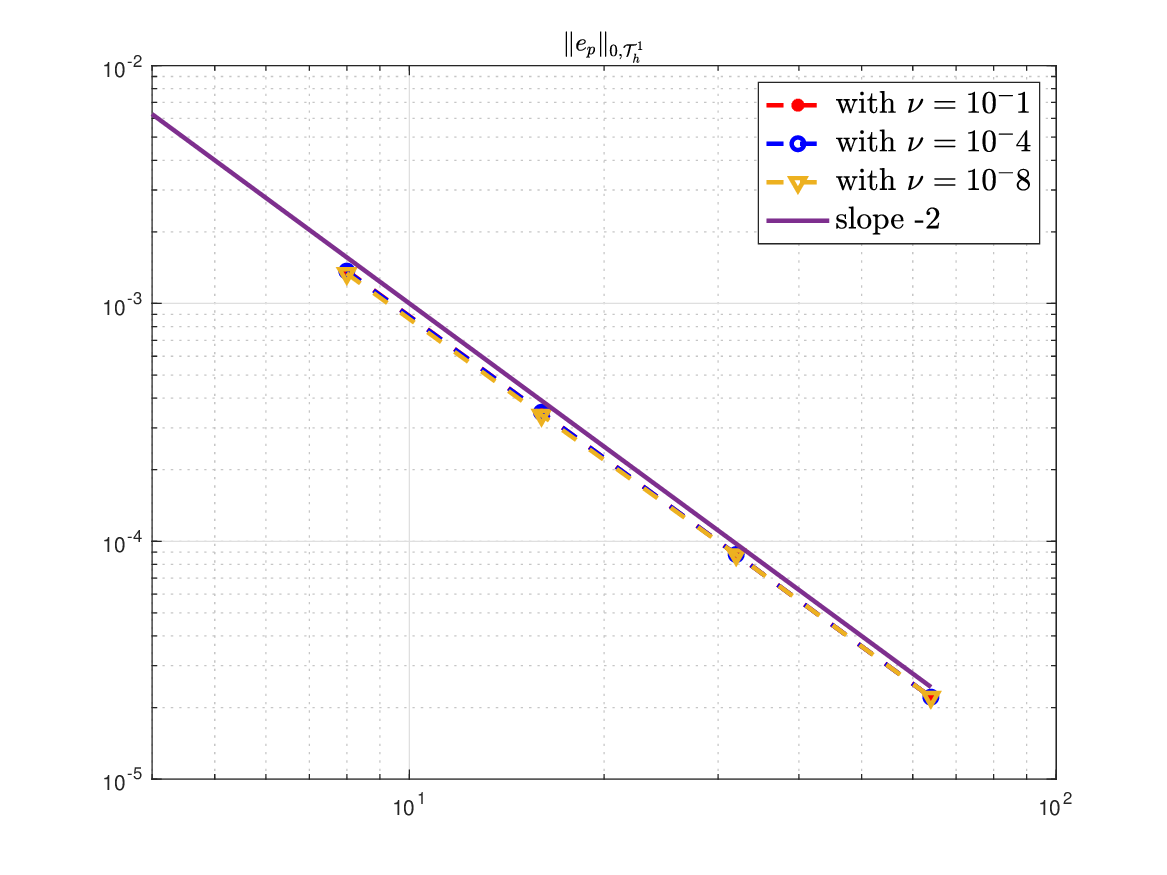}
                          \centering\includegraphics[height=4.7cm, width=4cm]{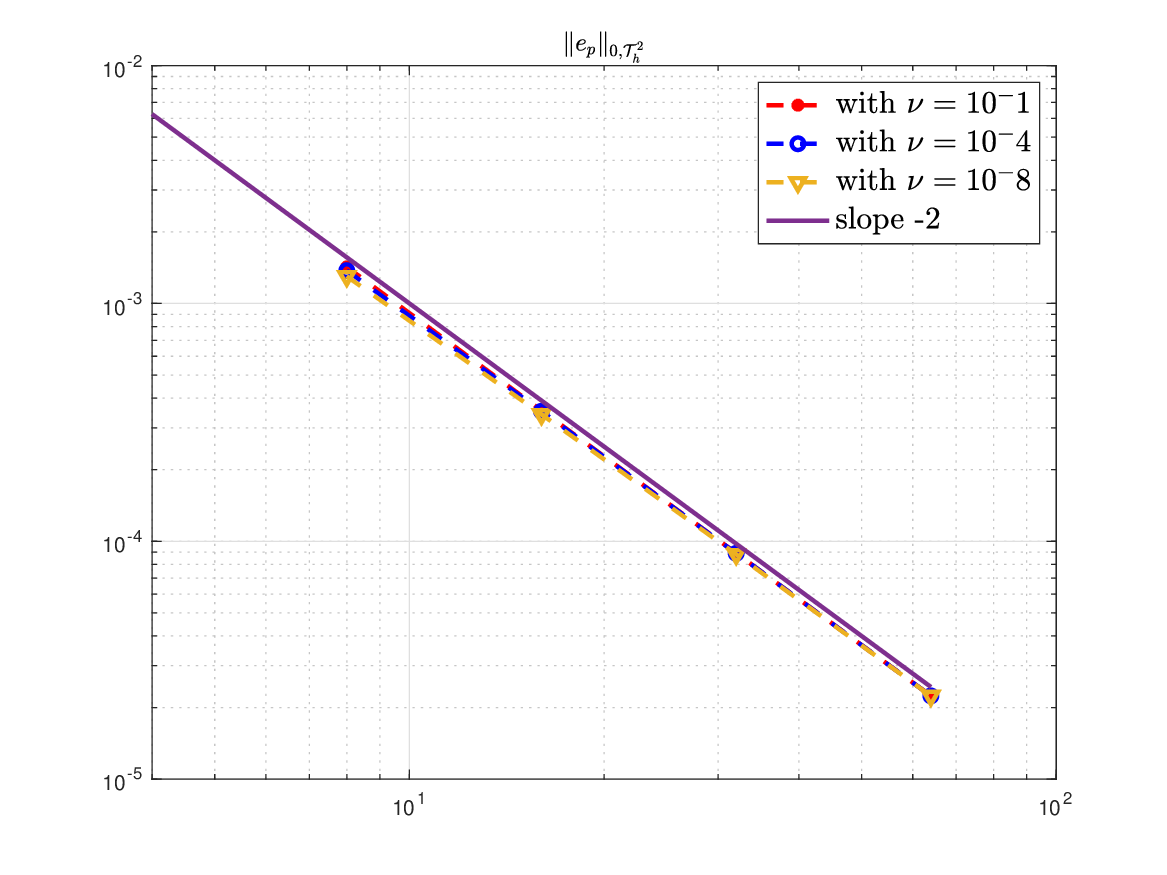}
                         \centering\includegraphics[height=4.7cm, width=4cm]{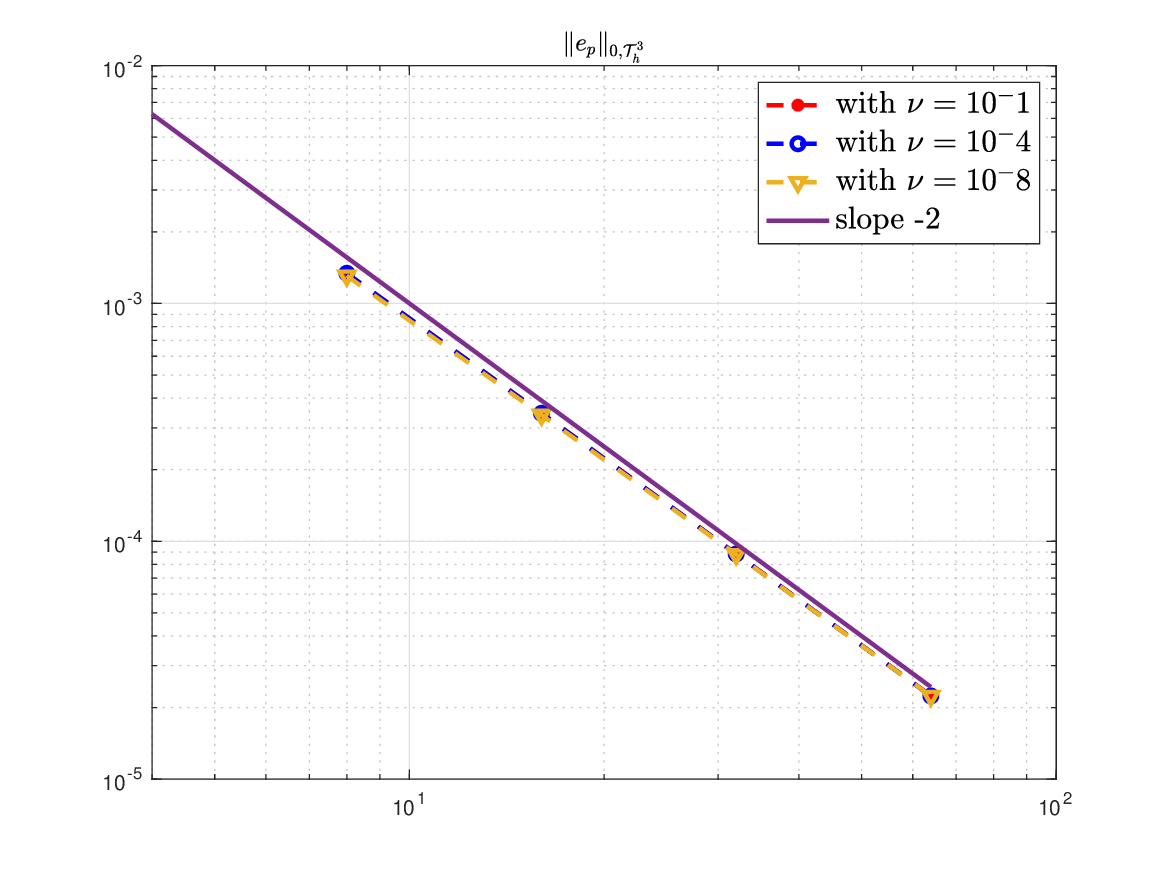}\\
		\caption{Test 3: Pressure error for $\nu =10^{-1},10^{-4},10^{-8}$ on the meshes $\CT_h^1, \CT_h^2$ and $\CT_h^3$.}		\label{FIG:error_nu2}
	\end{center}
\end{figure}
\begin{figure}[h]
	\begin{center}
			\centering\includegraphics[height=4.7cm, width=4cm]{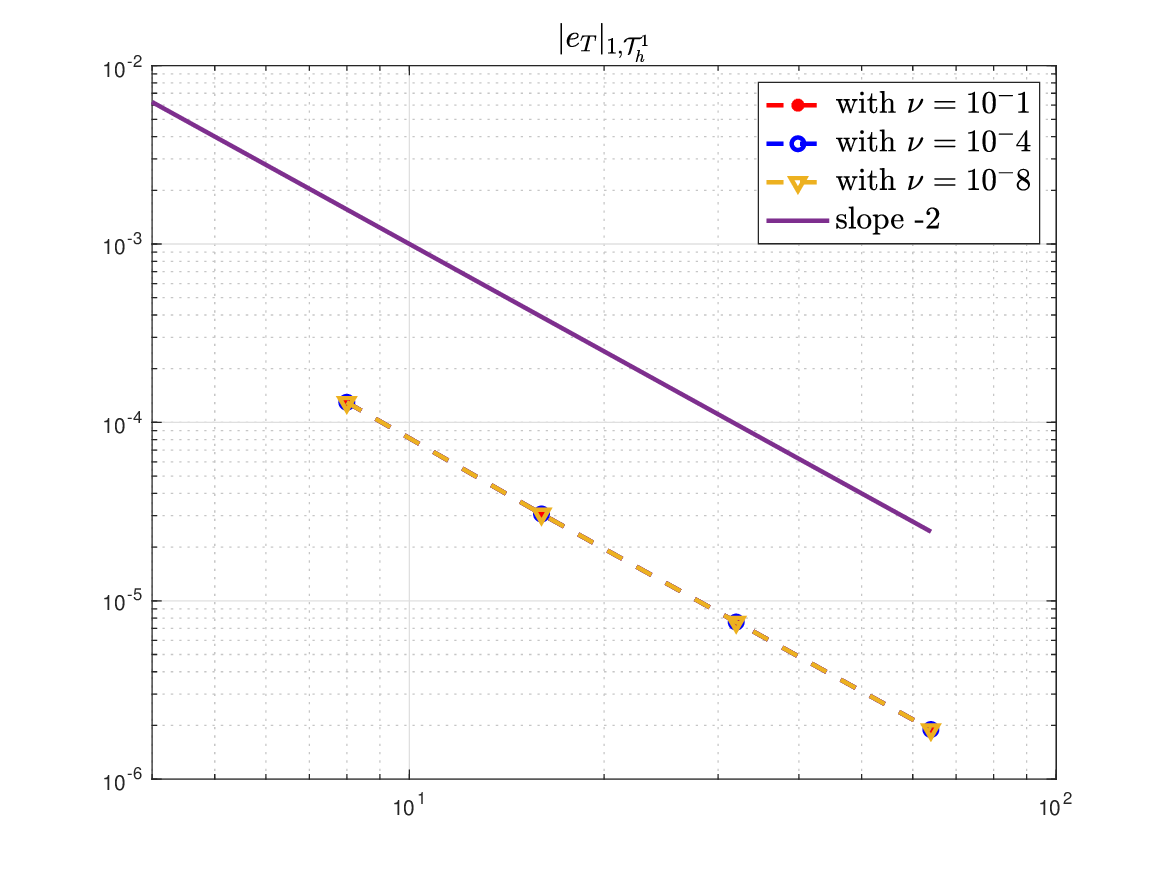}
                          \centering\includegraphics[height=4.7cm, width=4cm]{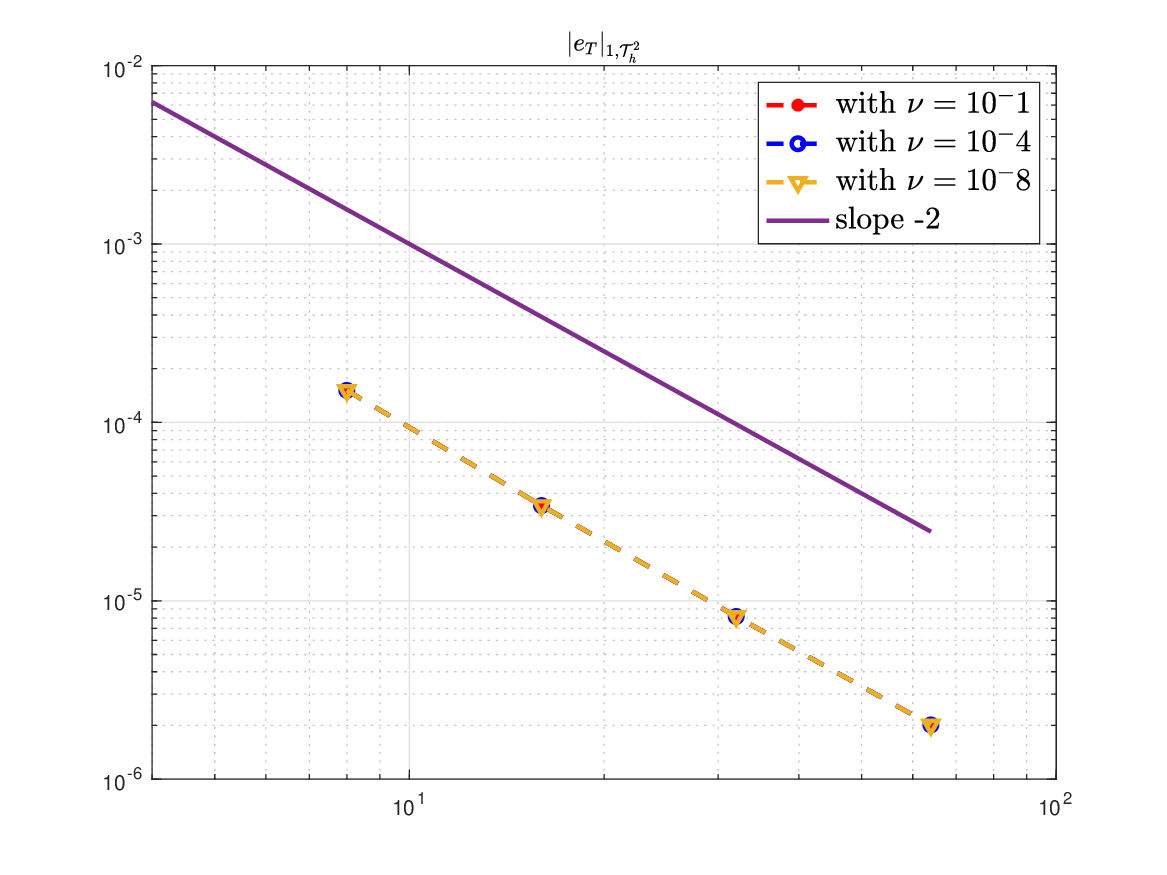}
                         \centering\includegraphics[height=4.7cm, width=4cm]{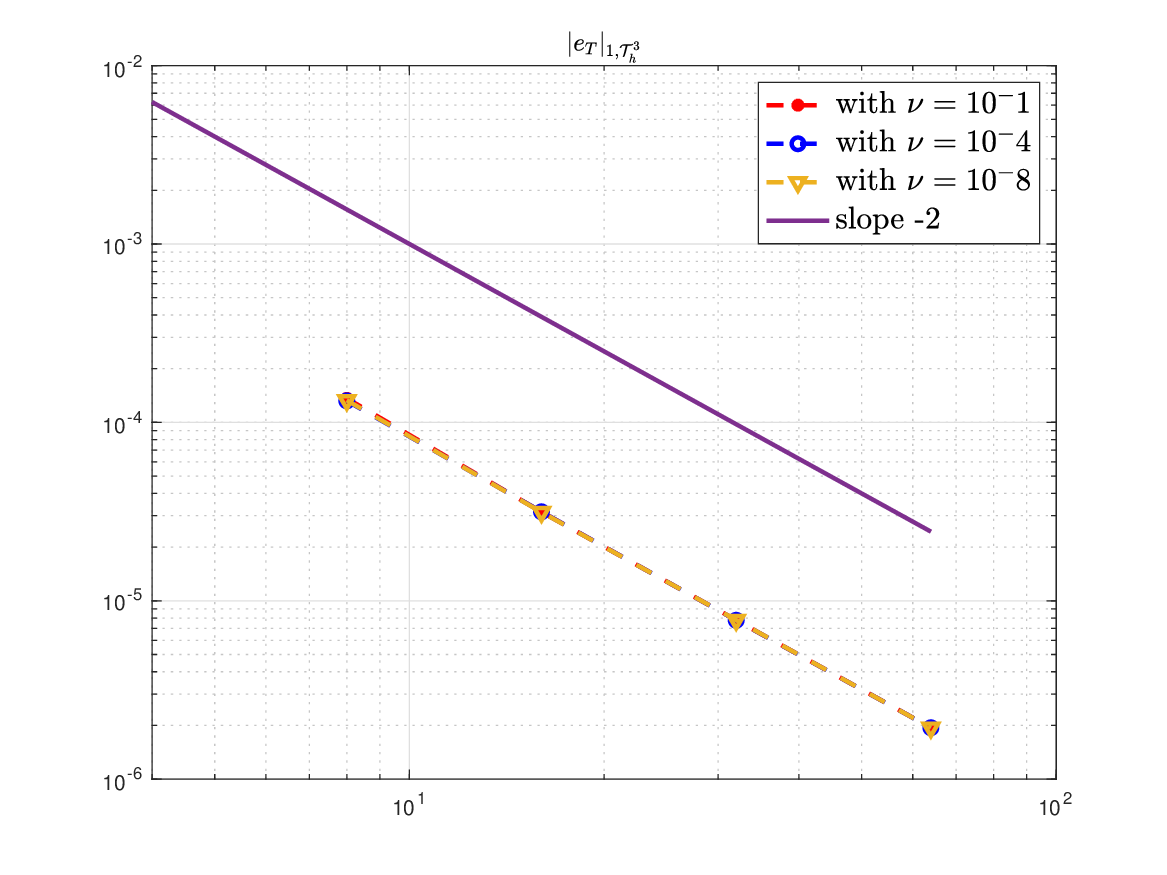}\\
		\caption{Test 3: Temperature error for $\nu =10^{-1},10^{-4},10^{-8}$ on the meshes $\CT_h^1, \CT_h^2$ and $\CT_h^3$.}
		\label{FIG:error_nu3}
	\end{center}
\end{figure}

Our results are shown in Figures \ref{FIG:error_nu1}, \ref{FIG:error_nu2} and \ref{FIG:error_nu3}. In Figure \ref{FIG:error_nu1}, we observe that the optimal experimental convergence rate for the velocity error is no longer achieved when the viscosity coefficient becomes smaller. This phenomenon occurs for different meshes, which leads to the conclusion that the loss of convergence for the velocity does not depend on the geometry of the mesh, but only on the physical parameter. A different behavior is observed for the temperature error and the pressure error: the optimal convergence rate is also attained when the viscosity becomes smaller.  We would like to note that this is independent of the considered polygonal mesh. The results obtained for the velocity error and the pressure error are in agreement with the numerical tests reported in \cite{NMTMA-17-210}. Our results also show that the method works for the  temperature error in the same way as for the pressure error. We emphasize that the derivation of an optimal error estimate that is robust with respect to the coefficients $\nu(\cdot)$ and $\kappa(\cdot)$ is not analyzed in our work and is a valuable option for future extensions of the study presented here.
\\ \\
\textbf{Declaration of competing interest}

The authors declare that they have no known competing financial interests or personal relationships that could have appeared
to influence the work reported in this paper.
\\ \\
\textbf{Data availability}

Data will be made available on request.

\bibliographystyle{siam}
\footnotesize
\bibliography{bib_LOQ}

\end{document}